\documentclass[10pt]{article}
\usepackage{fullpage}
\pdfoutput=1

\usepackage{scohesion}

\usepackage[style=alphabetic, maxbibnames=999]{biblatex}
\usepackage{authblk}

\title{Synthetic Spectra via a \\ Monadic and Comonadic Modality}
\author[1]{Mitchell Riley}
\author[2]{Eric Finster}
\author[1]{Daniel R. Licata}
\affil[1]{Wesleyan University}
\affil[2]{University of Cambridge}
\date{}

\begin{document}
\maketitle

\begin{abstract}
  Homotopy type theory allows for a synthetic formulation of homotopy
  theory, where arguments can be checked by computer and automatically
  apply in many semantic settings. Modern homotopy theory makes
  essential use of \emph{spectra}, a model for \emph{stable homotopy
    types}, where the suspension and loop space operations become an
  equivalence.  One can define a version of spectra analytically in type
  theory, but with this definition can be quite difficult to work with.
  In this paper, we develop an alternative synthetic approach to
  spectra, where spectra are represented by certain types, and
  constructions on them by type structure—--maps of spectra by
  ordinary functions, loop spaces by the identity type, and so on.
  There is an $\infty$-topos of \emph{parametrised spectra}, whose
  objects are an index space with a family of spectra over it, so
  standard homotopy type theory can be interpreted in this setting.  To
  isolate the spaces (as objects with trivial spectra) and the spectra
  (as objects with trivial base) within this more general setting, we
  extend type theory with a novel modality that is simultaneously a
  monad and a comonad.  Intuitively, this modality keeps the base of an
  object the same but replaces the spectrum over each point with a
  trivial one.  Because this modality induces a non-trivial endomap on
  every type, it requires a more intricate judgemental structure than
  previous modal homotopy type theories.  We show that the type theory
  is sound and complete for an abstract categorical semantics, in terms
  of a category-with-families with a weak endomorphism whose functor on
  contexts is a bireflection, i.e.\ has a counit an a unit that are a
  section-retraction pair.  Next, we augment the type theory with a pair
  of axioms, one which implies that the spectra are stable, and the
  other which relates the synthetic spectra to the analytic ones, and,
  working inside the type theory, show that our synthetic spectra have
  many of the properties expected of spectra.
\end{abstract}

\tableofcontents

\section*{Introduction}

Homotopy type theory provides a setting in which one can do
\emph{synthetic} homotopy theory: rather than working with concrete
topological spaces or simplicial sets, one works with types and the
universal constructions provided by the type formers of the
theory. Working this way requires clever new arguments for standard
results, but the payoff is that the same proof applies in any model of
the theory, not just in the homotopy theory of spaces.  It is
conjectured that all of axiomatic homotopy type theory (``axiomatic
HoTT'') as in~\cite{hottbook} can be interpreted in any $\infty$-topos
--- the interpretation of the basic Martin-L\"of type
theory~\cite{awodey-warren:id-types, lumsdaine-warren:local-universes,
  arndt-kapulkin:models}, Voevodsky's univalence
axiom~\cite{gepner-kock:univalence, mike:elegant-reedy, mike:all}
and a class of higher inductive types~\cite{lumsdaineshulman20hits}
have been worked out, though closure of
the universes under higher inductive types is in progress.  The basics of
synthetic homotopy theory are presented in~\cite[Chapter 8]{hottbook},
and a significant number of results have been developed and formalized
since.

One of the main advantages of axiomatic HoTT is that all
constructions performed in it are necessarily homotopy invariant.
However, this is a double-edged sword, as it rules out some of the first
definitions from algebraic topology one might hope to make. For example,
a key invariant of topological spaces is their homology and cohomology
groups, which are often easier to calculate than their homotopy
groups. Externally, the ordinary homology and cohomology of a space $X$
are defined via maps into $X$ from simplices of various dimensions. But
these simplices are all contractible, so working internally they are
indistinguishable from the point.  Instead, homology and cohomology can
be defined in type theory~\cite{cavallo:cohomology, graham:homology} via
the \emph{spectra} that represent them.

Classically, the study of spectra was motivated by the Freudenthal
suspension theorem (see~\cite[Chapter 8]{hottbook} for a proof in type
theory), which implies that, for nice spaces $X$ and $Y$, the sequence
of homotopy classes of maps \[ [X, Y] \to [\Sigma X, \Sigma Y] \to
[\Sigma^2 X, \Sigma^2 Y] \to \dots \] eventually
\emph{stabilises}. \emph{Stable homotopy theory} studies the phenomena
that survive after arbitrarily many suspensions, and it is an important
tool for obtaining results about unstable homotopy theory/spaces as well
--- e.g.\ the homotopy groups of spheres in a certain range coincide with
certain stable homotopy groups. One can form a category out of the above
observation, where the objects are nice pointed spaces and the hom-sets
are given by the colimit over the above diagram, but the resulting
category does not have very nice properties. This motivates passing to
a category of \emph{spectra} that includes the previous as a full
subcategory, but is much better behaved categorically.

Using $\infty$-categorical technology (by which we mean
$(\infty,1)$-categorical), the $\infty$-category of spectra can be
defined by starting with the $\infty$-category of pointed spaces
$\mathcal{S}_*$, and inverting the loop space functor $\Omega$ in a
universal way:
\begin{align*}
\Spec = \varprojlim \left(  \cdots \xrightarrow{\Omega} \mathcal{S}_* \xrightarrow{\Omega} \mathcal{S}_* \xrightarrow{\Omega} \mathcal{S}_* \right)
\end{align*}
In the $\infty$-category of spectra, the suspension $\Sigma$ and loop
space $\Omega$ functors are equivalences, in contrast to ordinary
spaces where $\Sigma$ and $\Omega$ are far from being
equivalences. This makes $\Spec$ a good place to do stable homotopy
theory: \emph{everything} is stable under suspension. Spectra behave
much like objects from algebra, such as modules over a ring, or chain
complexes. One has a zero object and biproducts, and every morphism of
spectra has a well-behaved kernel and cokernel. Like chain complexes,
$\Sigma$ and $\Omega$ shift objects up and down in dimension.

Unwinding the above definition, a spectrum can be presented concretely
as a sequence of pointed spaces $X_0,X_1,X_2,\ldots$, together with a
pointed equivalence between each space and the loop space of the next,
$X_0 \equiv_\star \Omega X_1$, $X_1 \equiv_\star \Omega X_2$, \ldots.
This definition of spectra is readily internalised in type
theory, replacing `pointed spaces' with `pointed types', and has been
formalised~\cite{lean:spectra}. However, most basic constructions on
spectra are still work in progress, as working with such spectra can
be difficult.  Moreover, re-building the theory of spectra from
scratch within type theory is not so much in the spirit of working
synthetically; we should be leveraging the type theory to provide some
of the structure built-in.

One approach to synthetic spectra would be to look for a type theory
that interprets in the $\infty$-category of spectra, but this category
is not an $\infty$-topos, so the type theory would not include all of
axiomatic HoTT (e.g.\ it is not even cartesian closed; like modules
over a ring, it is closed for a different notion of monoidal product).
However, spectra are a full subcategory of \emph{space-parametrised
  families of spectra}, which Joyal and Biedermann have shown form an
$\infty$-topos~\cite[Section 35]{joyal:notes}. The $\infty$-category
$P\Spec$ of parametrised spectra can be pleasantly described as
follows~\cite{ABG}: Given a space $X$, one can think of it as an
$\infty$-groupoid and consider the functor $\infty$-category
$\mathrm{Fun}(X, \Spec)$. The objects of this are our parametrised
spectra over $X$. This assignment from spaces to functor categories
assembles into a functor $\mathcal{S}^\op \to \mathrm{Cat}_\infty$,
with $\mathrm{Cat}_\infty$ the $\infty$-category of
$\infty$-categories. An $\infty$-categorical version of the
Grothendieck construction yields the $\infty$-category $P\Spec$ of
parametrised spectra.  Thus, we can begin by interpreting all of
axiomatic HoTT into this $\infty$-topos $P\Spec$ of parametrised
spectra.  This has important practical advantages over designing a new
type theory for just the spectra, not the families: we can reuse all
of the synthetic homotopy theory that has been developed, and we can
reuse existing proof assistants to work in this setting.

When working in axiomatic HoTT with the model in parametrised spectra
in mind, it will be necessary to have some syntactic way to say when a
type denotes a single spectrum, rather than a family of them --- for
example, the fact that suspension and loop space are an equivalence
holds only for spectra, not families thereof.  Moreover, it will be
necessary say when a type (in $P\Spec$) denotes an ordinary space; for
example, some constructions that we will perform apply only to
families of spectra over a space, not families of spectra over a
spectrum, so in certain places we will want to restrict dependency to
types that denote spaces.  It turns out that we can access both the
full subcategory of spaces and the full subcategory of spectra by way
of a \emph{modality}~\cite{hottbook,rss:modalities}. A basic
operation on $P\Spec$ is the functor that, given a parametrised
spectrum, extracts its underlying space. In the other direction, we
can assign to each space $X$ the parametrised spectrum $(X, 0)$, where
$0 : X \to \Spec$ denotes the constant functor at the zero spectrum,
which has the one-point space at each level (not the empty space,
since the spaces in an $\Omega$-spectrum are all pointed). This is
both left and right adjoint to the forgetful functor from $P\Spec$
to spaces $\mathcal{S}$ that extracts the index space:
\[
\begin{tikzcd}
P\Spec \ar[r] \ar[r, phantom, "\top",yshift=1.5ex, xshift=-0.8ex] \ar[r, phantom, "\top",yshift=-1.5ex, xshift=-0.8ex]&
\mathcal{S} \ar[l, bend left = 40, "0", pos=0.51] \ar[l, bend right = 40, "0" swap]
\end{tikzcd}
\]
and this diagram satisfies the additional coherence of a bireflective
subcategory~\cite{bireflectivity}.  This means that the roundtrip on
$P\Spec$, which we write as $\natural$, has a number of special
properties: is a left-exact idempotent monad and comonad, it is
adjoint to itself, and the counit followed by the unit
$\natural E \to E \to \natural E$ is the identity.  Then, we can
isolate the spaces within the parametrised spectra as the
$\natural$-modal types~\cite{hottbook,rss:modalities}, roughly the
types $A$ such that $A \equiv \natural A$ --- i.e.\ an object is a
space iff it is equivalent to the object consisting of the same base
space with the 0-spectrum over every point, because that means the
object has no interesting spectra to begin with.  Dually, we can
isolate the spectra as the $\natural$-connected types, the types such
that $\natural A$ is contractible --- i.e.\ an object is a single
spectrum iff it is a family of spectra over the point.

Monadic modalities are well-studied in axiomatic
HoTT~\cite{rss:modalities}, but to describe non-trivial comonadic
modalities, one must modify the judgemental structure of the context.
The adjunction between $\mathcal{S}$ and $P\Spec$ is a degenerate
example of \emph{axiomatic cohesion}~\cite{lawvere:cohesion}. An
$\infty$-topos $\mathcal{T}$ is cohesive (over $\mathcal{S}$) if there
is a string of adjoint functors relating $\mathcal{T}$ to $\mathcal{S}$,
satisfying certain conditions. These adjunctions induce a string of
adjoint functors $\shape \dashv \flat \dashv \sharp$ on $\mathcal{T}$ so
that $\shape$ and $\sharp$ are monads and $\flat$ is a comonad. In our
setting, we have $\shape \defeq \flat \defeq \sharp$, hence the name
$\natural$ for our roundtrip operation.  A `cohesive type theory'
capturing the structure of cohesion was introduced by
Shulman~\cite{mike:real-cohesive-hott}, where the three functors also
appear as unary type formers. The part of the theory we are especially
interested in is the $\flat/\sharp$ fragment, which on its own is called
`spatial type theory'. In spatial type theory, the context is divided
into two zones following~\cite{barber:dill,pfenning-davies}: a `modal' zone,
where the types of all variables are morally prefixed with a $\flat$,
followed by a context with ordinary variables. The rules for the $\flat$
and $\sharp$ modalities move variables between the zones to enforce the
correct relationships between the modalities and the judgemental version
of $\flat$ represented by the modal context zone.  In principle one
could use spatial type theory for $P\Spec$ by adding axioms asserting
that the modalities are all equivalent.  While we have no objection in
principle to adding axioms to type theories for synthetic homotopy
theory (we extend axiomatic HoTT, which already adds univalence and
higher inductive types as axioms), the ergonomics of such axioms can be
poor --- in this case, transport across these equivalences would be
pervasive in every construction.

In this paper, we instead describe an extension of the judgements and
rules of type theory, which give the $\natural$ modality the correct
properties without the need for any axioms (Section~\ref{sec:tt}).
Relative to existing modal dependent type theories, the primary
difficulty of the $\natural$ modality is that the bireflection induces a
non-trivial roundtrip on every type\footnote{In full detail, the
  round-trip only applies to types that depend on a space, not a
  spectrum.}: the unit map $A \to \natural A$ followed by the counit map
$\natural A \to A$. The syntax has a pair of novel features to handle
this. Firstly, each variable may be used in two ways: as normal, or
``marked'', which is written $\zx$ and means using it via the roundtrip
described above. Secondly, in contrast to spatial type theory, there is
no separation of the context into two zones, one modal and one
non-modal. The presence of both the unit and counit mean that modal and
ordinary variables can be mixed together without restriction.

Next, to demonstrate that the resulting type theory is quite ergonomic
to use, even without a proof assistant, we develop some synthetic
homotopy theory informally in the style of~\cite{hottbook}.  First, in
Section~\ref{sec:basic-props}, we show that our rules for the $\natural$
modality indeed give it the desired properties---we can prove
internally that it is a monad and comonad, idempotent, self-adjoint,
left-exact, interacts nicely with pointed types, and that the
$\natural$-modal (the ``spaces'') are closed under various type constructors.
Next, in Section~\ref{sec:redu}, we study the $\natural$-connected types
(the ``spectra''), showing that they are pointed and that ordinary
functions between them are automatically pointed, avoiding the need to
carry around proofs of pointedness.  By assuming a base type $\sphere$
(intended to be modelled in $P\Spec$ by the sphere spectrum), we give a
synthetic definition of the adjunction $\Sigma^\infty \dashv
\Omega^\infty$ that relates spaces and spectra, which allows transfer of
information between the unstable and stable worlds.

With the type-theoretic rules for $\natural$, our ``synthetic spectra''
do not have all of the properties of spectra.  This is in some sense an
advantage --- we can also interpret all of the above constructions in a
simpler model in parametrised pointed spaces, which we give some
intuition for in Section~\ref{sec:toy-model}.  But, to bring the
synthetic spectra closer to spectra, we identify two axioms.  The first,
studied in Section~\ref{sec:stab}, makes the internal category of
synthetic spectra \emph{stable}, in particular making $\Sigma \dashv \Omega$
an adjoint equivalence for these types. It turns out to be sufficient to
assert that coproducts and products of synthetic spectra coincide.
Together with the ambient setting of an $\infty$-topos and a left-exact
$\natural$ modality, this causes commutative squares of
spectra to be pushouts if and only if they are pullbacks, which,
together with finite limits/colimits and a zero object, is an equivalent
characterisation of the stability of an $\infty$-category.  The proof
that this axiom implies stability uses a recent advance in synthetic
homotopy theory, the Generalized Blakers-Massey Theorem~\cite{abfj:blakers-massey}.

The second axiom connects our synthetic spectra to `analytic' spectra,
defined concretely as sequences of pointed types with connecting maps.
Using the above stability axiom, we show that the
$\Sigma^\infty \dashv \Omega^\infty$ adjunction between pointed types
and synthetic spectra can be factored through analytic spectra,
defining in particular an adjunction between synthetic and analytic
spectra. For this part, we assume the correctness of an analytic
spectrification operation, which has been defined but not yet proved
correct in type theory. Our second axiom is that this adjunction
between synthetic and analytic spectra is an equivalence.  As an
application, we show that this axiom is is strong enough to fix the
homotopy groups (appropriately defined) of the synthetic sphere
spectrum to be the stable homotopy groups of the ordinary higher
inductive spheres, using some recent work on sequential
colimits~\cite{sbr:seq-colims}.

Finally, in Section~\ref{sec:semantics}, we give a sound and complete
categorical semantics for our type theory, in categories with families
enhanced with extra structure to interpret the modality. In
particular, we require a weak CwF endomorphism such that the underlying
functor on contexts is a bireflection, i.e.\ has a counit-unit
section-retraction pair.

\paragraph{Related work.}

Our rules for the natural modality follow the work on constructive modal
logics such constructive S4, intuitionistic linear logic, and adjoint
logic --- especially presentations with two contexts or different
judgements for modal assumptions --- and their generalisations with
dependent
types~\cite{barber:dill,bentonwadler96adjoint,pfenning-davies,alechina+01categoricals4,cervesatopfenning02llf,nanevski+07cmtt,reed:modal,kpb:lnld,vakar15linear,depaivaritter16fibrational}.
The most directly related calculi are spatial type
theory~\cite{mike:real-cohesive-hott} and the calculi for right adjoint
functors/comonads~\cite{clouston18fitch,drats,gsb:implementing}.  While
these previous works inform our design in Section~\ref{sec:tt}, none
consider a bireflective modality as we do here, and there were still
some interesting design questions specific to our setting.  For example,
because $\natural$ is both a monad and a comonad, the previous work
suggests two possible designs:
\cite{pfenning-davies,mike:real-cohesive-hott} have special context
structure corresponding to a comonad, and the counit is a ``silent''
operation (not marked in the proof term), while~\cite{gsb:implementing}
has special context structure for a monad, and the unit is a silent
operation.  In our setting, we have both a monad and a comonad, but the
``roundtrip'' of the unit followed by the counit is \emph{not} the
identity, so we cannot make both the unit and counit silent.  After some
experimentation, we chose to make the unit silent and the counit
explicit via ``marked'' variables.

Some additional related work
develops frameworks for modal type theories in
general~\cite{lsr:multi,gratzer+20mtt}, but our setting is not quite an
instance of these frameworks. The first~\cite{lsr:multi} lacks
dependent types, but can describe the simply-typed fragment of our type
theory.
The mode theories of the second~\cite{gratzer+20mtt}, do not
allow making the left adjoint types (``locks'') into CwF morphisms,
which corresponds in our setting to defining the natural operation on
contexts as a context of individually marked variables.  Additionally,
our contribution in Section~\ref{sec:tt} is an ``optimised'' syntax
where structural rules are combined with other rules, and as much is
admissible as possible, and the step from these frameworks to an
optimised syntax is currently one that must be undertaken for each type
theory separately in any case.

There is also work by Isaev on `indexed type
theories'~\cite{isaev:itt, isaev:itt-talk}, which correspond to
indexed ($\infty$-)categories. The theory of parametrised spectra can
be described in terms of an indexed $\infty$-category over spaces,
where the $\infty$-category over each fixed space $X$ is
$\mathrm{Fun}(X, \Spec)$. But there are significant benefits in
considering the total category as we do: the existing constructions of
type theory automatically apply to types that correspond to spectra,
whereas in indexed type theory one needs to add an indexed version of
each type constructor together with axioms relating the base and
indexed versions of each. Our ordinary universe $\univ$ is already a
universe of parametrised spectra; in an indexed type theory one has a
universe for the types in the base, but seems difficult to extend this
to indexed types, as only the total category has a universe.

\paragraph{Acknowledgements:} The work in this paper benefited heavily
from discussions with Michael Shulman and Mathieu Anel, in particular
to Mike for suggesting Axiom N. Thank you also to Ed Morehouse and
Alex Kavvos for many helpful conversations.

\newpage
\part{Theory}

\section{The $\natural$ Modality}\label{sec:tt}

In this paper, we begin with homotopy type theory as in~\cite{hottbook}
with $\Pi$, $\Sigma$, identity, and (higher) inductive types, and
univalent universes.  We then add a new \emph{modality}, a unary type
constructor $\natural$.  Unlike the monadic modalities studied
in~\cite{rss:modalities}, which can be described by axioms, our
$\natural$ modality also has some comonadic aspects, which require
changes to the judgements of the type theory.  We first describe a
judgemental version of the modality that applies to contexts and has the
desired unit and counit maps, and then give a type constructor that
internalises this judgemental operation as a type.

\subsection{New Judgemental Rules}

\begin{figure}[p]

\textbf{Marked context extension and variables:}

\[
\begin{array}{ccc}
 \inferrule*[left=ctx-ext-zero]{\Gamma \ctx \and \zc{\Gamma} \yields A : \univ}{\Gamma, \zx :: A  \ctx} &
 \inferrule*[left=var-zero]{~}{\Gamma, \zx :: A, \Gamma' \yields \zx : A}  &
 \inferrule*[left=var-roundtrip]{~}{\Gamma, x : A, \Gamma' \yields \zx : \zA}
 \end{array}
\]

\textbf{Natural on contexts:}

\[
\raisebox{-10px}{$
\inferrule*[fraction={-{\,-\,}-}]{\Gamma \ctx}{\zc{\Gamma} \ctx}$}
\qquad
\begin{array}{rl}
\zc{\cdot} &:\defeq \cdot \\
\zc{\Gamma, x : A} &:\defeq \zc{\Gamma}, \zx :: \zA \\
\zc{\Gamma, \zx :: A} &:\defeq \zc{\Gamma}, \zx :: A
\end{array}
\qquad
\inferrule*[fraction={-{\,-\,}-}]{ }
           {\zc{\zc{\Gamma}} \defeq \zc{\Gamma}}
\]

\textbf{Precomposition with the counit:}

\[
\inferrule*[left=pre-counit,fraction={-{\,-\,}-}]
           {\Gamma \yields a:A}{\zc{\Gamma} \yields \za : \zA}
\qquad
\inferrule*[fraction={-{\,-\,}-}]
           {\zc{\Gamma} \vdash a : A}
           {\zc{\Gamma} \vdash \za \defeq a : A}
\qquad
\inferrule*[fraction={-{\,-\,}-}]{\Gamma \vdash a : A}
           {\zc{\Gamma} \vdash \underline{\underline a} \defeq \underline a : \zA}
\]
\[
\raisebox{-10px}{$
\inferrule*[fraction={-{\,-\,}-}]{\Gamma \yields \Delta \tele}{\zc{\Gamma} \yields \Delta^{0\Gamma} \tele}$}
\qquad
\begin{array}{rl}
(\cdot)^{0\Gamma} &:\defeq \cdot \\
(\Delta, x : A)^{0\Gamma} &:\defeq \Delta^{0\Gamma}, x : A^{0\Gamma} \\
(\Delta, \zx :: A)^{0\Gamma} &:\defeq \Delta^{0\Gamma}, \zx :: A
\end{array}
\qquad
\inferrule*[left=pre-counit-gen,fraction={-{\,-\,}-}]
           {\Gamma,\Delta \yields a:A}{\zc{\Gamma}, \Delta^{0\Gamma} \yields a^{0\Gamma} : A^{0\Gamma}}
\]

\textbf{Precomposition with the unit/roundtrip:}
\[
\inferrule*[left=pre-unit,fraction={-{\,-\,}-}]
           {\Psi,\zc{\Gamma},\Delta \yields a : A}
           {\Psi,\Gamma,\Delta \yields a : A}
\qquad
\inferrule*[left=pre-roundtrip,fraction={-{\,-\,}-}]
           {\Gamma \yields a : A}
           {\Gamma \yields\za : \zA}
\qquad
\inferrule*[left=pre-roundtrip-gen,fraction={-{\,-\,}-}]
           {\Gamma,\Delta \yields a :A}
           {\Gamma, \Delta^{0\Gamma} \yields a^{0\Gamma} : A^{0\Gamma}}
\]

\caption{New Context Structure.  Dashed lines indicate admissible
  rules.}\label{fig:structural}
\end{figure}


At a high level, the reason we require some new judgemental structure
for presenting the $\natural$ modality is that it has both a unit $A \to
\natural A$ and a counit $\natural A \to A$ and the \emph{roundtrip} $A
\to \natural A \to A$ gives a non-trivial map for general types $A$. So
to $\beta$-reduce an introduction followed by an elimination, we need a
judgemental version of this roundtrip map $A \to A$ to reduce to.  This
must satisfy a naturality equation, which says that for any $f : A \to
B$, the composites $A \to B \to \natural B \to B$ and $A \to \natural A
\to A \to B$ are equal: post-composing with the roundtrip is the same as
precomposing with it.  To achieve this, we add post-composition with the
roundtrip as a new way of using variables: for a variable $x : A$, we
write $\zx$ for the roundtrip on $A$ applied to $x$, and say that the
variable is \emph{marked} or \emph{zeroed} or \emph{dull}.  For a
general term $a$, we define an admissible operation $\za$ that
``underlines all of the free variables of $a$'', and denotes
precomposing with the roundtrip on the context.  This gives normal forms
for the above naturality equations.  There is an interaction with
dependency, because applying the roundtrip to a term must also apply it
to the term's type, so if $a : A$ then $\za : \zA$.  For certain rules,
we will need a judgement classifying terms whose free variables are
\emph{only} used marked, which we call \emph{dull terms}. We accomplish
this by allowing variable declarations in the context to be marked as
well, written $\zx :: A$, which semantically is the same as $x :
\natural A$.  A variable declared marked in the context can only be used
marked, so a term in a context of only marked declarations must be
dull.  Any context $\Gamma$ can be turned into a context of only marked
declarations via a defined $\zc{\Gamma}$ operation, which semantically
is applying the $\natural$ modality to the context.  Finally, we have
admissible structural rules corresponding to precomposing a judgement
with the unit $\Gamma \to \zc{\Gamma}$ and counit $\zc{\Gamma} \to
\Gamma$.  We now discuss the rules in Figure~\ref{fig:structural} in
more detail.



\paragraph{Marked Context Extension and Variables.}
To make the type theory easier to use, we will make $\zc{\Gamma}$ an
\emph{admissible} (defined by induction on syntax) operation, rather
than a \emph{derivable} (new piece of formal syntax) one. The new piece
of formal syntax is a new context \emph{extension} $\Gamma,\zx :: A$
(\rulen{ctx-ext-zero}), which semantically is the same as
$\Gamma.\natural A$.  We think of this special context extension as
``marking'' a variable by writing $::$ in place of $:$, and by writing
an underscore under the variable name.  Formally, either one of the
underscore or the $::$ would be enough, but we find the syntax
clearer if variables are written the same in the context and at their
use sites, which will be marked; and we carry over the $::$ from spatial
type theory~\cite{mike:real-cohesive-hott}.

The restriction imposed by a marked context extension is that a
variable that is declared marked in the context $\Gamma,\zx :: A$ can
only be used marked as $\zx$ in a term, as indicated in the
\rulen{var-zero} rule. Semantically, this is using the counit $\natural
A \to A$, because $\zx :: A$ in the context is semantically $\natural
A$, but $A$ as a type on the right is semantically just $A$.  In
concrete syntax, it is best to think of $\zx$ as a term constructor
$\textsf{underline}(x)$, not that the underline is part of the variable
name.

A new ingredient for the $\natural$ modality is that there are two ways
to use an \emph{ordinary} unmarked variable declared in the context as
$\Gamma,x : A,\Gamma'$. In addition to the standard rule
$\Gamma,x:A,\Gamma' \vdash x : A$, we can also apply the
\rulen{var-roundtrip} rule, and use the variable $x$ by writing $\zx$.
We say that such a use is \emph{marked} or \emph{zeroed}, because
semantically, it corresponds to the projection from $\Gamma$ post-composed with the roundtrip $A \to \natural A \to A$.
In our intended models, this keeps the base of $x$ the same but
replaces the fibres of $x$ with the default sections of $A$.  We
intentionally use the same raw syntax for these two distinct typing
rules \rulen{var-zero} and \rulen{var-roundtrip}, one of which uses a
variable that is marked in the context (via the counit), and the other
uses a variable that is unmarked in the context (via the roundtrip).
This allows precomposition with the unit $\Gamma \to \zc{\Gamma}$ to be
a ``silent'' operation that leaves the raw syntax unchanged.
We explain the operation $\zA$ marking a type below.

\paragraph{Natural on Contexts}
Proceeding to the admissible rules in the figure, marking a context
$\zc{\Gamma}$ is defined inductively by marking all of the
variables in a context.  Semantically, these equations say (roughly ---
there are some subtleties with dependency that we discuss below) that
$\natural 1 = 1$, $\natural (\Gamma.A) = \natural \Gamma.\natural A$
(which are the equations of a strict CwF morphism~\cite[Definition
  2]{dybjer:cwf}) and that $\natural(\Gamma.\natural A) =
\natural \Gamma.\natural A$ (which is reasonable because of idempotence
$\natural \natural A = \natural A$).  The equation states that
$\zc{\Gamma}$ is idempotent---syntactically, $\zc{\Gamma}$ has only
marked variable declarations, which $\zc{\zc{ \Gamma}}$ leaves
unchanged (and the $\zA$ operation discussed next is also idempotent).

Putting together \rulen{ctx-ext-zero} and \rulen{var-zero} and the
definition of $\zc{\Gamma}$, the types of later marked variables can
depend on earlier unmarked ones (but can only use them marked). For
example, in a context $x : A, \zy::B$, the type $B$ is in context
$\zc{x:A} \defeq \zx::\zA$, so may refer to $\zx$ (but not $x$).  This
means we cannot put all of the marked variables in a separate context
zone preceding the unmarked variables~\cite{barber:dill,pfenning-davies}.

\paragraph{Precomposition with the Counit.}
We will often be interested in types and terms where \emph{every} use of a free
variable is marked/zeroed. We call such types and terms \emph{dull}. A
dull term in context $\Gamma$ is equivalently a term in the context
$\zc{\Gamma}$---because all the variables in $\zc{\Gamma}$ are marked,
any term $\zc{\Gamma} \yields a : A$ necessarily uses these variables
marked.  We can turn any $\Gamma \vdash a : A$ into a dull term
$\zc{\Gamma} \vdash \za : \zA$ by marking all the free variable uses in
$\za$ with an underscore (\rulen{pre-counit}).  (Note that we overload
the notation and write the $\za$ operation on general terms using the
same syntax as for marked/roundtripped variables.) Semantically, $\za$
is precomposing $a$ with the counit $\zc{\Gamma} \to \Gamma$.  Because
the type $A$ also depends on $\Gamma$, substitution by the counit will
also mark its variables, which we write as $\zA$. We think of types
as elements of a universe, so $\Gamma \vdash A : \univ$ implies
$\zc{\Gamma} \vdash \zA : \univ$ is another instance of this rule.
This means that when a term is marked/zeroed i.e.\ only varies over
the underlying space of the context, its corresponding type is also
only permitted to vary over the underlying space of the context.

We omit a formal definition of $\za$ on raw syntax, which is given by
recursion over the term syntax (like substitution), turning $x$ into
$\zx$, leaving marked variable uses unchanged, and proceeding
recursively otherwise (e.g.\ $\underline{f(a)} =
\underline{f}(\underline{a})$). In particular, zeroing commutes with the
type former for $\Idsym$-types: $(\zc{x = y}) \defeq (\zx = \zy)$, which
is part of what makes our modality left-exact.

One subtlety is that, because $\za$ is semantically
substitution/precomposition with the counit $\zc{\Gamma} \to \Gamma$,
it marks the \emph{free} variables of a term, but leaves the bound
variables the same. For example,
$\underline{\lambda x.f x} \defeq \lambda x.\underline f x$.  This
leads to the full form of the operation in \rulen{pre-counit-gen},
which marks the variables in $\Gamma$ in the context and changes all
occurrences of the $\Gamma$-variables in the term into marked ones,
but does not change the occurrences of $\Delta$-variables in the term.
Formally, $\Delta$ is a telescope (context in context), but we omit
the rules for $\Gamma \vdash \Delta \,\, \mathsf{telescope}$, with
formation rules analogous to those for contexts, and the operation of
concatenating a context and a telescope $\Gamma,\Delta$.  For example,
with $f$ in $\Gamma$ and $x$ in $\Delta$, $\underline{f x} = \zf x$.
However, when the variables declared in $\Gamma$ occur in the types in
$\Delta$, those occurrences must be marked, which we notate with
${\Delta}^{0\Gamma}$. The telescope ${\Delta}^{0\Gamma}$ is defined by
sending $x:A$ to $x:A^{0\Gamma}$ (\emph{not} $\zx::A^{0\Gamma}$,
differing from $\zc{\Gamma}$) and $\zx::A$ to $\zx :: A$ (since all
variables are already zeroed) for each variable in $\Delta$.
Officially, we should be annotating the underscores like $\za_\Gamma$
to indicate which variables in $a$ are to be zeroed, but when we
use this operation informally we always start with $\Gamma$
being all free variables and $\Delta$ empty, so we adopt a
convention that $\za$ means to mark all free variables of $a$.

The equations for precomposition with the counit state that if a term
starts out in $\zc{\Gamma}$, then marking has no effect $\za \defeq a$.
Syntactically, this is because a term in context $\zc{\Gamma}$ cannot
have any unmarked free variable uses, which are the only parts of a term
changed by the marking operation.  Note that this equation needs
$\zc{\zc{\Gamma}} \defeq \zc{\Gamma}$ to type check, and semantically
corresponds to the counit on $\zc{\Gamma}$ being the identity.
Consequently, marking is idempotent: $\zc{\za} \defeq \za$.

\paragraph{Precomposition with the Unit.}
There is an analogous operation of precomposition with the unit
$\Gamma \to \zc{\Gamma}$. Following~\cite{gsb:implementing}, we make
this a ``silent'' operation, i.e.\ it does not change the raw syntax
of the term or the type, only the typing derivation, as stated in
\rulen{pre-unit}. We refer to as use of the unit as ``unzeroing'' a
piece of the context, because variables that are marked/zeroed in the
context $\zc{\Gamma}$ in the premise become unmarked/unzeroed in the
conclusion.  The unit does \emph{not} unmark the \emph{uses} of the
variables in a term or type --- a use $\zx$ of a marked variable
$\zx :: A$ from $\zc{\Gamma}$ (typed by \rulen{var-zero}, which is the
counit $\natural A \to A$) becomes a use $\zx$ of $x : A$ from
$\Gamma$ (typed by \rulen{var-roundtrip}, which is the roundtrip
$A \to \natural A \to A$, the counit precomposed with the
unit).  Thus, the unit can be silent because we
use the same syntax for the counit on marked variables as for the
roundtrip on unmarked variables.
\rulen{var-roundtrip} in the typing derivation.

To work up to the rule in the figure, the most basic form, where
$\Psi$ and $\Delta$ are empty, says that any
$\zc{\Gamma} \vdash a : A$ is also $\Gamma \vdash a : A$.  For the
same reasons as for the counit, we will need a tail telescope $\Delta$
that is not ``unzeroed'' by the operation (i.e.\ the marks in $\Delta$
in the premise are still there in the conclusion), for inductively
pushing this operation under bound variables, which are not unzeroed
(and indeed, might not even be marked in the premise). We will also
sometimes find it useful to unzero a variable in the middle of the
context, without unzeroing its prefix, e.g.\ going from
$\Gamma,\zx::A \vdash \judge$ to $\Gamma,x : A \vdash \judge$.
Semantically, this is precomposition with the unit $A \to \natural A$
paired with the identity substitution on $\Gamma$.
It is an implicit requirement for the judgement in the conclusion to
be well-formed that $\Psi,\Gamma,\Delta$ is a well-formed context


\paragraph{Precomposition with the Roundtrip.}
Composing \rulen{pre-unit} and \rulen{pre-counit}, we have a rule
\rulen{pre-roundtrip} representing precomposition with the non-trivial
roundtrip $\Gamma \to \natural \Gamma \to \Gamma$.  (We have not seen a
use for a counit rule with a prefix $\Psi$ as in the unit rule, so we do
not include one, and consequently restrict the roundtrip to the setting
where both the unit and counit exist, when $\Psi$ is empty for the
unit.)  The section-retraction property of a bireflection states that
composing \rulen{pre-unit} and \rulen{pre-counit} in the other
direction, i.e.\ going from $\zc{\Gamma} \vdash a : A$ to $\Gamma \vdash
a : A$ to $\zc{\Gamma} \vdash \za : \zA$ should be the identity; because
the unit is silent, this is the same as the counit equation $\za \defeq
a$.

Returning to the rule \rulen{var-roundtrip}, the type $\zA$ in the
conclusion is typed by \rulen{pre-roundtrip}, because, as for the
counit, precomposing/substituting by the roundtrip on $\Gamma$
substitutes into the type $A$ as well.  Because variable uses $x : A$
and $\zx : \zA$ (in general) have different types, the marked-ness of a
variable usage cannot be na\"{\i}vely flipped at will in a term. For
example, if $x : A$ then $\zx = x$ may not be well-formed, as $\zA$ is
not in general the same type as $A$.  The only reason that we do not
need to analogously mark the type $A$ in the conclusion of
$\rulen{var-zero}$ as $\zA$ is that the marked context extension
\rulen{ctx-zero} `pre-zeroes' the type --- the type $A$ is in context
$\zc{\Gamma}$, so must already use only marked variables.

\paragraph{Well-formedness of the Conclusions.}
Whenever we can form a term $\Gamma \yields a : A$, we of course want
that $\Gamma \ctx$ and $\Gamma \yields A : \univ$. (Depending on
precisely how the type theory is set up, these are sometimes
presuppositions of the term judgement; later when checking that
\rulen{pre-counit} and \rulen{pre-unit} are admissible in
Section~\ref{sec:admissible-rules}, we follow~\cite{streicher:book} in
having them be consequences of the term judgement.) There are few
spots in the above rules where it is a bit subtle why these invariants
are maintained.  First, in \rulen{var-zero}, we have by
\rulen{ctx-ext-zero} that $\zc{\Gamma} \yields A : \univ$, but for the
use of $A$ on the right, we need $\Gamma,x:A,\Gamma' \yields A : \univ$.
In addition to the usual weakening with $x:A,\Gamma'$, this uses
\rulen{pre-unit}.  In \rulen{var-roundtrip}, we have
$\Gamma \yields A : \univ$, so by another application of
\rulen{var-roundtrip}, we also have $\Gamma \yields \zA : \univ$, so all
that is needed is the usual weakening.  In the definition of
$\zc{\Gamma}$ for unmarked variables, we begin with
$\Gamma \yields A : \univ$, and need $\zc{\Gamma} \yields \zA : \univ$,
which we have by \rulen{pre-counit} (with $\Delta$ empty).  In the
definition for marked variables, we start with
$\zc{\Gamma} \yields A :\univ$, and need
$\zc{\zc{\Gamma}} \yields A : \univ$ to apply \rulen{ctx-ext-zero},
which holds by idempotence.  In the equation $\za \defeq a$, we need
the same equation on types to see that $\zA \defeq A$, and similarly
for the $\underline \za \defeq \za$ equation.

\paragraph{Substitution for Marked Variables.}
There is a new case of standard substitution for substituting into
\rulen{var-roundtrip}, which is defined by
\[
\zx [a / x] :\defeq \za
\]
That is, when we substitute a term $a : A$ for a marked variable usage
$\zx$, the result is the marking of $a$.  This type checks for
$\Gamma,x:A \vdash \zx : \zA$ and $\Gamma \vdash a : A$ because the
\rulen{pre-roundtrip} rule gives $\Gamma \vdash \za : \zA$.
Semantically, $\zx$ is the roundtrip $A \to \natural A \to A$, and the
substitution post-composes this roundtrip with $a$; but $\za$ is $a$
pre-composed with the roundtrip $\Gamma \to \natural \Gamma \to
\Gamma$, and these are equal by naturality of the unit and counit.

A substitution principle for marked variables that is typical from other
comonadic type theories following~\cite{pfenning-davies} is
\[
\inferrule*[fraction={-{\,-\,}-}]
           {\Gamma, \zx :: A, \Gamma' \yields b : B \and \zc{\Gamma} \yields a: A}
           {\Gamma, \Gamma'[a/\zx] \yields b[a/\zx] : B[a/\zx]}
\]
Here, the term being substituted must already have all of its variables
marked, as indicated by the premise $\zc{\Gamma} \vdash a : A$ of the rule,
and this substitution principle is implemented by a syntactic
substitution, replacing $\zx$ with $a$ everywhere.  Given the admissible
rules in Figure~\ref{fig:structural}, we can in fact define this by
first unzeroing the variable $\zx$ to get $\Gamma, x : A, \Gamma'
\yields b : B$ and then doing an ordinary substitution $b[a/x]$.  Since
all uses of $x$ will be marked $\zx$ in $b$ and $B$, this will replace
$\zx$ with $\za$ everywhere---but since $\zc{\Gamma} \yields a : A$, we have
$\za \defeq a$, so we get the same result as the more specialised
principle would have given.

We prefer this style of presenting substitution, where the substitution
for marked variables is given by unmarking and then ordinary
substitution, because it corresponds more closely to what we will do
when working informally in this type theory.  When performing
substitutions $b[a/x]$ by hand, we can simply look in $b$ for each
instance of $x$ and $\zx$ and replace them with either $a$ or $\za$
accordingly, without having to mentally keep track of the context
through each subterm to see whether $x$ is marked or not (variables that
are not marked become marked in the premises of some rules), as we would
have to do if substitution for a marked variable required pre-marking
the term.  This is another benefit of having \rulen{var-zero} and
\rulen{var-roundtrip} rules be identical raw syntax.  In our experience
trying different systems for this setting, this choice seems critical
for the usability of the system for informal type theory.

\paragraph{Comparison with Spatial Type Theory.}
In spatial type theory~\cite{mike:real-cohesive-hott}, there is also a
special context extension $x :: A$, which is a judgemental version of
extending the context with $\flat A$. Such variables are
called \emph{crisp}. A usage of a crisp variable $x$ in a term
corresponds to a use of the counit $\flat A \to A$, like our
\rulen{var-zero} rule.  Because there is only a counit and not also a
unit, the type $A$ of a crisp variable $x :: A$ is only permitted to
depend on other crisp variables. A loose way to think about this is as
follows. Before we have access to any structural rules, dependency
forces us to apply modalities to an entire context at once. Given a
type-in-context $\Gamma \yields A : \univ$ presented as a fibration $p : A
\to \Gamma$, ordinary context extension corresponds to considering the
object $A$ as a context. If we want to make $A$ discrete, we have to
apply $\flat$ to everything, giving $\flat p : \flat A \to \flat
\Gamma$. So $\flat A$ can only depend on a discrete context, and the
judgemental version $x :: A$ has the same restriction.  Therefore, all
crisp variables must occur before regular ones, and the context
naturally divides into two zones.  In our system, however, the presence
of the unit map $A \to \natural A$ means that we can no longer neatly
divide the context in this way. For example, if we have a context $\zx
:: A, \zy :: B$, then we can precompose with the unit substitution just
on $\zx$, giving $x : A, \zy :: B$. This breaks the invariant that crisp
variables all occur before ordinary ones.

\subsection{The $\natural$ type}

\begin{figure}
\[
\inferrule*[left=$\natural$-form]{\zc{\Gamma} \yields A : \univ }{\Gamma
  \yields \natural{A} : \univ} \\
\qquad
\inferrule*[left=$\natural$-intro]{\zc{\Gamma} \yields a : A}{\Gamma \yields a^\natural : \natural{A}} \and
\quad
\inferrule*[left=$\natural$-elim]{\Gamma \yields b : \natural A}{\Gamma \yields b_\natural : A} \\\\
\]
\[
\inferrule*[left=$\natural$-beta]{\zc{\Gamma} \yields a : A}{\Gamma \yields a^\natural{}_\natural \defeq a : A} \and
\quad
\inferrule*[left=$\natural$-eta]{\Gamma \yields b : \natural A}{\Gamma \yields b \defeq \zb{}_\natural{}^\natural : \natural A}
\]
\caption{Rules for $\natural$}\label{fig:natural-rules}
\end{figure}

Using this judgement structure, it is now simple to describe the
$\natural$ type using the rules in Figure~\ref{fig:natural-rules}.

Recall that we refer to a term/type in context $\zc{\Gamma}$, i.e.\ a
term/type all of whose free variables are marked, as \emph{dull}.  The
formation rule says that for any dull type $A$ there is a type $\natural
A$.  Formally, this formation rule is analogous to $\sharp$ in spatial
type theory or dependent right adjoints~\cite{drats}, in that it asks
for a type under the left adjoint of $\natural A$, which in this case is
also $\natural$, represented by $\zc{\Gamma}$.  One alternate rule that
one could imagine is like $\flat$ in spatial type theory, $\zc{\Gamma}
\vdash A : \univ$ implies $\zc{\Gamma} \vdash \natural A : \univ$.  However,
this rule breaks admissibility of precomposition with the unit, because
it forces variables in the conclusion's context to be marked.

The introduction rule says that for any dull term of a dull type $a : A$
there is a term $a^\natural : \natural A$, again transposing $\natural$
on the right to $\natural$ on the left (roughly, $\natural \Gamma \to A$
implies $\Gamma \to \natural A$). This is the same as the introduction
rule for $\sharp$ and dependent right adjoints.  Note that the type $A$
must be assumed to be dull for the type $\natural A$ in the conclusion
to be well-formed.

The elimination rule says that for any (not necessarily dull) term $b :
\natural A$, there is a term $b_\natural : A$.  Semantically, this is
the counit $\natural A \to A$ precomposed with $b$.  Note that the type
$A$ must be assumed to be dull for $\natural A$ in the premise to be
well-formed --- for a non-dull type $A$, we have a counit $\natural \zA
\to \zA$, but in general we do not have a map $\natural \zA \to A$.

The computation or $\beta$-reduction rule says that
$a^\natural{}_\natural \defeq a$.  Whenever the left-hand side is
well-typed, the right-hand side is too, because of the silent unit rule
\rulen{pre-unit}. Note that $a$ is necessarily dull for the
\rulen{$\natural$-intro} rule to have been applied, and all of its free
variables are still marked zeroed on the right.

The uniqueness or $\eta$-rule says that $b \defeq
\zb{}_\natural{}^\natural$ for any term $b : \natural \zA$. Since $b$ is
not necessarily dull, it must be marked (precomposed with the counit)
before being used in the introduction rule $-^\natural$.  One must be
cautious in applying this rule from right to left, as not every possible
`unzeroing' of a term $\zb$ will be well-typed.

For a non-dull type $\Gamma \vdash A : \univ$, note that $\zA$ (given by
\rulen{pre-counit}) and $\natural \zA$ are very different.  In
parametrised pointed spaces/spectra, $\zA$ is $A$ with its dependency on
the fibres of the context $\Gamma$ replaced by the sections of $\Gamma$.
On the other hand, $\natural \zA$ also replaces the fibres of $A$ with
the trivial pointed space/spectrum.  For example, if $A$ is a closed
type then $\zA \defeq A \not \equiv \natural A$.
From this point of view, our notation $\zc{\Gamma}$ for marking a
context is confusing, because it semantically is $\natural \Gamma$;
however, we use this notation to emphasise that it is implemented by
``underlining all of the variables in $\Gamma$''.

We have not proved canonicity or normalisation for the $\natural$ type,
as our intended applications rely on many axioms, but we conjecture they
are true: the equations for $\zc{\Gamma}$ and $\za$ are proved rather
than asserted, and the $\natural$ type has a $\beta$ rule for weak head
reduction and a type-directed $\eta$ rule.

\subsection{Alternative Rules Without Marked Context Extension}

\begin{figure}
\[
\begin{array}{c}
 \inferrule*[left=var-roundtrip]{~}{\Gamma, x : A, \Gamma' \yields \zx : \zA}
 \qquad
 \inferrule*[left=pre-roundtrip,fraction={-{\,-\,}-}]
           {\Gamma \yields a : A}
           {\Gamma \yields\za : \zA}
           \qquad
 \inferrule*
           {\Gamma \yields a : A}
           {\Gamma \yields \zc{\za} \defeq \za  : \zA}
\\\\
\inferrule*[left=$\natural$-form]{\Gamma \yields A : \univ }{\Gamma
  \yields \natural{A} : \univ} \qquad
\qquad
\inferrule*[left=$\natural$-zero]{\Gamma \yields A : \univ }{\Gamma
  \yields \natural{A} \defeq \natural{\zA} : \univ}
\\ \\
\inferrule*[left=$\natural$-intro]{\Gamma \yields a : A}{\Gamma \yields
  a^\natural : \natural{A}}
\qquad
\inferrule*[left=$\natural$-elim]{\Gamma \yields b : \natural A}{\Gamma
  \yields b_\natural : \zA}
\qquad
\inferrule*[left=$\natural$-beta]{\Gamma \yields a : A}{\Gamma \yields a^\natural{}_\natural \defeq \za : \zA}
\qquad
\inferrule*[left=$\natural$-eta]{\Gamma \yields b : \natural A}{\Gamma
  \yields b \defeq b{}_\natural{}^\natural : \natural \zA \defeq \natural A}
\\\\
\end{array}
\]
\caption{Rules without marked context extension}\label{fig:rules-no-marked-context-extension}
\end{figure}



As an alternative to the above rules, it is also possible to give rules
for $\natural A$ without marked context extension
$\rulen{ctx-ext-zero}$, which we show in
Figure~\ref{fig:rules-no-marked-context-extension} (\rulen{pre-roundtrip}
needs to be generalized with a $\Delta$ like above, but we omit this
from the figure).

The idea with this variation is to have a judgemental account only of
the roundtrip idempotent, and not the object $\zc{\Gamma}$ that the
roundtrip splits through.  We still include the marked variable
\emph{uses} (representing the roundtrip applied to a variable) and the
admissible operation of underlining terms $\za$ (representing
precomposition with the roundtrip on the context, so it does not change
the context).  The formation, introduction, and elimination rules are
essentially $\natural : \univ \to \univ$ and $-^\natural : A \to
\natural A$ and $-_\natural : \natural A \to \zA$.  The $\beta$ rule
says that the composite $A \to \natural A \to \zA$ is the roundtrip, and
the $\eta$ rule says that the composite $\natural A \to \zA \to \natural
A$ is the identity.  We explicitly include an equation $\natural A
\defeq \natural \zA$ because semantically these types should be equal
--- in Figure~\ref{fig:natural-rules}, we only allow the latter to be
written, but that requires marked context extension to enforce.  The
$\eta$ rule requires $\natural A \defeq \natural \zA$ for both sides to
have the same type.

While the rules in Figure~\ref{fig:rules-no-marked-context-extension}
are shorter to describe, we prefer the rules from
Figures~\ref{fig:structural} and~\ref{fig:natural-rules} with marked
context extension for several reasons.  First, they do not require an
extra equation \rulen{$\natural$-zero}.  Second, when working informally
below, we will often assume a dull/marked variable $\zx :: A$, which has
the same meaning as assuming a variable $x : \natural A$, but using a
dull variable by writing $\zx$ is a bit terser than writing
$x_\natural$, and substitution does not need to go through a
$\beta$-reduction.  However, the main reason we prefer the marked
variables in the context is because they are necessary for future
work on adding the `smash product of spectra' as a tensor type,
as in linear logic.  The base of every type can be duplicated, so when
splitting the context to type check $c : C \vdash (a,b) : A \otimes B$,
we can e.g.\ allow a split into $c : C \vdash a : A$ and $\zc{c} :: \zC
\vdash b : B$, where one component receives all of $c$ and the other
receives only the base of $c$.  Presenting this nicely requires marked
variables in the context.

\section{Basic Properties of $\natural$}\label{sec:basic-props}

For the next several sections, we work in axiomatic
HoTT~\cite{hottbook} (with $\Pi$, $\Sigma$, identity, and (higher)
inductive types, and univalent universes) with our new rules for
$\natural$.  First, we develop the basic structure of the $\natural$
type internally, proving that $\natural$ behaves like both the $\flat$
and $\sharp$ modalities of spatial type
theory~\cite{mike:real-cohesive-hott}.

\begin{definition}[Unit and counit for $\natural$]
The introduction and elimination rules immediately give, for any type $A$, unit and counit maps
\begin{align*}
\eta_A :\defeq (\lambda x.\zx{}^\natural) &: A \to \natural \zA \\
\varepsilon_A :\defeq (\lambda n. n_\natural) &: \natural \zA \to \zA
\end{align*}
\end{definition}

The fact that the we only have a counit for \emph{dull} types is what
defeats the `no-go theorem' for comonadic modalities~\cite[Theorem
  4.1]{mike:real-cohesive-hott}. In general, there is no way to go from
a term of $\natural \zA$ or $\zA$ to a term of the non-zeroed type $A$.

\begin{proposition}
  The counit and unit are a section-retraction pair, i.e.\ the roundtrip $\eta_A \circ \varepsilon_A :
  \natural \zA \to \zA \to \natural \zA$ is the identity.  The composite
  $\varepsilon_A \circ \eta_A : A \to \natural \zA \to \zA$ is equal to
  $\lambda x.\zx$.
\end{proposition}

\begin{proof}
For $n : \natural \zA$ we have
\[
\eta(\varepsilon(n)) \defeq \zc{n_\natural}^\natural \defeq \zc{n}_\natural^\natural \defeq n
\]
by the definition of $\zc{ }$ and the $\eta$-law.  For the composite on $x : \zA$, we get
\[
\varepsilon(\eta(x)) \defeq {\zx{}^\natural}_{\natural} \defeq \zx
\]
by the $\beta$-law.
\end{proof}

\begin{definition}\label{def:natural-is-a-functor}
We can define the functorial action of $\natural$ on a map, in any ambient context: given $f : A \to B$ we define $\natural \zf : \natural \zA \to \natural \zB$ by:
\begin{align*}
\natural f (x) :\defeq [\zf(\zx{}_\natural)]^\natural
\end{align*}
yielding a map $(A \to B) \to (\natural \zA \to \natural \zB)$.  When $f$ is $\lambda y.y$ we get
\[
\natural (\lambda y.y) (x) \defeq [\zc{(\lambda y.y)}(\zx{}_\natural)]^\natural \defeq
[(\lambda y.y)(\zx{}_\natural)]^\natural
\defeq
(\zx{}_\natural)^\natural \defeq x
\]
When $f$ is $f_2 \circ f_1$ we first have
\[
\natural (f_2 \circ f_1) (x) \defeq [\zc{(f_2 \circ f_1)}(\zx{}_\natural)]^\natural
\defeq [(\zc{f_2} \circ \zc{f_1})(\zx{}_\natural)]^\natural
\defeq [\zc{f_2}(\zc{f_1}(\zx{}_\natural))]^\natural
\]
But we also have
\[
(\natural (f_2) \circ \natural (f_1)) (x)
\defeq
\natural (f_2) ([\zc{f_1}(\zx{}_\natural)]^\natural)
\defeq
[\zc{f_2}(\zc{([\zc{f_1}(\zx{}_\natural)]^\natural)}_\natural)]^\natural
\defeq
[\zc{f_2}(([\zc{f_1}(\zx{}_\natural)]^\natural)_\natural)]^\natural
\defeq
[\zc{f_2}(\zc{f_1}(\zx{}_\natural))]^\natural
\]
So $\natural f$ preserves identity and composition definitionally.
\end{definition}

\begin{remark}
Note that, in contrast with $\flat$ of spatial type theory, we do not need the function $f$ to be `crisp', i.e., only use modal variables. Here, we can turn any function $f : A \to B$ into a `crisp' one $\zf : \zA \to \zB$ by zeroing, allowing us to apply \rulen{$\natural$-intro} to $\zf(\zx{}_\natural) : \zB$.
\end{remark}

The $\eta$-rule for $\natural \zA$ implies that any term of natural type
is equal to the marked version of it:
\begin{proposition}\label{prop:zero-in-natural}
  For any dull type $\zc{\Gamma} \vdash A : \univ$ and not necessarily
  dull term $\Gamma \vdash a : \natural A$ there is a definitional
  equality $\Gamma \vdash a \defeq \za : \natural A$
\end{proposition}

\begin{proof}
  Suppose a term $\Gamma \vdash a : \natural A$.  By the $\eta$-rule, we
  have $a \defeq {\za_\natural}^\natural$.  But applying the admissible
  roundtrip rule \rulen{pre-roundtrip}, we have $\Gamma \vdash \za :
  \natural A$ (using the fact that $\zA \defeq A$ because $A$ is dull).
  Applying the $\eta$-rule to that gives $\za \defeq
  {\zc{\za}_\natural}^\natural$. But $\underline{\zn} \defeq \zn$).

  Semantically, this is because, for any $f : \Gamma \to \natural A$,
  the composite with the roundtrip $\Gamma \to \natural \Gamma \to
  \Gamma \to \natural A$ is still equal to $f$ --- first, use naturality
  of the unit/counit to see this is equal to $\Gamma \to \natural A \to
  \natural \natural A \to \natural A$ and then the latter two maps are
  inverse by idempotence of $\natural$.
\end{proof}

\begin{proposition}
The unit and counit are natural, so for any $f : A \to B$ the diagrams
\begin{mathpar}
\begin{tikzcd}
A \ar[r, "f"] \ar[d, "\eta_A" swap] & B \ar[d, "\eta_B"] \\
\natural \zA \ar[r, "\natural f" swap] & \natural \zB
\end{tikzcd}
\and
\begin{tikzcd}
\natural \zA \ar[r, "\natural f"] \ar[d, "\varepsilon_A" swap] & \natural \zB \ar[d, "\varepsilon_B"] \\
\zA \ar[r, "\zf" swap] & \zB
\end{tikzcd}
\end{mathpar}
commute.
\end{proposition}
\begin{proof}
On the left: \[\zf((\za^\natural){}_\natural)^\natural \defeq \zf(\za)^\natural \defeq \zc{f(a)}^\natural \]
On the right:
\[ (\zf(\zx{}_\natural)^\natural)_\natural \defeq \zf(\zx{}_\natural) \defeq \zf(x{}_\natural) \]
\end{proof}

We now consider types that are equivalent to their `underlying space'.
\begin{definition}
A type $A$ is \emph{modal} if the unit $\eta_A : A \to \natural \zA$ is an equivalence. We define
\begin{align*}
\Modal &:\defeq \sm{X : \univ} \isEquiv(\lambda x. \zx{}^\natural)
\end{align*}
for the type of modal types.
\end{definition}
This definition is studied in detail in \cite[Definition 7.7.5]{hottbook},\cite[Section 1]{rss:modalities}. In the intended model, the modal
types are the spaces, embedded in $P\Spec$ as a space equipped with the
constant zero family of spectra.  We will sometimes need to restrict
statements to such spaces, so it is important that we can carve out a
subuniverse of spaces using the modality.

\begin{proposition}\label{prop:underline-modal} If $A$ is modal then $\zA$ is also modal.
\end{proposition}
\begin{proof}
Given a witness $w :
\isModal(A)$, we have $\zw : \isModal(\zA)$, showing $\zA \equiv
\natural \zA$.
\end{proof}

\begin{proposition}\label{prop:natural-type-is-a-space}
For any $A$, the type $\natural \zA$ is modal.
\end{proposition}
\begin{proof}
  The counit $(\lambda z. z_\natural) : \natural \natural \zA \to \natural \zA$ is an inverse to the unit $(\lambda v. \zv{}^\natural) : \natural \zA \to \natural \natural \zA$.
Using Proposition~\ref{prop:zero-in-natural},
for the roundtrip on $\natural \zA$  we have
\[\zv{}^\natural{}_\natural \defeq \zv \defeq v \]
and for the roundtrip on $\natural \natural \zA$ we have
\[\zc{z_\natural}{}^\natural \defeq \zz{}_\natural{}^\natural \defeq z.\]
\end{proof}

Note that the unit and counit are \emph{not} in general an equivalence between $\natural \zA$ and $\zA$ --- the use of the $\eta$-law to prove $\zv \defeq v$ does not apply in $\zA$.  Intuitively, $\natural \zA$ is the base space of $\zA$, while $\zA$ itself only zeroes out the dependence of $A$ on the ambient context.  However, we do have:

\begin{proposition}\label{prop:space-if-zero-equiv}
A type $A$ is modal iff $(\lambda x. \zx) : A \to \zA$ is an equivalence.
\end{proposition}
\begin{proof}
  Suppose $\lambda x.\zx : A \to \zA$ is an equivalence, with inverse $g
  : \zA \to A$.  We show that $\eta_A$ is a quasi-equivalence, which can
  be improved to an equivalence.  The inverse is $g \circ \varepsilon_A :
  \natural \zA \to A$.  For $x:A$, we have
  \[
  g[(\zx^\natural)_\natural] \defeq g(\zx) = x
  \]
  using the inverse law for $g \circ (\lambda x.\zx)$.

  For the other composite, let $(\lambda
  x.\zx, g, w) : A \equiv \zA$, so that $w$ is the witness that $(\lambda
  x.\zx)$ and $g$ are inverses. Now observe that $\zc {(\lambda
    x.\zx, g, w)} : \zA \equiv \zA$ and by definition of
  $\zc{\ }$, the maps are $\lambda x.\zx$ and $\zc{g}$, so
  we also have $\zc{g} \circ (\lambda x.\zx) = \id_{\zA}$.
  Then, for $y : \natural \zA$, the composite is
  \[
  \zc{g(y_\natural)}^\natural \defeq [\zg(\zy_{\natural})]^\natural
  \defeq
  ((\lambda x.\zx^\natural) \circ \zg \circ (\lambda x:\zA.\zx)) (y_\natural)
  =
  (\lambda x.\zx^\natural)(y_\natural)
  \defeq
  (\zy_\natural)^\natural
  \defeq
  y
  \]



Conversely, if $A$ is modal then $\zA$ is also modal by
Proposition~\ref{prop:underline-modal}.  By the $\beta$-law, the map
$(\lambda x.\zx) : A \to \zA$ is equal to the composite $\varepsilon_{A}
\circ \eta_A : A \to \natural \zA \to \zA$, which is the composite of
two equivalences.  First, $\eta_A$ is an equivalence because $A$ is
modal.  Second, $\varepsilon_{A}$ is an equivalence, because it is
left-inverse to $\eta_{\zA}$, which is an equivalence because ${\zA}$ is
modal, and a left-inverse of a map that is an equivalence is its
inverse.
\end{proof}

\begin{corollary}\label{cor:space-if-zero-id}
A dull type $\zA$ is modal iff $x = \zx$ for any $x : \zA$.
\end{corollary}

\begin{proof}
By function extensionality, $(\lambda x.x) = (\lambda x.\zx)$, and transporting the fact that the identity function is an equivalence along this allows us to use Proposition~\ref{prop:space-if-zero-equiv}.
\end{proof}

In the remainder of this section we show that $\natural$ has all the properties of both $\flat$ and $\sharp$ from spatial type theory.

\subsection{Monadic Properties}

We begin with the monadic properties; those shared with $\sharp$.  The
following characterises maps \emph{into} $\natural \zA$ (from the
ambient context) via an induction principle, so we call it a ``right''
universal property (even though the domain on the left of the $\Pi$ type
is what changes).  Below, we also consider a ``left'' universal
property, which characterises maps out of $\natural \zA$.

\begin{proposition}[{Right $\natural$-induction, cf.\ \cite[Theorem 3.4]{mike:real-cohesive-hott}}]\label{right-natural-induction}
Suppose $P : \natural \zA \to \univ$ is a type family such that each $P(v)$ is modal. Given a dependent function $f : \prod_{x : A} P(\zx{}^\natural)$, there is $g : \prod_{v : \natural \zA} P(v)$ such that $g(\zx{}^\natural) = f(x)$ for all $x : A$.
\end{proposition}
\begin{proof}

Because each $P(v)$ is modal, we have inverses $r_v : \natural \zP(\zv) \to P(v)$. So it is enough to produce a function $g' : \prod_{v : \natural \zA} \natural \zP(\zv)$. Zeroing $f$ gives a function $\zf : \prod_{x : \zA} \zP(\zx{}^\natural)$, which we can use to define
\begin{align*}
g'(v) :\defeq \zf(\zv{}_\natural)^\natural
\end{align*}
and $g'(v)$ has type $\natural \zP(\zv{}_\natural{}^\natural) \defeq \natural \zP(\zv)$ as required. To get the goal function $g : \prod_{v : \natural \zA} P(v)$ we then post-compose with $r_v$:
\begin{align*}
g(v) :\defeq r_v(\zf(\zv{}_\natural)^\natural)
\end{align*}
This has the correct computation property:
\begin{align*}
g(\zx{}^\natural) \defeq r_v(\zf(\zx{}^\natural{}_\natural)^\natural) \defeq r_v(\zf(\zx{})^\natural) \defeq r_v(\zc{f(x)}{}^\natural) = f(x)
\end{align*}
as $r_v$ is an inverse of $(\lambda x. \zx^\natural)$.
\end{proof}

\begin{theorem}[{Right universal property, cf.\ \cite[Theorem 3.6]{mike:real-cohesive-hott}}]\label{thm:right-univ-prop}
Suppose $B : \natural \zA \to \univ$ is a type family with each $B(v)$ modal. Then precomposition with $\eta_A : A \to \natural \zA$ is an equivalence
\begin{align*}
\prd{v : \natural \zA} B(v)  \equiv \prd{x : A} B(\zx{}^\natural)
\end{align*}
\end{theorem}
\begin{proof}
  The inverse is given by right $\natural$-induction (Proposition
  \ref{right-natural-induction}), and the roundtrip on $f : \prd{x :
    A} B(\zx{}^\natural)$ is exactly the $g(\zx^\natural) = f(x)$
  equation given above.

  For the other composite, suppose $h : \prd{v : \natural \zA} B(v)$,
  and $y:\natural \zA$, and we need to show that $h(y)$ is equal to the
  $g(y)$, for the $g$ determined by right $\natural$-induction on
  $h$-precomposed-with-$\eta_A$, $\lambda x.h(\zx^\natural)$.
  However, by the $\eta$-law, $y \defeq (\zy_\natural)^\natural$, so
  \[
  g(y) \defeq g((\zy_\natural)^\natural) = h((\zy_\natural)^\natural) = h(y)
  \]
  (This proof is morally doing another right $\natural$-induction to
  reduce $y$ to something of the form $\zx^\natural$, but we have not
  yet proved that $b =_{B(y)} b'$ is modal when $B(y)$ is modal, so we cannot
  use Proposition \ref{right-natural-induction} directly, but instead
  $\eta$-expand explicitly.)

\end{proof}

\begin{corollary}
$(\lambda A. \natural \zA) : \univ \to \univ$ with unit $\lambda x.\zx^\natural : A \to \natural \zA$ is a monadic modality in the sense of~\cite{rss:modalities}.
\end{corollary}
\begin{proof}
The precomposition equivalence
\begin{align*}
\prd{v : \natural \zA} \natural \zB(\zv) \to \prd{x : A} \natural \zB(\zx{}^\natural)
\end{align*}
of the previous proposition, where $\natural \zB$ is modal by
Proposition~\ref{prop:natural-type-is-a-space}, is precisely the
definition of a `uniquely eliminating modality'~\cite[Definition
  1.2]{rss:modalities}, one of the several equivalent definitions of a
monadic modality.
\end{proof}

Being a monadic modality has many formal consequences. In particular:
\begin{lemma}[Properties of a modality, {\cite{rss:modalities}}]\label{lem:modality-consequences}
\begin{enumerate}~
\item $A$ is modal iff $(\lambda x. \zx^\natural) : A \to \natural \zA$ admits a retraction.
\item If the input types are modal then all the following are modal:
\begin{mathpar}
1 \and A \times B \and x =_A y \and \fib_f(x) \and B \times_A C \and A \equiv B
\end{mathpar}
\item If $A$ is any type and $P : A \to \univ$ is such that every $P(x)$ is modal, then $\prd{x : A} P(x)$ is modal. If additionally $A$ is modal, then $\sm{x : A} P(x)$ is modal.
\item For any types $A$ and $B$, the canonical map $\natural(\zA \times \zB) \to \natural \zA \times \natural \zB$ is an equivalence.
\item If $A$ is a proposition, then so is $\natural \zA$.
\end{enumerate}
\end{lemma}

Also like $\sharp$, the $\natural$ modality preserves $\Sigma$-types and is left-exact.

\begin{proposition}[$\natural$ preserves $\Sigma$]\label{prop:natural-preserves-sigma}
For types $A : \univ$ and $B : A \to \univ$, we have
\begin{align*}
\natural \left( \sm{x : \zA} \zB(x) \right) \equiv \sm{u : \natural \zA} \natural \zB(\zu{}_\natural)
\end{align*}
\end{proposition}
\begin{proof}
We know that the right-hand side is modal, so to define a map from left-to-right it is sufficient to provide $f : \sm{x : A} B(x) \to \sm{u : \natural \zA} \natural \zB(\zu{}_\natural)$ and then apply $\natural$-induction. For this we have $f(a, b) :\defeq (\za{}^\natural, \zb{}^\natural)$.

The other way, we are provided $u : \natural \zA$ and $v : \natural \zB(\zu{}_\natural)$, with which we can produce \[(\zu{}_\natural, \zv{}_\natural)^\natural : \natural \left( \sm{x : \zA} \zB(x) \right).\]

To show the roundtrip on the left is the identity, suppose $v :
\natural \left( \sm{x : \zA} \zB(x) \right)$. By $\natural$-induction
($=_\natural$ is modal by Lemma~\ref{lem:modality-consequences}) we
assume this is of the form $\zc{(a, b)}^\natural$, and then the round
trip is just
\begin{align*}
(\za{}^\natural{}_\natural, \zb{}^\natural{}_\natural)^\natural \defeq (\za, \zb)^\natural \defeq \zc{(a, b)}^\natural
\end{align*}
by the $\beta$-rule for $\natural$.

For the other roundtrip, starting with $(u, v) : \sm{u : \natural \zA} \natural \zB(\zu{}_\natural)$, 
the computation rule for $\natural$-induction gives that the roundtrip is equal to
\begin{align*}
f(\zu{}_\natural, \zv{}_\natural) \defeq (\zu{}_\natural{}^\natural, \zv{}_\natural{}^\natural) \defeq (u, v)
\end{align*}


\end{proof}

\begin{proposition}[{$\natural$ is left-exact, cf.\ \cite[Theorem 3.7]{mike:real-cohesive-hott}}]\label{prop:natural-is-left-exact}
For $x, y : A$, there is an equivalence $(\zx{}^\natural = \zy{}^\natural) \equiv \natural(\zx = \zy)$ such that
\begin{mathpar}
\begin{tikzcd}[column sep = small]
& (\zx{}^\natural = \zy{}^\natural) \ar[dd] \\
(x = y) \ar[ur, "\ap_{(\zc{-}){}^\natural}"]  \ar[dr, "(\zc{-}){}^\natural" swap] \\
& \natural(\zx = \zy)
\end{tikzcd}
\end{mathpar}
commutes
\end{proposition}
\begin{proof}
We can define maps both ways by:
\begin{align*}
c \mapsto \ap_{(\zc{-}){}^\natural}(c_\natural) : \natural(\zx = \zy) \to (\zx{}^\natural = \zy{}^\natural) \\
p \mapsto (\ap_{(-)_\natural}(\zp))^\natural : (\zx{}^\natural = \zy{}^\natural) \to \natural(\zx = \zy)
\end{align*}
For showing the roundtrip on $c : \natural(\zx = \zy)$ is the identity,
first, for any type $A$ with $x,y : A$ and $p : x =_A y$, we have
$\ap_{(\lambda x.\zx)}(\zp) = \zp$ as paths in $\zx =_{\zA} \zy$ (which type checks
because $\zc{\zx} = \zx$).  By path induction, it suffices to
show
\[
\ap_{(\lambda x.\zx)}(\zc{\refl{x}})
\defeq
\ap_{(\lambda x.\zx)}(\refl{\zx})
=
\refl{\zc{x}} \defeq \zc{\refl x}
\]
Then we have
\[
[\ap_{(-)_\natural}(\zc{\ap_{(\zc{-}){}^\natural}(c_\natural)})]^\natural
\defeq
[\ap_{(-)_\natural}(\ap_{(\zc{-}){}^\natural}(\zc{c}_\natural))]^\natural
=
[\ap_{({(\zc{-}){}^\natural})_\natural}(\zc{c}_\natural)]^\natural
\defeq
[\ap_{\zc{-}}(\zc{c}_\natural)]^\natural
=
[\zc{\zc{c}_\natural}]^\natural
\defeq
[\zc{c}_\natural]^\natural
\defeq
c
\]


The other direction is easier: the roundtrip on $p$ is:
\begin{align*}
&\ap_{(\zc{-})^\natural}(\ap_{(-)_\natural}(\zp)^\natural{}_\natural) \\
&\defeq \ap_{(\zc{-})^\natural}(\ap_{(-)_\natural}(\zp)) \\
&= \ap_{(\zc{-})_\natural{}^\natural}(\zp) \\
&\defeq \ap_{\id}(\zp) \\
&= \zp\\
&= p
\end{align*}
The last equality is by Lemma~\ref{cor:space-if-zero-id} and that $p : (\zx{}^\natural = \zy{}^\natural)$ is a term of a dull modal type.

The triangle commutes by path-induction:
\begin{align*}
\ap_{(-)_\natural}(\ap_{(\zc{-}){}^\natural}(\refl{\zx}))^\natural
&= \ap_{(\zc{-})^\natural{}_\natural}(\refl{\zx}){}^\natural \\
&= (\refl{\zx}){}^\natural \\
&\defeq \zc{(\refl{x})}^\natural
\end{align*}
\end{proof}

Left-exactness has some additional formal consequences outlined by~\cite{rss:modalities}.
\begin{proposition}
\begin{itemize}
\item $\natural$ preserves pullbacks, \cite[Theorem 3.1]{rss:modalities}.
\item $\natural$ preserves $n$-types and more generally $n$-truncated maps for all $n$, \cite[Corollary 3.9]{rss:modalities}.
\end{itemize}
\end{proposition}

\subsection{Comonadic Properties}

Now we turn to the comonadic properties of $\natural$; the properties
it shares with the $\flat$ modality of spatial type theory. First, as
remarked in Section~\ref{sec:tt}, we can derive a substitution
principle of dull terms for dull variables:
\begin{definition}[Dull substitution]\label{def:dull-subst}
  The dull substitution principle is
  \begin{mathpar}
    \inferrule*[Left=dull-subst, fraction={-{\,-\,}-}]
               {\Gamma, \zx :: A \vdash c : C \and
                 \zc{\Gamma} \vdash a : A}
              {\Gamma \vdash c[a/\zx] :\defeq c[a/x] : C[a/\zx]}
  \end{mathpar}
To see that this type checks, precompose with the unit on $A$ to get
$\Gamma, x : A \vdash c : C$ and the unit on $\Gamma$ to get $\Gamma
\vdash a : A$, and then a normal substitution $c[a/x]$ has type $C[a/x]
\defeq C[a/\zx]$.
\end{definition}

\begin{remark}\label{rem:dull-subst}
  Definition~\ref{def:dull-subst} corresponds to the substitution
  principle for crisp variables in spatial type theory, where a crisp
  variable can be substituted by a term containing only crisp variables.
  In our setting, given a term $a$ in a general context, we can mark it
  and then substitute it for a dull variable:
  \begin{mathpar}
    \inferrule*[fraction={-{\,-\,}-}]
               {\Gamma, \zx :: A \vdash c : C \and
                 \Gamma \vdash a : A}
              {\Gamma \vdash c[\za/\zx] \defeq c[\za/x] : C[\za/\zx]}
  \end{mathpar}
  Because $\zc{\Gamma} \vdash \za : A$ is a dull term, we can use
  \rulen{dull-subst}.  This is equal to the ordinary substitution
  $c[a/x]$ because all of the uses of $x$ in $c$ and $C$ must be marked,
  so $a$ will be marked during substitution; we can prove inductively that
  \[
    \inferrule*[fraction={-{\,-\,}-}]
               {\Gamma, \zx :: A \vdash c : C \and
                 {\Gamma} \vdash a : A}
               {\Gamma \vdash c[\za/\zx] \defeq c[a/x] : C[\za/\zx]}
  \]
\end{remark}

\begin{proposition}[Left $\natural$-induction]\label{prop:flat-style-rules}
A $\flat$-style eliminator is derivable for $\natural$:
\begin{mathpar}
\inferrule*[Left=``$\flat$''-elim, fraction={-{\,-\,}-}]{\Gamma, x : \natural \zA \yields C : \univ \and \Gamma \yields v : \natural \zA \and \Gamma, \zu :: \zA \yields c : C[\zu{}^\natural/x]}{\Gamma \yields (\flatstyleind{v}{\zu}{c}) : C[v/x]} \\\\
\inferrule*[Left=``$\flat$''-beta, fraction={-{\,-\,}-}]{\Gamma, x : \natural \zA \yields C : \univ \and \zc{\Gamma} \yields v : \zA \and \Gamma, \zu :: \zA \yields c : C[\zu{}^\natural/x]}{\Gamma \yields (\flatstyleind{v^\natural}{\zu}{c}) \defeq c[v/\zu] : C[v^\natural/x]}
\end{mathpar}
\end{proposition}
\begin{proof}
The eliminator can be derived by dull substitution:
\begin{align*}
(\flatstyleind{v}{\zu}{c}) :\defeq c[\zc{v}{}_\natural/\zu]
\end{align*}
which has type $C[\zc{v}{}_\natural{}^\natural/x] \defeq C[v/x]$ as required. The $\beta$-rule follows immediately from the $\beta$-rule for $\natural$, using the fact that
$\zv \defeq v$ for $\zc{\Gamma} \vdash v : A$, as all variable uses in $v$ are already marked.
\end{proof}

This induction principle can be rephrased as a characterisation of maps
out of $\natural A$ (which we call a ``left'' universal property to
distinguish it from the characterisation of maps into $\natural A$
given in Theorem~\ref{thm:right-univ-prop}, though confusingly it is the
natural on the codomain that differs between the two sides).
\begin{theorem}[{Left universal property, cf.\ \cite[Theorem 6.16]{mike:real-cohesive-hott}}]\label{thm:left-univ-prop}
For any $A, B : \univ$, post-composition (and functoriality of $\natural$) with $(-)_\natural : \natural \zB \to \zB$ induces an equivalence
\begin{align*}
\natural (\natural \zA \to \natural \zB) \equiv \natural (\natural \zA \to \zB)
\end{align*}
or more dependently, for $A : \univ$ and $B : \natural \zA \to \univ$, fibrewise post-composition (and functoriality of $\natural$) with $(-)_\natural : \natural \zB(\zv) \to \zB(\zv)$ yields an equivalence
\begin{align*}
\natural \left(\prd{v:\natural \zA} \natural \zB(\zv)\right) \equiv \natural \left (\prd{v:\natural \zA} \zB(\zv)\right)
\end{align*}
\end{theorem}
\begin{proof}
The counit map $\natural \zB(\zv) \to \zB(\zv)$ is always a section of the unit map, so the post-composition map in the statement of the Theorem is also a section.

We just have to check that the roundtrip on $\natural \left (\prd{v:\natural \zA} \zB(\zv)\right)$ is the identity. Suppose we have an $f : \natural \left (\prd{v:\natural \zA} \zB(\zv)\right)$. Unfolding the definition of post-composition and functoriality of $\natural$ with both the unit and the counit is:
\begin{align*}
(\lambda x. (\zf_\natural(\zx))^\natural{}_\natural)^\natural
&\defeq (\lambda x. (\zf_\natural(\zx)))^\natural \\
&\defeq (\lambda x. (\zf_\natural(x)))^\natural \\
&\defeq (\zf_\natural)^\natural \\
&\defeq f
\end{align*}
where $\zx \defeq x : \natural \zA$ by
Proposition~\ref{prop:zero-in-natural}. Note that $(\lambda
x. (\zf_\natural(x)))^\natural$ is well-typed, as $x$ is not free below
the \rulen{$\natural$-intro}, so does not need to be marked.
\end{proof}
Because the rules for $\flat$ are derivable, we could instead have repeated the proof for $\flat$ verbatim, using $\mathsf{let}$-bindings, but the above is more direct. 

\begin{remark}
One may wonder why the applications of $\natural$ are required around the two sides. Thinking syntactically, they are necessary to block access to `spectral' information from the context. Without $\natural$ on the right, one could use a $b : \zB$ in the context to form constant functions $\const_b : \natural \zA \to \zB$ which have no corresponding maps $\natural \zA \to \natural \zB$ on the left.

Another way to see that they are necessary is to consider the model in families of pointed types. There, it is clear that the map is not an equivalence without the $\natural$ present: $\upst (\natural \zA \to \natural \zB)$ is equivalent to the point in every fibre, but $\upst (\natural \zA \to \zB)$ may be non-trivial.
\end{remark}

\begin{corollary}[Dull Self-adjointness]\label{cor:self-adjoint}
For any $A, B : \univ$ there is an equivalence
\begin{align*}
\natural (\zA \to \natural \zB) \equiv \natural (\natural \zA \to \zB)
\end{align*}
\end{corollary}
\begin{proof}
Combining the left and right universal properties we get
\begin{align*}
\natural (\zA \to \natural \zB) \equiv \natural (\natural \zA \to \natural \zB) \equiv \natural (\natural \zA \to \zB)
\end{align*}
\end{proof}

\begin{proposition}[{cf.\ \cite[Theorem 3.11]{mike:real-cohesive-hott}}]
In the presence of univalence, the type $\Modal \defeq \sm{A : \univ} \isModal(A)$ of modal types is modal.
\end{proposition}
\begin{proof}
We can use a simpler proof than the one for $\sharp$. By Corollary~\ref{cor:space-if-zero-id}, we just have to show that for any $A : \univ$ and $w : \isModal(A)$, there is an equality $(A, w) = (\zA, \zw)$. And there is: by assumption $A$ is modal so $A \equiv \zA$, again by Corollary~\ref{cor:space-if-zero-id}. In the second component, $\isModal(A)$ is a proposition, so we are done.

%
\end{proof}

\begin{remark}\label{rem:dull-induction}
  When working with inductive types in a theory with more structured contexts such as ours, one has make sure that the induction principles are strong enough. In spatial type theory, one needs to assert or prove `crisp induction principles', for when the motive of an elimination rule depends on a \emph{crisp} variable of type being eliminated.  The crisp induction principle for a type constructor is a judgemental way of saying that the modality preserves the type constructor.  For example, crisp coproduct case analysis is a judgemental way of saying that $\flat(A+B) \equiv \flat A + \flat B$, because the crisp induction principle gives crisp variables of type $A + B$ the same universal property as an ordinary variable of type $\flat A + \flat B$.  In our setting, a dull induction principle for coproducts looks like:
\begin{mathpar}
\inferrule*[left=dull-$+$-elim, fraction={-{\,-\,}-}]{\Gamma, \zz :: \zA
  + \zB \yields C : \univ \\\\ \Gamma, \zx :: \zA \yields p : C[\inl(\zx)/\zz] \and \Gamma, \zy :: \zB \yields q : C[\inr(\zy)/\zz] \\\\ \zc{\Gamma} \yields s : \zA + \zB}{\Gamma \yields \dullcase(\zz.C, \zx.p, \zy.q, s) : C[s/\zz]}
\end{mathpar}
In spatial type theory, crisp induction principles are proven using the adjointness of $\flat$ and $\sharp$, and because $\natural$ is self-adjoint, we could repeat the proof almost verbatim. But we can show they are valid more directly, using the \rulen{pre-counit} and \rulen{pre-unit} rules.

With the above inputs, we can apply \rulen{pre-unit} to obtain
\begin{align*}
\Gamma, z : \zA + \zB &\yields C : \univ \\
\Gamma, x : \zA &\yields p : C[\inl(\zx)/z] \\
\Gamma, y : \zB &\yields q : C[\inr(\zy)/z] \\
\Gamma &\yields s : \zA + \zB
\end{align*}
Now note that $C[\inl(\zx)/z] \defeq C[\inl(x)/z]$ and $C[\inr(\zx)/z]
\defeq C[\inr(x)/z]$, because $z$ is only used marked in $C$, as in
Remark~\ref{rem:dull-subst}. These inputs are now of the right shape to
apply the ordinary \rulen{$+$-elim} rule.

Using an analogous construction, we can derive dull induction principles
for $\Idsym$-types, pushouts, etc.
\end{remark}

\begin{proposition}[$\natural$ preserves pushouts]\label{prop:natural-preserves-pushouts}
Suppose $\zf : \zC \to \zA$ and $\zg : \zC \to \zB$ are dull functions between dull types. Then \[\natural(\zA +_{\zC} \zB)  \equiv \natural \zA +_{\natural \zC} \natural \zB,\] the pushout of $\natural \zf : \natural \zC \to \natural \zA$ and $\natural \zg : \natural \zC \to \natural \zB$.
\end{proposition}
\begin{proof}
Morally, this follows because we just proved that $\natural$ is a left adjoint, but we can also write out the maps explicitly as follows.

From left-to-right, we extract a term of $p : \zA +_{\zC} \zB$, and then do case analysis. On $a : \zA$, we have $\inl(\za^\natural) : \zA +_{\natural \zC} \natural \zB$. Similarly, on $b : \zB$ we have $\inr(\zb^\natural) : \zA +_{\natural \zC} \natural \zB$. To complete the cocone we have to to provide for any $c : \zC$, a path $\inl(\zf(\zc{c})^\natural) = \inr(\zg(\zc{c})^\natural)$ in $\natural \zA +_{\natural \zC} \natural \zB$. The glue constructor for the $\natural \zA +_{\natural \zC} \natural \zB$ pushout gives us a path
\begin{align*}
\glue(\zc{c}^\natural) : \inl(\natural \zf(\zc{c}^\natural)) = \inr (\natural \zg(\zc{c}^\natural))
\end{align*}
And this type is equal to $\inl(\zf(\zc{c})^\natural) = \inr (\zg(\zc{c})^\natural)$ by the definition of the functorial action of $\natural$, and the $\beta$-rule.

The right-to-left direction is similar. We begin with case analysis on $\zz$. On $n : \natural \zA$ and $m : \natural \zB$ we have $\inl(\zn_\natural) : \zA +_{\zC} \zB$ and $\inr(\zm_\natural) : \zA +_{\zC} \zB$ respectively. For any $o : \natural \zC$, we need a path $\inl(\natural \zf(\zo)_\natural) = \inr(\natural \zg(\zo)_\natural)$. Expanding the definition of functoriality and applying the $\beta$-rule, this is a path
\begin{align*}
\inl(\zf(\zo_\natural)) = \inr(\zg(\zo_\natural))
\end{align*}
and $\glue(\zo_\natural)$ is such a path in $\zA +_{\zC} \zB$. All together this produces a dull term of $\zA +_{\zC} \zB$, so applying \rulen{$\natural$-intro} we are done.

Checking the roundtrips are the identity is straightforward, using left $\natural$-induction and dull pushout induction, which can be derived as in Remark~\ref{rem:dull-induction}.

%
\end{proof}

The sequential colimit (see e.g.\ \cite{sbr:seq-colims}) of a sequence
\begin{align*}
A(0) \xrightarrow{a(0)} A(1) \xrightarrow{a(1)} A(2) \xrightarrow{a(2)} \dots
\end{align*}
is given by the higher inductive type $\colim A_n$ with point and path constructors
\begin{align*}
\iota &: \prd{n : \NN} A(n) \to \colim_n A_n \\
\kappa &: \prd{n : \NN} \prd{x : A(n)} \iota(n+1, a(n,x)) = \iota(n,x)
\end{align*}

\begin{proposition}[$\natural$ preserves sequential colimits]\label{prop:natural-preserves-seq-colimits}
Suppose we have a diagram
\begin{align*}
\zA(0) \xrightarrow{\za(0)} \zA(1) \xrightarrow{\za(1)} \zA(2) \xrightarrow{\za(2)} \dots
\end{align*}
of dull types and dull functions between them. Then \[ \natural (\colim_n \zA(n)) \equiv \colim_n \natural \zA(n) \]
where the sequential colimit on the right is over the diagram
\begin{align*}
\natural \zA(0) \xrightarrow{\natural \za(0)} \natural A(1) \xrightarrow{\natural \za(1)} \natural \zA(2) \xrightarrow{\natural \za(2)} \dots
\end{align*}
\end{proposition}
\begin{proof}
The proof is analogous to that for pushouts.
\end{proof}

Finally, we give compound rules for $\Pi x : \natural A.B$, which we
refer to as ``dull $\Pi$-types''
(c.f. \cite{mike:real-cohesive-hott,depaivaritter16fibrational,agda-flat}, or the connective $A \to B$ in
linear logic as a fusion of $! A \multimap B$).
\begin{proposition}\label{prop:dull-pis}
  Dull $\Pi$-types are definable with
  \[
   \begin{array}{c}
  \inferrule*{\zc{\Gamma} \vdash A : \univ \and
    \Gamma,\zx::A \vdash B : \univ}
             {\Gamma \vdash \prd{\zx::A} B : \univ}
  \qquad
  \inferrule*{\Gamma,\zx::A \vdash b : B}
             {\Gamma \vdash \lambda \zx::A.b : \prd{\zx::A} B : \univ}
  \qquad
  \inferrule*{\Gamma \vdash f : \prd{\zx::A} B \and
              \zc{\Gamma} \vdash a : A
             }
             {\Gamma \vdash f(a) : B[a/\zx]}
  \\ \\
  (\lambda \zx::A.b)(a) \defeq b[a/\zx]
  \qquad\qquad
  f : \prd{\zx::A} B \defeq \lambda \zx::A.f(\zx)
  \end{array}
  \]
\end{proposition}

\begin{proof}

Recall from Definition~\ref{def:dull-subst} that dull substitution
$B[a/\zx]$ is syntactically just $B[a/x]$. We define
\[
\begin{array}{l}
\prd{\zx::A} B :\defeq \prd{y:\natural A} B[\zy_\natural/\zx]\\
\lambda \zx::A. b :\defeq \lambda y.b[\zy_\natural/\zx] \\
f(a) :\defeq f (a^\natural)
\end{array}
\]
The $\beta/\eta$ equations follow from the same for $\Pi$ and
$\natural$.

\end{proof}

\subsection{Pointed Types}

The modality also interacts nicely with pointed types and functions. Recall the following standard definitions:
\begin{definition}
A \emph{pointed type} is a pair $(A, a)$ of a type $A$ and a term $a : A$. A \emph{pointed map} from $(A, a)$ to $(B, b)$ is a function $f : A \to B$ and a path $p : f(a) = b$. Write $\univ_\star$ for the type of pointed types and
\begin{align*}
A \to_\star B :\defeq \sm{f : A \to B} f(a) = b
\end{align*}
for the type of pointed maps.
\end{definition}
Note that a type or map being pointed is structure, not a property. However, it is common to abuse notation and write $A \to_\star B$ rather than $(A, a) \to_\star (B, b)$ when the points of $A$ and $B$ can be inferred from context.

\begin{lemma}\label{lem:pointed-modal}
If $B$ is modal then then $A \to_\star B$ is modal.
\end{lemma}
\begin{proof}
By Lemma~\ref{lem:modality-consequences}, modal types are closed under
$\Sigma$, $=$ and $\Pi$ where the codomain is modal.
\end{proof}


\begin{proposition}
For any pointed type $(A, a)$, there are canonical points $\za : \zA$ and $\za^\natural : \natural \zA$ for which the unit $\eta_A : A \to \natural \zA$ and counit $\varepsilon_A : \natural \zA \to \zA$ are pointed maps.
\end{proposition}
\begin{proof}
Both are immediate from the definitions:
\begin{align*}
(\lambda x. \zx^\natural)(a) &\defeq \za^\natural \\
(\lambda n. n_\natural)(\za^\natural) &\defeq \za^\natural{}_\natural \defeq \za
\end{align*}
\end{proof}

The following is a standard characterisation of equivalence between $\Sigma$-types:
\begin{lemma}\label{lem:sigma-equiv}
Suppose we have type families $P : A \to \univ$ and $Q : B \to \univ$. If we have an equivalence $f : A \equiv B$ and a family of equivalences $g_x : P(x) \equiv Q(f(x))$ then the map
\begin{align*}
(\lambda (x,p). (f(x), g_x(p))) : \sm{x : A} P(x) \to \sm{y : B} Q(y)
\end{align*}
is an equivalence.
\end{lemma}
\begin{proof}
We have $\Sigma : (\Sigma A : \univ. A \to \univ) \to \univ$, and the given data is equivalent to $(A,P) =_{(\Sigma A : \univ. A \to \univ)} (B,Q)$ by univalence and the definitions of paths in $\Sigma$ and $\Pi$ types.
\end{proof}


\begin{proposition}\label{prop:remove-nat-left-pointed}
Let $(A, a)$ and $(B, b)$ be pointed types with $B$ modal. Precomposition with $A \to_\star \natural \zA$ induces an equivalence
\begin{align*}
(\natural \zA \to_\star B) \equiv (A \to_\star B)
\end{align*}
\end{proposition}
\begin{proof}
We have seen (Theorem~\ref{thm:right-univ-prop}) that there is an equivalence
\begin{align*}
w : (\natural \zA \to B)  \equiv (A \to B).
\end{align*}
So we need to show that for every $f : \natural \zA \to B$, there is an equivalence
\begin{align*}
(f(\za^\natural) = b) \equiv (w(f)(a) = b)
\end{align*}
but $w$ is precomposition with the unit, so the type on the right is also $(f(\za^\natural) = b)$.
\end{proof}

\begin{proposition}\label{prop:remove-nat-right-pointed}
For dull pointed types $(\zA, \za)$ and $(\zB, \zb)$, post-composition with $(-)_\natural : \natural \zB \to \zB$ induces an equivalence
\begin{align*}
\natural (\natural \zA \to_\star \natural \zB) \equiv \natural (\natural \zA \to_\star \zB)
\end{align*}
\end{proposition}
\begin{proof}
Because $\natural$ commutes with $\Sigma$ by
Proposition~\ref{prop:natural-preserves-sigma}, we can again apply
Lemma~\ref{lem:sigma-equiv} so it is enough to show that the map
\begin{align*}
\sm{f : \natural (\natural \zA \to \natural \zB)} \natural (\zf_\natural(\za^\natural) = \zb^\natural) \equiv \sm {g : \natural(\natural \zA \to \zB)} \natural (\zg_\natural (\za^\natural) = \zb)
\end{align*}
induced by post-composition is an equivalence.

First, we know by Theorem~\ref{thm:left-univ-prop} that
\begin{align*}
w : \natural (\natural \zA \to \natural \zB) \equiv \natural (\natural \zA \to \zB)
\end{align*}
Explicitly, this is
\begin{align*}
w(f) \defeq (\lambda x. \zf_\natural(x)_\natural)^\natural
\end{align*}
So we need to produce for every $f : \natural (\natural \zA \to \natural \zB)$ an equivalence
\begin{align*}
\natural (\zf_\natural(\za^\natural) = \zb^\natural) \equiv \natural ((\zw(\zf))_\natural (\za^\natural) = \zb)
\end{align*}
On the right we calculate
\begin{align*}
(\zw(\zf))_\natural (\za^\natural) \defeq [(\lambda x. \zf_\natural(x)_\natural)^\natural]_\natural (\za^\natural) \defeq
(\lambda x. \zf_\natural(x)_\natural)(\za^\natural) \defeq
\zf_\natural(\za^\natural)_\natural
\end{align*}
And then we have
\begin{align*}
\natural (\zf_\natural(\za^\natural) = \zb^\natural)
&\defeq \natural ((\zf_\natural(\za^\natural))_\natural{}^\natural = \zb^\natural) && \text{(by the $\eta$-rule)} \\
&\defeq \natural ((\zw(\zf))_\natural{}(\za^\natural)^\natural  = \zb^\natural) && \text{(by the above calculation)}\\
&\equiv \natural ((\zw(\zf))_\natural (\za^\natural) = \zb) && \text{(by left exactness and idempotence)}
\end{align*}
One then has to check that this equivalence is the one that is induced by composition with the counit, but this follows from the definition of the left-exactness equivalence.
\end{proof}

\begin{proposition}[Dull Self-adjointness for Pointed Types]\label{prop:pointed-self-adjoint}
For any dull pointed types $A$ and $B$, there is an equivalence
\begin{align*}
\natural (\zA \to_\star \natural \zB) \equiv \natural (\natural \zA \to_\star \zB)
\end{align*}
\end{proposition}
\begin{proof}

  By Proposition~\ref{prop:remove-nat-left-pointed} (where $\natural
  \zB$ is modal by Proposition~\ref{prop:natural-type-is-a-space}), $(\zA
  \to_\star \natural \zB) \equiv (\natural \zA \to_\star \natural \zB)$.
  Using univalence, $(\lambda X:\univ.\natural \zX)$ preserves equivalences, so we have
  \[
  \natural (\zA \to_\star \natural \zB) \equiv \natural (\natural \zA \to_\star \natural \zB)
  \]
  Then by Proposition~\ref{prop:remove-nat-right-pointed} we have
  $\natural (\natural \zA \to_\star \natural \zB) \equiv \natural (\natural \zA \to_\star \zB)$.
\end{proof}

\section{Reduced Types}\label{sec:redu}

Informally, the modal types $A \equiv \natural \zA$ studied above are
those with \emph{no} synthetic spectral information---we think of
$\natural A$ as forgetting the spectra and replacing them with the
trivial one, so if $A$ is equivalent to $\natural \zA$, then $A$ had no spectral
information to begin with.  Dually, we can consider types with
\emph{only} synthetic spectral information, which can be defined by
demanding that its underlying space is contractible.  In our intended
model, such a type corresponds to an individual spectrum indexed by the
point --- finding the spectra among the parametrised families of spectra
as those families where the index space is trivial.  However, since the
type theory thus far admits more models than the main intended one, we
will refer to such a type as \emph{reduced} rather than as ``a
spectrum'' (though when reading it can be helpful to substitute ``a
spectrum'' for ``reduced'' to gain intuition).  This is an instance of a
general definition: for any monadic modality $\bigcirc$, a
\emph{$\bigcirc$-connected} type $A$ is one where $\bigcirc A$ is
contractible~\cite{rss:modalities}, so a reduced type is a $\natural$-connected
one.

\begin{definition}
A type $E$ is \emph{reduced} if $\natural \zE$ is contractible.  We have
\begin{align*}
\Redu &:\defeq \sm{E : \univ} \isContr(\natural \zE)
\end{align*}
for the type of reduced types.
\end{definition}

\begin{remark}
  In the pointed spaces model of Section~\ref{sec:toy-model}, the
  reduced types are ``synthetic pointed types''.  A reduced type $A$
  has $\dnst A$ contractible, so $\upst A$ is just a single
  type, and its section $\ptst A$ is just an element of $\upst A$.
  Moreover, any function $f : A \to B$ for $A$ and $B$ reduced is a
  ``synthetic pointed map''---inside the type theory, we do not need to
  carry around the data saying that $f$ preserves the point, but in the
  pointed spaces model it will be interpreted as a function that
  preserves the sections of $A$ and $B$, i.e.\ the points.
\end{remark}

\begin{remark}
With reduced types/spectra in mind, we can clarify why the syntactic
property of a type $\zE$ being dull is not the same as it being
modal. Thinking of our intended model, the only spectrum (reduced type)
that is also a space (modal) is the point.  However, we can have a
non-trivial spectrum (reduced type) that is dull, which describes the
relationship of $\zE$ to the \emph{context}---a dull spectrum is one
that only varies over the underlying space of the context.
\end{remark}

In what follows, we will often say ``For a dull \zX \ldots'', which
should be parsed formally as a dull $\Pi$ type $\prd{\zX::\ldots}
\ldots$ as in Definition~\ref{prop:dull-pis}.  In particular ``For a dull
reduced type $\zE$ means $\prd{\zE :: \Redu}{\ldots}$, though we will
omit the projection of the type part and use $\zE$ directly as a type.
This means that proof that $\zE$ is reduced is dull, as well as the type
itself.

\begin{definition}\label{def:reduced-is-pointed}
Any dull reduced type $\zE$ has a canonical point $\star_{\zE} : \zE$
given by the composite $1 \to \natural \zE \to \zE$, where the first map
is part of the data of $\natural \zE$ being contractible, and the second
is $\varepsilon_E$.
\end{definition}

\begin{proposition}\label{prop:reduced-map-is-pointed}
Any dull map $\zf : \zE \to \zF$ between dull reduced types is a pointed map in a canonical way.
\end{proposition}
\begin{proof}
In the diagram
\begin{mathpar}
\begin{tikzcd}[column sep=1em]
& 1 \ar[dl] \ar[dr] \\
\natural \zE \ar[rr, "\natural \zf"] \ar[d] && \natural \zF \ar[d] \\
\zE \ar[rr, "\zf" swap] && \zF
\end{tikzcd}
\end{mathpar}
the top triangle commutes by contractibility of $\natural \zF$, and the bottom square commutes by naturality of the counit, so we have a path equipping $\zf$ with the structure of a pointed map.
\end{proof}

We can't show that this pointing of $\zf : \zE \to \zF$ is unique, but we can show that the underlying space of `$\zf$ is pointed' is contractible, i.e.\  `$\zf$ is pointed' is itself reduced. Hence:
\begin{proposition}\label{prop:spectra-map-is-pointed}
If $\zE$ and $\zF$ are dull reduced types then $\natural(\zE \to \zF) \equiv \natural(\zE \to_\star \zF)$.
\end{proposition}
\begin{proof}
We verify
\begin{align*}
\natural(\zE \to_\star \zF)
&\defeq \natural\left (\sm{f : \zE \to \zF} f(\star_{\zE}) = \star_{\zF} \right) \\
&\equiv \sm{f : \natural(\zE \to \zF)} \natural(\zf_\natural(\star_{\zE}) = \star_{\zF} ) && \text{(Proposition~\ref{prop:natural-preserves-sigma})} \\
&\equiv \sm{f : \natural(\zE \to \zF)} \zf_\natural(\star_{\zE})^\natural =_{\natural \zF} \star_{\zF}{}^\natural && \text{(Proposition~\ref{prop:natural-is-left-exact})} \\
&\equiv \natural(\zE \to \zF) && \text{($\natural \zF$ is contractible)}
\end{align*}
\end{proof}

\begin{definition}
If $A$ is any type and $\zx : \natural \zA$ is a dull point of its base, the \emph{reduced type over $x$} is the type
\[ A_{\zx} :\defeq \sm{y : A} (\zx = \zy{}^\natural) \]
\end{definition}
When $\zA$ is dull, ${\zA}_{\zx}$ is canonically pointed by $(\zx{}_\natural, \refl{\zx}) : \sm{y : A} (\zx = \zy{}^\natural)$.

\begin{proposition}
$A_{\zx}$ is reduced.
\end{proposition}
\begin{proof}
We calculate:
\begin{align*}
\natural \left( \sm{y : A} (\zx = \zy{}^\natural) \right)
&\equiv \sm{u : \natural \zA} \natural (\zx = \zu{}_\natural{}^\natural) && \text{(Proposition~\ref{prop:natural-preserves-sigma})}\\
&\equiv \sm{u : \natural \zA} (\zx = \zu{}_\natural{}^\natural) && \text{($=_{\natural A}$ is modal by Proposition~\ref{lem:modality-consequences})}\\
&\defeq \sm{u : \natural \zA} (\zx = u)
\end{align*}
which is contractible.
\end{proof}
This lets us internalise the idea that every type is a `space-valued family of spectra':
\begin{corollary}\label{cor:sum-of-fibres}
For any type $A$,
\[ A \equiv \sm{x : \natural \zA} A_{\zx}\]
\end{corollary}
\begin{proof}
Expanding $A_{\zx}$ on the right, and using the fact that $x = \zx$:
\begin{align*}
\sm{x : \natural \zA} \sm{y : A} (\zx = \zy{}^\natural) \equiv \sm{x : \natural \zA} \sm{y : A} (x = \zy{}^\natural)
\end{align*}
Interchange the $\Sigma$-types and then you have a contractible pair.
\end{proof}

Reduced types are closed under many operations:
\begin{proposition}\label{prop:reduced-types-closed}
Reduced types are closed under $\Sigma$-types, identity types, pullbacks, pushouts, suspensions and loop spaces.
\end{proposition}
\begin{proof}
Closure under $\Sigma$, identity types and pullbacks holds for any lex
modality \cite[Theorem 3.1]{rss:modalities}. Closure under pushouts
follows from
Proposition~\ref{prop:natural-preserves-pushouts}. Suspension types and
loop spaces (defined via identity types) are special cases of pullbacks and pushouts.
\end{proof}

\begin{remark}
  Since a reduced type $E$ is a type, we can apply the ordinary higher
  inductive suspension type constructor $\Sigma \zE$ to it, and
  similarly, for an element $x:E$, form a loop space via the identity
  type $\Omega(E,x) :\defeq x=_E x$.  Thinking of the reduced types as
  spectra, these operations on types do turn out to correspond
  semantically to the correct operations of suspension and loop space
  on spectra.  This is because $\Spec$ is a full subcategory of
  $P\Spec$, and the limits and colimits defining suspension and loop
  space already land in $\Spec$, so they coincide with the limits and
  colimits calculated in the subcategory.
\end{remark}

\begin{remark}
  A common abuse of notation is, for a pointed type $(X,x)$, to write
  $\Omega X$ for the pointed type $(x =_X x, \refl{x})$, leaving $x$ to be
  inferred from context.  Any dull reduced type $\zE$ is pointed by
  $\star_{\zE}$ from Definition~\ref{def:reduced-is-pointed}, and when we
  write $\Omega \zE$ for a dull reduced type, $\star_{\zE}$ is the point
  that should be inferred.
\end{remark}

\begin{proposition}
  For dull reduced types $\zE$ and $\zF$, suspension and loop space are dull adjoints, i.e.
  \[
  \natural(\Sigma \zE \to \zF) \equiv \natural(\zE \to \Omega \zF)
  \]
\end{proposition}

\begin{proof}
  By Proposition~\ref{prop:spectra-map-is-pointed}, it is equivalent to show
  \[
  \natural(\Sigma \zE \to_\star \zF) \equiv \natural(\zE \to_\star \Omega \zF)
  \]
  so this follows by functoriality of $\natural$ from the fact that
  suspension and loop space are adjoint for (general) types and pointed
  maps~\cite[Lemma 6.5.4]{hottbook}.
\end{proof}

\subsection{Synthetic Stabilisation}

We now have defined the modal types, i.e.\ the spaces (in the intended
model in $P\Spec$), and the reduced types, i.e.\ the spectra.  One
important feature of spectra that we would like to capture synthetically
is an adjunction relating these:
\begin{mathpar}
\begin{tikzcd}[row sep = 4em, column sep = 4em]
\Redu \text{ (spectra)}\ar[d, bend left=30, "\Omega^\infty"]\\
\Modal_\star \text{ (spaces)} \ar[u, bend left=30, "\Sigma^\infty"] \ar[u, phantom, "\dashv"]
\end{tikzcd}
\end{mathpar}
The $\Sigma^\infty$ operation `freely stabilises' a pointed modal
type. Once we have imposed some axioms, we will find that the homotopy
groups of the reduced type $\Sigma^\infty X$ correspond to the stable
homotopy groups of the modal type $X$.

Recall from the introduction that the $\infty$-category of spectra may
be defined as the limit
\begin{align*}
\Spec = \varprojlim \left(  \cdots \xrightarrow{\Omega} \mathcal{S}_* \xrightarrow{\Omega} \mathcal{S}_* \xrightarrow{\Omega} \mathcal{S}_* \right)
\end{align*}
and that spectra can presented concretely as a sequence of pointed
spaces $X_0,X_1,\ldots$ with $X_0 \equiv \Omega X_1$,
$X_1 \equiv \Omega X_2$, \ldots.  Then $\Omega^\infty$ sends a
spectrum to the space $X_0$.  It is named $\Omega^\infty$ as we have
`applied $\Omega$ infinitely many times' to reach the end of the
limit.  The left adjoint $\Sigma^\infty$ sends a space $X$ to the
spectrification of its suspension pre-spectrum (see
Section~\ref{sec:norm} below), i.e.\ we can make a sequence of pointed
spaces $X, \Sigma X, \Sigma \Sigma X, \ldots$ with maps
$X \to_\star \Omega \Sigma X$,
$\Sigma X \to_\star \Omega \Sigma \Sigma X$ \ldots given by the unit
of the $\Sigma \dashv \Omega$ adjunction. These maps are not
equivalences, which is corrected by replacing each space by a certain
colimit, a process called spectrification.

Here, we instead define an abstract/synthetic variant of this
adjunction, which can be interpreted in this way in $P\Spec$.  We will
do so by assuming the existence of a distinguished reduced type $\sphere
: \Redu$. Note that $\sphere$ is a closed type, so is dull. In our
intended model this will be interpreted as the sphere spectrum
(i.e.\ $\Sigma^\infty S^0$) living over the point.  In this section, we
will use the notation $\sphere$ and $\Sigma^\infty/\Omega^\infty$
because of the intended interpretation in $P\Spec$, but these will be
abstract operations that exist in any model (up to the choice of
$\sphere$).

Using $\sphere$, we define $\Omega^\infty$ as follows:
\begin{definition}
  For $\zE : \Redu$, define the modal type
  \begin{align*}
    \Omega^\infty \zE :\defeq \natural(\sphere \to \zE)
  \end{align*}
  which is pointed by the constant zero map
  $(\lambda \_. \star_{\zE})^\natural$.
\end{definition}

\begin{remark}
  To gain intuition for why this is the right definition, we can
  consider the pointed spaces model of Section~\ref{sec:toy-model}. The analogue of $\sphere$ there
  is $\mathbb{B}$, the two-point space $S^0$ living over the
  point. The base of the function type $\mathbb{B} \to \zE$ is
  equivalent to the type of all basepoint preserving maps from the
  upstairs of $\mathbb{B}$ to the upstairs of $\zE$. The basepoint is
  fixed, and there is one free point that can be mapped to any point
  of the upstairs $\zE$. So the base space of $\mathbb{B} \to \zE$
  indeed corresponds to the upstairs of $\zE$.

  For spectra, similar reasoning applies, using the
  $\Sigma^\infty \dashv \Omega^\infty$ adjunction more explicitly, and
  that the sphere spectrum is the stabilisation of the two-point space:
  \begin{align*}
    \mathrm{Map}_{\Spec}(\sphere, E) \equiv \mathrm{Map}_{\Spec}(\Sigma^\infty S^0,  E) \equiv \mathrm{Map}_{\mathcal{S}_\star}(S^0, \Omega^\infty E) \equiv \mathrm{Map}_{\mathcal{S}}(1, \Omega^\infty E) \equiv \Omega^\infty E
  \end{align*}
  So we take `$\mathrm{Map}_{\Spec}(\sphere, E)$' as our definition of $\Omega^\infty E$.
\end{remark}

One application of $\Omega^\infty$ is defining the homotopy groups of a spectrum, in the sense that is used in stable homotopy theory:
\begin{definition}\label{def:stable-homotopy-groups}
The \emph{$\sphere$-shifted homotopy groups} of a reduced type $\zE$ are defined by
\begin{align*}
\pi^s_n \zE :\defeq \pi_n(\Omega^\infty \zE)
\end{align*}
\end{definition}
Simply calculating $\pi_n(\zE)$ will not do the correct thing: we will
see later that, after imposing a stability axiom
(Section~\ref{sec:stab}), reduced types are
$\infty$-connected~\ref{cor:spec-inf-connected}, so $\pi_n(\zE) \equiv
1$ always.

The left adjoint $\Sigma^\infty \dashv \Omega^\infty$ then has a surprisingly simple formula:
\begin{definition}
For $X$ a pointed modal type, define $\Sigma^\infty X :\defeq X \topsmash \sphere$.
\end{definition}
This is left adjoint roughly because maps of reduced types are pointed
maps, and $- \topsmash A \dashv A \to_\star -$ for any $A$.

First, we check that $\Sigma^\infty$ lands in reduced types:
\begin{proposition}\label{prop:sigmainf-reduced}
For any modal type $X$, $\Sigma^\infty X$ is reduced.
\end{proposition}
\begin{proof}
The pushout diagrams defining $X \vee \sphere$ and $X \topsmash \sphere$
are
\begin{mathpar}
\begin{tikzcd}
1 \ar[d] \ar[r] \arrow[dr, phantom, "\ulcorner", very near end] & X \ar[d] \\
\sphere \ar[r] & X \vee \sphere
\end{tikzcd}

\begin{tikzcd}
X \vee \sphere \ar[d] \ar[r] \arrow[dr, phantom, "\ulcorner", very near end] & X \times \sphere \ar[d] \\
1 \ar[r] & \Sigma^\infty X
\end{tikzcd}
\end{mathpar}
Because $\natural$ preserves pushouts by
Proposition~\ref{prop:natural-preserves-pushouts} and products
by Proposition~\ref{prop:natural-preserves-sigma}, and $\sphere$ is
reduced, so $\natural \sphere \equiv 1$, we can calculate

\begin{mathpar}
  \begin{tikzcd}
1 \ar[d] \ar[r] \arrow[dr, phantom, "\ulcorner", very near end] & \natural \zX \ar[d] \\
1 \ar[r] & \natural (\zX \vee \sphere) \equiv \natural \zX
\end{tikzcd}

\begin{tikzcd}
\natural \zX \ar[d] \ar[r] \arrow[dr, phantom, "\ulcorner", very near end] & \natural \zX \ar[d] \\
1 \ar[r] & \natural(\Sigma^\infty \zX) \equiv 1
\end{tikzcd}
\end{mathpar}
because the top of the second diagram is the identity, and the pushout
of a map along the identity is the same map.

\end{proof}

\begin{proposition}\label{prop:sigmainf-omegainf-adjoint}
$\Sigma^\infty$ and $\Omega^\infty$ are (dull) adjoints: there is an equivalence \[\natural (\Sigma^\infty \zX \to \zE) \equiv (\zX \to_\star \Omega^\infty \zE)\] where $\zX$ is a pointed modal type and $\zE$ is reduced.
\end{proposition}
\begin{proof}
\begin{align*}
        \zX \to_\star \Omega^\infty \zE
&\defeq (\zX  \to_\star \natural(\sphere \to_\star \zE)) \\
&\equiv \natural(\zX \to_\star \natural(\sphere \to_\star \zE)) &&(- \to_\star \natural - \text{ is modal })\\
&\defeq \natural(\natural \zX  \to_\star (\sphere \to_\star \zE)) && (\text{Proposition~\ref{prop:pointed-self-adjoint}})\\
&\equiv \natural(\zX  \to_\star (\sphere \to_\star \zE)) && \text{($\zX$ modal)} \\
&\equiv \natural(\zX  \topsmash \sphere \to_\star \zE) \\
&\defeq \natural(\Sigma^\infty \zX  \to_\star \zE) \\
&\equiv \natural(\Sigma^\infty \zX  \to \zE) && (\text{Proposition~\ref{prop:spectra-map-is-pointed}, Proposition~\ref{prop:sigmainf-reduced}})
\end{align*}
Here we use currying for pointed maps $A \to_\star (B \to_\star C)
\equiv (A \topsmash B) \to_\star C$, which has been proved in type theory~\cite[Theorem 4.3.28]{floris:thesis}.
\end{proof}

\begin{remark}
  In the pointed spaces model, for a pointed modal type $X$,
  $X \topsmash \mathbb{B}$ works out to be $\dnst \zX \topsmash 1$ in the
  base (which is indeed contractible). Over this point, we are
  calculating the cofibre of the map
  \[ X + 1 \equiv X \wedge S^0 \to X \times S^0 \equiv X + X \] which
  is the identity on the first component $X$ and the basepoint
  inclusion on the second. So the first copy of $X$ is crushed to a
  point, and identified with the basepoint of the second copy of $X$.

  In all, the operation takes a pointed modal type $X$ to a reduced
  type $\Sigma^\infty X$, moving $X$ from the base to the fibre,
  internalising the analytic pointing as synthetic pointing.

  In $P\Spec$ we have to work a little harder to justify this
  definition. $\Sigma^\infty$ is a functor from pointed spaces to
  spectra, but we can precompose with the functor
  $(-)_+ : \mathcal{S} \to \mathcal{S}_*$ to get a functor from
  unpointed spaces. This is typically written
  $\Sigma^\infty_+ : \mathcal{S} \to \Spec$, and is left adjoint to
  the functor given by computing the pointed space $\Omega^\infty$ and
  forgetting the basepoint.

  As a left adjoint, $\Sigma^\infty_+$ preserves colimits, and because
  every space is the colimit of its points, we calculate
  \[ \Sigma^\infty_+(X) \equiv \Sigma^\infty_+(\colim_X 1) \equiv
    \colim_X \Sigma^\infty_+(1) \equiv \colim_X \Sigma^\infty(S^0)
    \equiv \colim_X \sphere \]

  I.e., the colimit of the constant diagram on $X$ at $\sphere$. For
  us, such a constant diagram is given by the parametrised spectrum
  $X \times \sphere$, placing a copy of $\sphere$ over every point in
  $X$. The desired colimit can be computed as the cofibre of
  $X \to X \times \sphere$ which is the basepoint in the second
  component, crushing the base $X$ to a point. We can take this as our
  type-theoretic definition of $\Sigma^\infty_+$:
  \[ \Sigma^\infty_+(X) :\defeq \cofib(X \to X \times \sphere) \]

  To get back to $\Sigma^\infty(X)$ for $X$ a pointed type, we have to
  crush the `extra copy of $\sphere$' that was added over the new
  basepoint: this is now taking the cofibre of the composite
  $\sphere \to X \times \sphere \to \Sigma^\infty_+ X$ which is the
  basepoint in the \emph{first} component. Combining the two cofibre
  diagrams, in all we have calculated the smash product
  $X \topsmash \sphere$.
\end{remark}

\subsection{Commutativity of the Adjunctions}

So far, we have seen that we have the three adjunctions in the following
diagram (the two vertical sides are the same):


\begin{mathpar}
\begin{tikzcd}[row sep = 4em, column sep = 4em]
\Modal_\star \ar[r, bend left=15, "\Sigma"] \ar[d, bend left=20, "\Sigma^\infty"] \ar[r, phantom, "\bot"] \ar[d, phantom, "\vdash"]  & \Modal_\star \ar[l,bend left=15,"\Omega"] \ar[d,bend left=20, "\Sigma^\infty"] \ar[d, phantom, "\vdash"] \\
\Redu \ar[r, bend left=15, "\Sigma"] \ar[u, bend left=20, "\Omega^\infty"] \ar[r, phantom, "\bot"] & \Redu \ar[l,bend left=15, "\Omega"]\ar[u,bend left=20, "\Omega^\infty"]
\end{tikzcd}
\end{mathpar}

We can also show that the diagram commutes.

\begin{remark}
  In pointed spaces it is clear this should be true, as the
  $\Sigma^\infty \dashv \Omega^\infty$ adjunction simply moves pointed
  spaces into the fibre over a point and back, and the suspension/loop
  space of a reduced type is calculating the suspension/loop space of
  the unique fibre.

  For spectra, consider a spectrum again presented as a sequence of
  pointed spaces $(E_0, E_1, \dots)$. The loop space of such an
  $\Omega$-spectrum can be calculated by shifting the spaces over by
  one, giving $\Omega$-spectrum $(\Omega E_0, E_0, E_1, \dots)$, so
  extracting the $0$th space commutes with calculating the loop space.

  The suspension of an $\Omega$-spectrum can be calculated by shifting
  the spaces the other way, giving $(E_1, E_2, \dots)$.  Leaving aside
  spectrification briefly, $\Sigma^\infty X$ is given by the
  prespectrum $(X, \Sigma X, \Sigma^2 X, \dots)$, so
  $\Sigma^\infty \Sigma X \defeq (\Sigma X, \Sigma^2 X, \dots)$ is
  exactly $\Sigma^\infty X$ shifted by one.
\end{remark}

\begin{proposition}\label{prop:omegainf-omega-commutes}
$\Omega^\infty$ commutes with the ordinary loop space operation $\Omega$:
\begin{align*}
\Omega^\infty \Omega \zE \equiv_\star \Omega \Omega^\infty \zE
\end{align*}
naturally in $\zE$.
\end{proposition}
Note that the left-hand side is the loop space on reduced (and therefore
pointed) types, while the right-hand side is the loop space on pointed
modal types, but both are implemented by the usual loop space on types.

This is easy, after the following basic fact about loop spaces.
\begin{lemma}
For any pointed types $A$ and $B$,
\[(A \to_\star \Omega B) \equiv_\star \Omega(A \to_\star B) \]
naturally in $B$.
\end{lemma}
\begin{proof}
\begin{align*}
(A \to_\star \Omega B) &\equiv_\star (A \to_\star (S^1 \to_\star B)) \\
&\equiv_\star (S^1 \to_\star (A \to_\star B)) \\
&\equiv_\star \Omega(A \to_\star B)
\end{align*}
The first equivalence is essentially the universal property of the
higher inductive circle $S^1$, while the second follows from exchange
for function types.
\end{proof}

\begin{proof}[Proof of Proposition]
\begin{align*}
\Omega^\infty \Omega \zE \defeq \natural(\sphere \to_\star \Omega \zE) \equiv_\star \natural (\Omega(\sphere \to_\star \zE)) \equiv_\star \Omega\natural (\sphere \to_\star \zE) \defeq  \Omega \Omega^\infty \zE
\end{align*}
For any pointed type, $\natural \Omega(A,a) \equiv_\star \Omega(\natural A,\za^\natural)$ follows from
Proposition~\ref{prop:natural-is-left-exact}.
\end{proof}

In the other direction, we have:
\begin{proposition}\label{prop:sigma-infty-susp-commute}
$\Sigma^\infty$ commutes with $\Sigma$:
\begin{align*}
\Sigma \Sigma^\infty X \equiv \Sigma^\infty \Sigma X
\end{align*}
naturally in $X$.
\end{proposition}
\begin{proof}
  First, some properties of the suspension and smash higher inductive
  types in ordinary homotopy type theory are that
  $\Sigma X \equiv S^1 \topsmash X$~\cite[Proposition 4.2.1]{brunerie:thesis} and smash is
  associative~\cite[Definition 4.3.33]{floris:thesis}.  Thus, we can
  calculate
  \[\Sigma \Sigma^\infty X \defeq \Sigma(X \topsmash \sphere) \equiv S^1 \topsmash (X \topsmash \sphere) \equiv (S^1 \topsmash X) \topsmash \sphere \equiv \Sigma^\infty \Sigma X \]
\end{proof}

\section{Stability}\label{sec:stab}

Classically, the category of spectra has a number of special
properties, including a zero object, biproducts (products and
coproducts that are isomorphic), any pushout square is a pullback
square and vice versa, and suspension and loop space are inverse (not
only adjoint).  Thus far, our definition of synthetic spectra as
reduced types (dully) has a zero object: $1$ is initial as well as
terminal because $\natural(1 \to \zE) \equiv \natural(\zE) \equiv 1$.
To establish the other properties, it turns out to suffice to add an
apparently weaker axiom asserting that products and coproducts in
$\Redu$ coincide---we will show that this implies stability in the
sense of pullback and pushout squares coinciding, which in turn
implies that suspension and loop space are an equivalence.


Reduced types are pointed, so the coproduct in the category of reduced
types is actually the wedge $\vee$.  For any pointed types $A$ and $B$
there is a canonical wedge inclusion
$\iota_{A,B} : A \wedge B \to A \times B$.

\begin{named}{Axiom S (for Stable)}
For any dull reduced types $\zE$ and $\zF$, the wedge inclusion $\iota_{\zE,\zF} : \zE \wedge \zF \to \zE \times \zF$ is an equivalence.
\end{named}

Unfolding the informal use of dull $\Pi$-types,
(Proposition~\ref{prop:dull-pis}), this is asserting the existence of
a term
\begin{align*}
\mathsf{ax}_S : \prd{E : \natural \Redu} \prd{F : \natural \Redu} \isEquiv(\iota_{E_\natural, F_\natural})
\end{align*}

\begin{remark}
Some care has to be taken when devising an internal version of the external fact that the subcategory $\Spec \hookrightarrow P\Spec$ is stable. The obvious thing to try is an an axiom that applies to any reduced type. The issue is that when asserting the existence of a closed term, like the axiom above, the term can be weakened to any ambient context.

Semantically, this means we have to consider whether the axiom holds in all slice categories $P\Spec/\Gamma$. In our situation, it does \emph{not}. A reduced type in $P\Spec/\Gamma$ consists of a family of spectra $E$ over the base of $\Gamma$, together with spectrum maps between the fibres of $E$ and the corresponding fibres of $\Gamma$. This category is not stable in general, as it lacks a zero object (among other things). If the context $\Gamma$ itself consists of a single spectrum $F$, then the existence of a zero object would imply that any map $E \to F$ splits.

The axiom \emph{does} hold in $P\Spec/X$ when $X$ is a space, as reduced types in $P\Spec/X$ are exactly the families of spectra over $X$, comprising a stable category. Syntactically this requirement corresponds to having a dull context, hence the restriction to types in $\natural \Redu$ in Axiom S.

We can weaken our $\mathsf{ax}_S$ to an arbitrary context $\Gamma$ and still be safe, as the spectra $E_\natural$ and $F_\natural$ will only depend on the underlying space of $\Gamma$.
\end{remark}

\begin{remark}
This axiom rules out the pointed spaces model, as the wedge and product of reduced types in that model correspond to the ordinary wedge and product of the pointed types in the unique fibre. But it \emph{doesn't} rule out the trivial model where $\natural$ is the identity functor. There, the only reduced type is the point, and the wedge inclusion $1 \wedge 1 \to 1 \times 1$ is certainly an equivalence.
\end{remark}

\begin{proposition}\label{prop:reduced-smash-join-contractible}
For two dull reduced types $\zE$ and $\zF$, the smash product $\zE
\topsmash \zF$ and the join $\zE \join \zF$ are both contractible.
\end{proposition}
\begin{proof}
Recall that the smash product $\zE \topsmash \zF$ is the cofibre of the wedge inclusion, i.e.\ the pushout
\[
\begin{tikzcd}
\zE \vee \zF \ar[d] \ar[r] \arrow[dr, phantom, "\ulcorner", very near end] & \zE \times \zF \ar[d] \\
1 \ar[r] & \zE \topsmash \zF
\end{tikzcd}
\]
Axiom S asserts that the top map is an equivalence, so the
bottom-right corner is equivalent to the bottom-left.  For the join,
we have $\zE \join \zF \equiv \Sigma(\zE \topsmash \zF)$
by~\cite[Theorem 4.19]{cavallo:cohomology}, and the suspension of a
contractible type is contractible.
\end{proof}
Here, $\zE \topsmash \zF$ is the smash product \emph{of types}, which
does not correspond to the `smash product of spectra', an important
operation in stable homotopy theory.

Next, we show that Axiom S indeed makes pullbacks and pushouts in
$\Redu$ coincide, mainly as a consequence of the Little Blakers-Massey
Theorem~\cite{abfj:blakers-massey}.  As a first step:

\begin{lemma}\label{lem:suspension-loop-pullback}
For any pointed type $A$, there is a pullback square
\begin{mathpar}
\begin{tikzcd}
\Sigma \Omega A \ar[d] \ar[r] & A \wedge A \ar[d, "\iota"] \\
A \ar[r, "\Delta" swap] & A \times A
\end{tikzcd}
\end{mathpar}
where the map on the left is the counit of the $\Sigma \dashv \Omega$ adjunction.
\end{lemma}

This will follow quickly from the following consequence of descent for pushouts:
\begin{theorem}[{\cite[Theorem 2.2.12]{egbert:thesis}}]\label{thm:descent-pushouts}
Consider a commuting cube of types
\begin{mathpar}
\begin{tikzcd}
& S' \arrow[dl] \arrow[dr] \arrow[d] \\
A' \arrow[d] & S \arrow[dl] \arrow[dr] & B' \arrow[dl,crossing over] \arrow[d] \\
A \arrow[dr] & X' \arrow[d] \arrow[from=ul,crossing over] & B \arrow[dl] \\
& X,
\end{tikzcd}
\end{mathpar}
and suppose the vertical squares are pullback squares. Then the commuting square
\begin{mathpar}
\begin{tikzcd}
A' \sqcup_{S'} B' \arrow[r] \arrow[d] & X' \arrow[d] \\
A\sqcup_{S} B \arrow[r] & X
\end{tikzcd}
\end{mathpar}
is a pullback square.
\end{theorem}

\begin{proof}[Proof of Lemma~\ref{lem:suspension-loop-pullback}]
Consider the commutative cube
\begin{mathpar}
\begin{tikzcd}
& \Omega A \arrow[dl] \arrow[dr] \arrow[d] \\
1 \arrow[d] & 1\arrow[dl] \arrow[dr] & 1 \arrow[dl,crossing over] \arrow[d] \\
A \arrow[dr] & A \arrow[d] \arrow[from=ul,crossing over] & A \arrow[dl] \\
& A \times A,
\end{tikzcd}
\end{mathpar}
where the vertical map $A \to A \times A$ is the diagonal, and the maps $A \to A \times A$ on the left and right are the identity on one component and constant at the point on the other. All of the vertical squares are pullbacks.

Now note that the pushout of the top span is $\Sigma \Omega A$, and the
pushout of the bottom span is $A \wedge A$, so by the above theorem we
have the desired pullback square (with the top/left bottom/right maps
swapped, but pullbacks are symmetric in the sense that $A \times_C B
\equiv B \times_C A$).
\end{proof}

\begin{corollarystab}\label{cor:stab-susp-loop-equiv}
For any dull reduced type, the canonical map $\Sigma \Omega \zE \to \zE$ is an equivalence.
\end{corollarystab}
\begin{proof}
By Lemma~\ref{lem:suspension-loop-pullback}, this map is the pullback
of the wedge inclusion $\iota$ along $\Delta$.  By Axiom S, $\iota$ is
an equivalence, and the pullback of an equivalence along any map is an
equivalence.
\end{proof}

\begin{definition}
Recall~\cite[Definition 7.5.1]{hottbook} that a type $A$ is $n$-connected if its
$n$-truncation is contractible, and a map is $n$-connected if its fibre is
an $n$-connected type for all base points.  A type or function is
$\infty$-connected if it is $n$-connected for every $n$.
\end{definition}

We will use the following Lemmas about $n$-connected types:

\begin{lemma}\label{lem:fibres-unique}
Suppose that a type $A$ is 0-connected and $P : A \to \mathsf{Prop}$ is a family
of propositions ($-1$-types).  Then if $P(a)$ holds for some $a:A$, then
$P(a')$ holds for all $a'$.
\end{lemma}

\begin{proof}
  Assume an $a$ such that $P(a)$ and another point $a'$.  Since $A$ is
  0-connected, its 0-truncation is contractible, and therefore its
  0-truncation is a proposition~\cite[Theorem 7.1.10]{hottbook}, so we
  get a path $|a'| =_{\trunc{0}{A}} |a|$. Commuting the truncation
  with the loop space gives $\trunc{-1}{a' = a}$~\cite[Theorem
  7.3.12]{hottbook}.  That is, $a'$ is merely equal to $a$.  But
  $P(a')$ is a proposition by assumption, so to prove it, we can
  assume $a' = a$, and then transport the assumed proof of $P(a)$.
\end{proof}

\begin{lemma}\label{lem:connected-closed}
  \begin{itemize}
  \item If $A$ and $B$ are $0$-connected then so is $A \times B$
  \item If $A,B$ are $0$-connected and $C$ is $1$-connected, then for
    any maps $f : A \to C$ and $g : B \to C$, the pullback $A \times_C
    B$ is $0$-connected.
  \end{itemize}
\end{lemma}

\begin{proof}
  For the first part, truncation preserves products~\cite[Theorem
  7.3.8]{hottbook}, so to show $\trunc{0}{A \times B}$ is
  contractible, we can equivalently show that
  $\trunc{0}{A} \times \trunc{0}{B}$ is contractible.  But
  $\trunc{0}{A}$ and $\trunc{0}{B}$ are contractible by assumption,
  and $1 \times 1 \equiv 1$.

  For the second, the pullback is given by the type
  $\sm{(x,y):A \times B} f(x) =_{C} g(y)$.
  By \cite[Theorem 7.3.9, Theorem 7.3.12]{hottbook}, we have
\begin{align*}
& \trunc{0}{\sm{(x,y):A \times B} f(x) =_{C} g(y)} \\
& \equiv \trunc{0}{ \sm{(x,y):A \times B} \trunc{0}{f(x) =_{C} g(y)}   }\\
& \equiv \trunc{0}{ \sm{(x,y):A \times B} |f(x)| =_{\trunc{1}{C}} |g(y)| }\\
& \equiv \trunc{0}{A \times B} \\
& \equiv 1
\end{align*}
The second-to-last step is because $\trunc{1}{C}$ is contractible by assumption, so any identity type in it is as well,
and the last step is by the previous part, since $A$ and $B$ are 0-connected.
\end{proof}

\begin{corollarystab}\label{cor:spec-inf-connected}
Dull reduced types and dull maps between them are $\infty$-connected.
\end{corollarystab}
\begin{proof}
For types, we prove by induction on $n$ that every dull reduced type
$\zE$ is $n$-connected.  Every type is $(-2)$-connected, since the
$(-2)$-truncation is contractible by definition.  For the inductive step,
suppose $\zE$ is reduced, and we want to show that it is
$(n+1)$-connected.  Then $\Omega \zE$ is also reduced by
Proposition~\ref{prop:reduced-types-closed} and dull, so by the
inductive hypothesis (which applies to all dull reduced types, so in
particular $\Omega \zE$) it is $n$-connected.  Suspension increases
connectivity by 1~\cite[Theorem 8.2.1]{hottbook}, so $\Sigma \Omega \zE$ is
$(n+1)$-connected.  But by Corollary~\ref{cor:stab-susp-loop-equiv},
$\zE \equiv \Sigma \Omega \zE$, so $\zE$ is $(n+1)$-connected as well.

Now for maps, fix an $n$.  The fibre of a dull map $\zf : \zE \to \zF$
over the basepoint is the type $\fib_{\zf}(\star_{\zF}) := \sm{x : \zE}
\zf(x) = \star_{\zF}$, which is a dull reduced type and thus is an
$\infty$-connected type by above, and therefore $n$-connected.  We now
show that this implies that all fibres are $n$-connected using
Lemma~\ref{lem:fibres-unique}.  First, $\zF$ is a dull reduced type, and
thus by the previous part it is in particular 0-connected.  For any type
$A$, the type ``$A$ is $n$-connected'' is a proposition, because it
unfolds to ``the $n$-truncation of $A$ is contractible'', and being
contractible (like all h-levels) is a proposition~\cite[Theorem
  7.1.7]{hottbook}.  Thus, $\fib{\zf}(\star_{\zF})$ being $n$-connected
implies the same for any basepoint $x : \zF$, so $\zf$ is an
$n$-connected map.
\end{proof}


We will now make use of the Little Blakers-Massey Theorem and its dual, which is the specialisation of the Generalised Blakers-Massey Theorem~\cite{abfj:blakers-massey} to the identity modality. The Generalised Theorem has been formalised in The HoTT Library~\cite{hott:blakers-massey-formalised}. First, some notation:
\begin{definition}
For $f : A \to B$, let $\Delta f$ denote the canonical map $A \to A \times_B A$.
\end{definition}
\begin{definition}
For $f : A \to B$ and $g : C \to D$, the pushout product $f \pushpr g$ is defined to be the canonical gap map
\begin{mathpar}
\begin{tikzcd}
A \times C \ar[r] \ar[d] \ar[dr, phantom, "\ulcorner", very near end] & A \times D\ar[d] \ar[ddr, bend left] \\
B \times C \ar[r] \ar[drr, bend right] & P \ar[dr, dotted, "f \pushpr g"] \\
&& B \times D
\end{tikzcd}
\end{mathpar}
where $P :\defeq (A \times D) \sqcup_{A \times C} (B \times C)$.
\end{definition}

\begin{lemma}\label{lemma:diag-fibre}
The fibres of $\Delta f$ are given by
\[ \fib_{\Delta f}(a, a', p) \equiv ((a, p) =_{\fib_f(f(a'))} (a', \refl{f(a)})),\] and in particular, \[\fib_{\Delta f}(a, a, \refl{a}) \equiv \Omega_{(a, \refl{a})} \fib_f(f(a)). \]
\end{lemma}
\begin{proof}
Direct calculation:
\begin{align*}
\fib_{\Delta f}(a, a', p)
&:\defeq \sm{x : A} (x, x, \refl{f(x)}) = (a, a', p) \\
&\equiv \sm{x : A} \sm{l : x = a} \sm{r : x = a'} !\ap_f(l) \cdot \ap_f(r) = p \\
&\equiv \sm{r : a = a'} \ap_f(r) = p \\
&\equiv \sm{r : a = a'} !\ap_f(r) \cdot p = \refl{f(a')}\\
&\equiv \sm{r : a = a'} r_*(p) = \refl{f(a')}\\
&\equiv (a, p) = (a', \refl{f(a')})
\end{align*}
\end{proof}

\begin{proposition}\label{prop:pushpr-fibrewise-join}
The pushout product is the `external fibrewise join', in the sense that
for $b : B$ and $d : D$,
\[ \fib_{f \pushpr g}(b, d) \equiv \fib_f(b) \join \fib_g(d) \]
\end{proposition}
\begin{proof}
This is another application of descent. Consider the cube
\begin{mathpar}
\begin{tikzcd}
& \fib_f(b) \times \fib_g(d) \arrow[dl] \arrow[dr] \arrow[d] \\
\fib_f(b) \arrow[d] & A \times C \arrow[dl] \arrow[dr] & \fib_g(d) \arrow[dl,crossing over] \arrow[d] \\
A \times D \arrow[dr] & 1 \arrow[d] \arrow[from=ul,crossing over] & B \times C \arrow[dl] \\
& B \times D
\end{tikzcd}
\end{mathpar}
The back vertical map extracts the $A$ and $C$ from the fibres, the side
vertical maps extract the $A$ or $C$ and pair with $b$ or $d$ in the
other component, and the front vertical map is $(b,d)$.  By singleton
contractibility and paths in products being component-wise, all the
vertical sides are pullback squares. The corresponding square
\begin{mathpar}
\begin{tikzcd}
\fib_f(b) \join \fib_g(d) \arrow[r] \arrow[d] & 1 \arrow[d] \\
P \arrow[r, "f \pushpr g" swap] & B \times D
\end{tikzcd}
\end{mathpar}
is a pullback by Theorem~\ref{thm:descent-pushouts}, so
$\fib_f(b) \join \fib_g(d)$ is equivalent to the fibre of
$f \pushpr g$ at $(b, d)$.
\end{proof}

We now quote the Theorem that actually does the work:
\begin{theorem}[Little Blakers-Massey Theorem, {\cite[Corollary 4.1.4, Theorem 3.5.1]{abfj:blakers-massey}}]\label{thm:lbmt}
Consider the following square.
\begin{mathpar}
\begin{tikzcd}
Z \ar[r, "g"] \ar[d, "f" swap] & Y \ar[d, "h"] \\
X \ar[r, "k" swap] & W
\end{tikzcd}
\end{mathpar}
\begin{itemize}
\item If the square is a pushout and $\Delta f \pushpr \Delta g$ is an equivalence, then the square is also a pullback.
\item If the square is a pullback and $h \pushpr k$ is an equivalence, then the square is also a pushout.
\end{itemize}
\end{theorem}



\begin{theoremstab}\label{thm:reduced-pullback-iff-pushout}
A dull commutative square in reduced types is a pullback square iff it is a pushout square.
\end{theoremstab}
\begin{proof}
  Suppose a dull commutative square in reduced types
  \begin{mathpar}
\begin{tikzcd}
\zE\ar[r, "\zg"] \ar[d, "\zf" swap] & \zG \ar[d, "\zh"] \\
\zF \ar[r, "\zk" swap] & \zH
\end{tikzcd}
\end{mathpar}
By Definition~\ref{def:reduced-is-pointed} and
Proposition~\ref{prop:reduced-map-is-pointed}, we have base points
$\star_{\zE}$, $\star_{\zF}$, $\star_{\zG}$, $\star_{\zH}$, and
the maps are all pointed. Thus, the fibres all have dull points---
e.g.\ $\fib_{\zf}(\star_{\zF}) := \sm{z:\zE} \zf(z) = \star_{\zF}$ has a
  point given by $\star_{\zE}$ and the path showing $\zf$ is pointed.

To use Theorem~\ref{thm:lbmt}, we just need to show that $\Delta \zf
\pushpr \Delta \zg$ and $\zh \pushpr \zk$ are equivalences, which we
will do using the ``contractible fibres'' definition of equivalence.

The fibre of $\Delta \zf \pushpr \Delta \zg$ over the basepoint $p_0
:\defeq ( (\star_{\zE},\star_{\zE},\refl{\zf(\star_{\zE})}),
(\star_{\zE},\star_{\zE},\refl{\zg(\star_{\zE})}))$ is $\Omega
(\fib_{\zf}(\star_{\zF}))\join \Omega(\fib_{\zg}(\star_{\zG}))$ by
Lemma~\ref{lemma:diag-fibre} (where $f(\star_{\zE}) = \star_{\zF}$
and $g(\star_{\zE}) = \star_{\zG}$ because the maps are pointed) and
Proposition~\ref{prop:pushpr-fibrewise-join}.  Again by
Proposition~\ref{prop:pushpr-fibrewise-join}, the fibre of $h \pushpr k$
over $q_0 :\defeq (\star_{\zH},\star_{\zH})$ is
$\fib_{\zh}(\star_{\zH}) \join \fib_{\zk}(\star_{\zH})$.  To show
that these are both contractible, by
Proposition~\ref{prop:reduced-smash-join-contractible}, it suffices to
show that the pieces of the join are dull, reduced types.  By
Proposition~\ref{prop:reduced-types-closed}, reduced types are closed
under fibres and loop spaces, and for the basepoints these types are
dull because the basepoints are.  This shows that $\fib_{\Delta \zf
  \pushpr \Delta \zg}(p_0)$ and $\fib_{\zh \pushpr \zk}(q_0)$ are
contractible.

Since the fibres over $p_0$ and $q_0$ are contractible, to show that
general fibres $\fib_{\Delta \zf \pushpr \Delta \zg}(p)$ and $\fib_{\zh
  \pushpr \zk}(q)$ over any $p$ and $q$ are contractible, by
Lemma~\ref{lem:fibres-unique} it suffices to show that $p : (\zE
\times_{\zF} \zE) \times (\zE\times_{\zG} \zE)$ and $q :
\zH \times \zH$ are elements of 0-connected types, since being
contractible is a proposition.  For $q$, by
Corollary~\ref{cor:spec-inf-connected}, $\zH$ is $\infty$-connected
and in particular $0$-connected, so $\zH \times \zH$ is
0-connected by Lemma~\ref{lem:connected-closed}.  For $p$, again using
closure under products, we need to show that $(\zE\times_{\zF}
\zE)$ and $(\zE\times_{\zG} \zE)$ are 0-connected.  By
Corollary~\ref{cor:spec-inf-connected}, $\zE$ is 0-connected and
$\zF,\zG$ are 1-connected, so the pullbacks are 0-connected as well by
Lemma~\ref{lem:connected-closed}.  This use of 0-connectedness is
important because the fibre of a map between dull reduced types is only
dull when the point over which we are taking the fibre is, so the
argument in the previous paragraph can be applied directly to $p$ and
$q$.
\end{proof}

We showed that the counit $\Sigma \Omega \zE \to \zE$ is an equivalence in Corollary~\ref{cor:stab-susp-loop-equiv}, and we can now show that the unit is as well:
\begin{corollarystab}\label{cor:susp-loop-unit-equiv}
For any $\zE : \Redu$, the unit map $\zE \to \Omega \Sigma \zE$ is an equivalence.
\end{corollarystab}

\begin{proof}

The pushout square defining the suspension is
\[
\begin{tikzcd}
\zE \ar[d] \ar[r] \arrow[dr, phantom, "\lrcorner", very near start] \arrow[dr, phantom, "\ulcorner", very near end] & 1 \ar[d, , "\mathsf{s}"] \\
1 \ar[r, "\mathsf{n}", swap] & \Sigma \zE
\end{tikzcd}
\]
By Theorem~\ref{thm:reduced-pullback-iff-pushout}, this is also a pullback square.  However, we also have a pullback square
\[
\begin{tikzcd}
\mathsf{n} =_{\Sigma \zE} \mathsf{s} \ar[d] \ar[r] \arrow[dr, phantom, "\lrcorner", very near start] \arrow[dr, phantom, "\ulcorner", very near end] & 1 \ar[d, "\mathsf{s}"] \\
1 \ar[r, "\mathsf{n}", swap] & \Sigma \zE
\end{tikzcd}
\]
so uniqueness of pullbacks gives an equivalence $\zE \equiv (\mathsf{n}=\mathsf{s})$.  If we
consider the suspension to be pointed by $N$, then $\Omega(\Sigma \zE)$
is the type $\mathsf{n}=\mathsf{n}$.  But since $\zE$ is pointed, we have a path $\!
\mathsf{mer}(\star_{\zE}) : \mathsf{s} = \mathsf{n}$, and composition with this path gives an
equivalence $(\mathsf{n} = \mathsf{s}) \equiv (\mathsf{n} = \mathsf{n})$, so $\zE \cong \Omega(\Sigma \zE)$.

The unit of the adjunction $\Sigma \vdash \Omega$ sends $e : \zE$ to the
path $\mathsf{mer}(e) \cdot !\mathsf{mer}(\star_{\zE})$, so the composite equivalence is
indeed the unit.
\end{proof}

%
%

\section{Relating Synthetic and Analytic Spectra}\label{sec:norm}

In the previous section, we showed that Axiom S gives the reduced
types/synthetic spectra many of the properties that we expect spectra to
have.  In this section, we investigate an additional axiom, which
relates the synthetic spectra represented by reduced types to the
concrete/analytic spectra that can be defined in pure homotopy type
theory.  This also results to be proved using the reduced types and
transferred to analytic spectra, and vice versa.

One can define (``$\Omega$-'')spectra internally in type
theory~\cite{floris:thesis, cavallo:cohomology} as sequences of types
and connecting maps.

\begin{definition}
A \emph{sequential prespectrum} $J$ is a sequence of pointed modal types $J : \mathbb{N} \to \Modal_\star$ together with pointed maps $\alpha_n : J_n \to_\star \Omega J_{n+1}$. A \emph{sequential spectrum} is a prespectrum such that the $\alpha_n$ are pointed equivalences. The types of such objects are denoted $\SeqPreSpec$ and $\SeqSpec$ respectively.
\end{definition}
\begin{definition}
A morphism of sequential (pre)spectra $f : \Mor(I,J)$ is a sequence of pointed maps $f_n : I_n \to_\star J_n$ that commute with the $\alpha_n$.
\end{definition}

\begin{remark}
We need to restrict the types in the sequence to be modal so that semantically they correspond to sequences of spaces. Otherwise we would be describing a spectrum object in $P\Spec_\star$, something more complicated than an ordinary spectrum.
\end{remark}

Our goal is to relate our reduced types with these sequential spectra. We do this by describing a series of (dull) adjoints:
\begin{mathpar}
\begin{tikzcd}[column sep=large]
\Modal_\star \ar[r, bend left, "\mathrm{susp}"] \ar[r, phantom, "\perp"] & \SeqPreSpec \ar[r, phantom, "\perp"] \ar[r, bend left, "\mathrm{spec}"] \ar[l, bend left, "0th", pos=0.48] & \SeqSpec \ar[r, phantom, "\perp"] \ar[r, bend left, "L"] \ar[l, bend left, "\iota"] & \Redu \ar[l, bend left, "{R}"]
\end{tikzcd}
\end{mathpar}

The left two adjunctions take place almost entirely in pure homotopy
type theory: $\mathrm{susp}$ takes a space $X$ to the suspension
prespectrum $(X, \Sigma X, \Sigma \Sigma X, \ldots)$, and
$\mathrm{spec}$ is spectrification, inverting the connecting maps to
equivalences. Their composite $\mathrm{spec} \circ \mathrm{susp}$ is
thus an analytic analogue of the $\Sigma^\infty$ stabilisation functor
we defined in Section~\ref{sec:redu}, taking a modal type to its
suspension analytic spectrum.  The right adjoint $\iota$ is
forgetful/an inclusion, while $0th$ selects the $0th$ term of a
pre-spectrum, and the composite $0th \circ \iota$ is an analytic
analogue of $\Omega^\infty$. All that takes us out of ordinary HoTT is
the requirement that the types involved are modal types. The new
construction in this section is the rightmost adjunction relating
analytic and synthetic spectra, and an axiom stating that it is an
adjoint equivalence, making the two notions of spectra coincide.

\begin{definition}
For any pointed modal type $X$, we have the \emph{suspension sequential prespectrum} $\mathrm{susp} X$ where $(\mathrm{susp} X)_n :\defeq \Sigma^n X$, and the structure maps are the unit maps $\alpha_n : \Sigma^n X \to \Omega\Sigma^{n+1} X$. 
\end{definition}

Note that all the $\Sigma^n X$ are modal, by Proposition~\ref{prop:natural-preserves-pushouts}.

\begin{proposition}
$\mathrm{susp}$ is left adjoint to taking the zeroth type of a sequential prespectrum.
\end{proposition}
\begin{proof}
Suppose we have a map $X \to_\star J_0$. We need a map $\Sigma X \to_\star J_1$ such that
\begin{mathpar}
\begin{tikzcd}
X \ar[r] \ar[d] & \Omega \Sigma X \ar[d] \\
J_0 \ar[r] & \Omega J_1
\end{tikzcd}
\end{mathpar}
commutes, equivalently, one such that
\begin{mathpar}
\begin{tikzcd}
\Sigma X \ar[r, equals] \ar[d] & \Sigma X \ar[d] \\
\Sigma J_0 \ar[r] & J_1
\end{tikzcd}
\end{mathpar}
commutes. Then $\Sigma X \to_\star J_1$ is forced to be the composite of $\Sigma X \to \Sigma J_0 \to J_1$. This argument iterates to produce a map $\Mor(\mathrm{susp}X, J)$, showing the data of such a morphism is equivalent to that of a map $X \to_\star J_0$.
\end{proof}

\begin{definition}
For a sequential prespectrum $J$, the \emph{spectrification} of $J$ is given by
\begin{align*}
(\mathrm{spec} J)_n :\defeq \colim_k \Omega^k J_{n+k}
\end{align*}
\end{definition}
Each $(\mathrm{spec} J)_n$ is modal by Proposition~\ref{prop:natural-preserves-seq-colimits} and Lemma~\ref{lem:modality-consequences}. That these types actually assemble into a sequential spectrum has not yet been proven in type theory, so we leave it as an unjustified assertion --- note that this assertion is a statement in pure homotopy type theory, and is not dependent on the modal extension we make in this paper, a proof of the assertion would certainly apply when the types concerned happen to be modal. 

\begin{named}{Assertion 1}
This formula defines a sequential spectrum, and the operation is left adjoint to the inclusion of sequential spectra into sequential prespectra.
\end{named}

Now we turn to the $L \dashv R$ adjunction relating these sequential
spectra to our synthetic spectra'. First, we can extract a sequential
prespectrum from any reduced type $\zE$.
\begin{definition}
For $\zE : \Redu$, define $R\zE : \SeqPreSpec$ by
\begin{align*}
(R\zE)_n :\defeq \Omega^\infty \Sigma^n \zE
\end{align*}
with connecting maps
\begin{align*}
(R\zE)_n \defeq \Omega^\infty \Sigma^n \zE \to_\star \Omega^\infty \Omega \Sigma \Sigma^n \zE \equiv_\star \Omega \Omega^\infty \Sigma \Sigma^{n} \zE \equiv_\star \Omega \Omega^\infty \Sigma^{n+1} \zE \defeq \Omega (R\zE)_{n+1}
\end{align*}
The first map is given by functoriality of $\Omega^\infty$ on the unit $X \to_\star \Omega \Sigma X$ for $X = \Sigma^n \zE$.
The second is derived from Proposition~\ref{prop:omegainf-omega-commutes}
and the third is essentially by definition, depending on how iterated suspension is defined.
\end{definition}

\begin{propositionstab}
$R\zE$ is a sequential spectrum.
\end{propositionstab}
\begin{proof}
By Proposition~\ref{prop:reduced-types-closed}, reduced types are closed
under suspensions, so an induction shows that $\Sigma^n \zE$ is reduced.
Therefore the unit map $\Sigma^n \zE \to_\star \Omega \Sigma \Sigma^n \zE$
is an equivalence by Corollary~\ref{cor:susp-loop-unit-equiv}, so the connecting map defined above is a composite of equivalences.
\end{proof}

Conversely, suppose we have a sequential prespectrum $J : \SeqPreSpec$. We can produce a reduced type $LJ : \Redu$.
\begin{definition}\label{def:seqspec-to-spec}
For $J : \SeqPreSpec$ let $LJ : \Redu$ be
\begin{align*}
LJ :\defeq \colim (\Sigma^\infty J_0 \to \Omega \Sigma^\infty J_1 \to \Omega^2 \Sigma^\infty J_2 \to \dots)
\end{align*}
where the maps $\Sigma^\infty J_n \to \Omega \Sigma^\infty J_{n+1}$ are given by
\begin{align*}
\Sigma^\infty J_n \to \Omega \Sigma \Sigma^\infty J_n \equiv \Omega \Sigma^\infty \Sigma J_n \to \Omega \Sigma^\infty \Sigma \Omega J_{n+1} \to \Omega \Sigma^\infty J_{n+1}
\end{align*}
where the maps are the unit of $\Sigma \vdash \Omega$, Proposition~\ref{prop:sigma-infty-susp-commute}, functoriality on the connecting map $J_n \to \Omega J_{n+1}$ of the prespectrum, and then the counit of $\Sigma \vdash \Omega$.
\end{definition}
To see that this type is reduced, by
Proposition~\ref{prop:natural-preserves-seq-colimits}, $\natural L J$ is equivalent
to the colimit of the $\natural \Omega^n \Sigma^\infty J_i$, and
reduced types are closed under loop spaces by Proposition~\ref{prop:reduced-types-closed}, and $\Sigma^\infty$ of any type is reduced by Proposition~\ref{prop:sigmainf-reduced}, so each of the terms of that colimit is contractible, so the colimit is as well.

\begin{remark}\label{rem:suspinf-s0-is-sphere}
Let $S : \SeqPreSpec$ denote the sphere (analytic) prespectrum, i.e.\ the suspension prespectrum of $S^0$. Then $LS \equiv \sphere$, because at each level of the colimit we have
\[\Omega^n \Sigma^\infty S^n \defeq \Omega^n (S^n \topsmash \sphere) \equiv \Omega^n (\Sigma^n \sphere) \equiv \sphere \]
The first equivalence is $\Sigma^n A \equiv S^1 \topsmash (S^1 \topsmash
\ldots (S^1 \topsmash A)) \equiv (S^1 \topsmash S^1 \ldots \topsmash S^1) \topsmash A
\equiv S^n \topsmash A$,
using associativity of smash and $\Sigma A \equiv S^1 \topsmash A$~\cite[Proposition 4.2.1]{brunerie:thesis},
\cite[Definition 4.3.33]{floris:thesis}.
The second is given by iterating $\Omega \Sigma \zE \equiv \zE$ for a dull reduced type $\zE$ (Corollary~\ref{cor:susp-loop-unit-equiv}), noting that the suspension of a dull reduced type is reduced by Proposition~\ref{prop:reduced-types-closed}, and $\sphere$ is assumed to be reduced.
\end{remark}

\begin{propositionstab}
The operations $L$ and $R$ are dull adjoints: \[ \natural(L\zJ \to_\star \zE) \equiv \Mor(\zJ, R\zE) \]
\end{propositionstab}
\begin{proof}
Suppose we have a dull map $\zk : L\zJ \to_\star \zE$. We can forget the fact that $\zk$ is pointed using Proposition~\ref{prop:spectra-map-is-pointed}. The data of $\zk$ (using the universal property for maps out of a colimit) is equivalent to a sequence of dull maps $\zk_n : \Omega^n \Sigma^\infty \zJ_n \to \zE$ so that the squares
\begin{mathpar}
\begin{tikzcd}
\Omega^n \Sigma^\infty \zJ_n \ar[d, "\zk_n" swap] \ar[r] & \Omega^{n+1} \Sigma^\infty \zJ_{n+1} \ar[d, "\zk_{n+1}"] \\
\zE \ar[r, equal] & \zE
\end{tikzcd}
\end{mathpar}
commute.

We can transpose the $\zk$ across the adjunctions to get maps $\hat{k}_n : J_n \to_\star \Omega^\infty \Sigma^n \zE$. This type on the right is exactly $(R\zE)_n$, so we just have to show that this collection of maps forms a morphism of sequential prespectra. This mostly involves unwinding the definition of the map $\Omega^n \Sigma^\infty \zJ_n \to \Omega^{n+1} \Sigma^\infty \zJ_{n+1}$

Precompose the upper left corner with the equivalence $\Omega^{n+1}\Sigma^\infty\Sigma \zJ_n
\equiv \Omega^n\Omega\Sigma^\infty\Sigma \zJ_n \equiv
 \Omega^n\Omega\Sigma\Sigma^\infty \zJ_n \equiv
\Omega^n\Sigma^\infty \zJ_n$:
\begin{mathpar}
\begin{tikzcd}[column sep=5em]
\Omega^{n+1} \Sigma^\infty \Sigma \zJ_n \ar[d,"\dots" swap] \ar[r,"\Omega^{n+1}\Sigma^\infty \hat{\alpha}_n"] & \Omega^{n+1} \Sigma^\infty \zJ_{n+1} \ar[d,"\zk_{n+1}"] \\
\zE \ar[r, equal] & \zE
\end{tikzcd}
\end{mathpar}
Since $\Sigma \dashv \Omega$, and for dull reduced types Axiom S implies that $\Sigma$ and $\Omega$ are inverses,
we also have an adjunction $\Omega \dashv \Sigma$ on dull reduced types.
Transposing across this and $ \Sigma^\infty \dashv \Omega^\infty$ vertically, such squares are equivalent to squares
\begin{mathpar}
\begin{tikzcd}
\Sigma  \zJ_n \ar[d, "\dots" swap] \ar[r, "\hat{\alpha}_n"] & \zJ_{n+1} \ar[d, "\hat{\zk}_{n+1}"] \\
(R\zE)_{n+1} \ar[r, equal] & (R\zE)_{n+1}
\end{tikzcd}
\end{mathpar}
Now transposing `diagonally' along $\Sigma \dashv \Omega$, we get
\begin{mathpar}
\begin{tikzcd}
\zJ_n \ar[d, "\dots" swap] \ar[r, "\alpha_n"] & \Omega \zJ_{n+1} \ar[d, "\Omega \hat{\zk}_{n+1}"] \\
\Omega (R\zE)_{n+1} \ar[r, equal] & \Omega (R\zE)_{n+1}
\end{tikzcd}
\end{mathpar}
Finally, precompose the lower left corner with the equivalence $\beta_n : (R\zE)_n \equiv \Omega (R\zE)_{n+1}$ to get
\begin{mathpar}
\begin{tikzcd}
\zJ_n \ar[d, "\dots" swap] \ar[r, "\alpha_n"] & \Omega \zJ_{n+1} \ar[d, "\Omega \hat{\zk}_{n+1}"] \\
(R\zE)_n \ar[r, "\beta_n" swap] & \Omega (R\zE)_{n+1}
\end{tikzcd}
\end{mathpar}

All we have left is to check is that the vertical map on the left is equal to $\hat{k}_n$. Tracing through what we have done, it is equal to
\begin{align*}
  \zJ_n &\to \Omega \Sigma \zJ_n
  \to \Omega \Omega^\infty  \Sigma^{n+1} \Omega^{n+1} \Sigma^\infty \Sigma \zJ_n
  \to \Omega \Omega^\infty  \Sigma^{n+1} \Omega^{n+1} \Sigma \Sigma^\infty \zJ_n \\
  &\to \Omega \Omega^\infty \Sigma^{n+1} \Omega^n \Sigma^\infty \zJ_n
  \to \Omega \Omega^\infty \Sigma^{n+1} \zE
  \to \Omega^\infty \Omega \Sigma^{n+1} \zE
  \to \Omega^\infty \Sigma^n \zE
\end{align*}
In the above sequence, the map
$\Omega \Omega^\infty \Sigma^{n+1} \Omega^n \Sigma^\infty \zJ_n \to
\Omega \Omega^\infty \Sigma^{n+1} \zE$ is given by
$\Omega \Omega^\infty \Sigma^{n+1} \zc{k}_n$, so by naturality, we can
move this use of $\zc{k}_n$ to the end. The above chain is then equal to the composite
\begin{align*}
  \zJ_n &\to \Omega \Sigma \zJ_n
  \to \Omega \Omega^\infty  \Sigma^{n+1} \Omega^{n+1} \Sigma^\infty \Sigma \zJ_n
  \to \Omega \Omega^\infty  \Sigma^{n+1} \Omega^{n+1} \Sigma \Sigma^\infty \zJ_n \\
  &\to \Omega \Omega^\infty \Sigma^{n+1} \Omega^n \Sigma^\infty \zJ_n
  \to \Omega^\infty \Omega \Sigma^{n+1} \Omega^n \Sigma^\infty \zJ_n
  \to \Omega^\infty \Sigma^n \Omega^n \Sigma^\infty \zJ_n
  \to \Omega^\infty \Sigma^n \zE
\end{align*}
For this to be the transpose of $\zc{k}_n$, what we need is for the
composite
$\zJ_n \to \Omega^\infty \Sigma^n \Omega^n \Sigma^\infty \zJ_n$,
leaving off the last map, is equal to the unit.  And it is, by liberal
use of the triangle inequalities, as every map in the string is either
unit/counit or an equality of composites in a system of commuting
adjunctions.

\end{proof}

Having established this adjunction, we introduce the following axiom, which identifies the synthetic and analytic spectra:

\begin{named}{Axiom N}
The adjunction between $\Redu$ and $\SeqSpec$ is a dull adjoint equivalence, i.e.\ the map $LR\zE \to_\star \zE$ is an equivalence and $\Mor(\zJ, RL\zJ)$ is a level-wise equivalence.
\end{named}

As an application, we show that this axiom fixes the stable homotopy groups of $\sphere$, in the sense of Definition~\ref{def:stable-homotopy-groups}, to be the actual stable homotopy groups of the ordinary spheres.

\begin{remark}\label{rem:composite-adjunction}
The composite right adjoint $\Redu \to \SeqSpec \to \SeqPreSpec \to \Modal_\star$ in the diagram at the beginning of this section is $(RE)_0 \defeq \Omega^\infty E$. So the composite adjunction must be equivalent to the adjunction
\begin{mathpar}
\begin{tikzcd}[column sep=large]
\Modal_\star \ar[r, bend left, "\Sigma^\infty", pos=0.52] & \Redu \ar[l, bend left, "\Omega^\infty"]
\end{tikzcd}
\end{mathpar}
that we already have from Proposition~\ref{prop:sigmainf-omegainf-adjoint}.
\end{remark}

\begin{definition}
Let $Q\zX :\defeq \Omega^\infty \Sigma^\infty \zX$.
\end{definition}

\begin{lemmanorm}\label{lem:qx}
\begin{align*}
Q\zX \equiv \colim_k \Omega^k \Sigma^k \zX
\end{align*}
\end{lemmanorm}
\begin{proof}
  By Remark~\ref{rem:composite-adjunction}, we have
  $Q \zX \defeq \Omega^\infty \Sigma^\infty \zX \equiv \mathrm{0th}(\iota(R(L(\mathrm{spec}(\mathrm{susp}(\zX))))))$.
  Axiom N allows us to chop the $L \dashv R$ adjunction off this roundtrip on $\Modal_\star$,
  and $\iota$ is just forgetful, so we get
  $\mathrm{0th}(\mathrm{spec}(\mathrm{susp}(\zX)))$.
  Then $\colim_k \Omega^k \Sigma^k \zX$ is exactly the 0th type of the spectrification of the suspension prespectrum of $\zX$.
\end{proof}

\begin{propositionnorm}
\[ \pi^s_n(\sphere)\equiv \colim_k \pi_{n+k}(S^k)\]
\end{propositionnorm}
\begin{proof}
We will make use of some properties of sequential colimits proven in~\cite{sbr:seq-colims}. Specifically, sequential colimits commute with taking loop spaces~\cite[Corollary 7.4]{sbr:seq-colims} and truncations~\cite[Corollary 7.6]{sbr:seq-colims}, and thus calculating homotopy groups.
\begin{align*}
\pi^s_n(\sphere)
&\defeq \pi_n(\Omega^\infty \sphere) && \text{(by definition)} \\
&\equiv \pi_n(\Omega^\infty \Sigma^\infty S^0) && \text{(by Remark~\ref{rem:suspinf-s0-is-sphere})} \\
&\equiv \pi_n(\colim_k \Omega^k \Sigma^k S^0) && \text{(by the Lemma~\ref{lem:qx})}  \\
&\equiv \colim_k \pi_n(\Omega^k \Sigma^k S^0) && \text{(sequential colimit commutes with $\pi_n$)} \\
&\equiv \colim_k \pi_{n+k}(\Sigma^k S^0) && \text{(definition of $\pi$)} \\
&\equiv \colim_k \pi_{n+k}(S^k) && \text{(definition of $S^n$)}\\
\end{align*}
\end{proof}


%
%

\newpage
\part{Metatheory}

\section{Fully-Annotated Syntax and Proofs of Admissible Rules}\label{sec:admissible-rules}

The rules for the natural type in Figures~\ref{fig:structural} and~\ref{fig:natural-rules}
can be applied to various precise formulations of type theory.  For
example, if the ambient Martin-L\"of type theory is thought of
algebraically as an essentially algebraic theory / quotient
inductive-inductive type~\cite{qiits}, then we could make $\zc{\Gamma}$
a new context former, $\zc{a}$ a new term former, the unit a new
explicit substitution, and the rules defining these operations new
judgemental equalities.  However, this kind of algebraic formulation
does not immediately capture two important aspects of our syntax.  The
first is that the $\zc{\Gamma}$ and $\za$ are definable in terms of
marked context extension and marked variables --- this would need to be
recovered as part of a canonicity proof.  The second is that the unit is
``silent,'' i.e.\ it does not change the raw proof term --- in the
algebraic style, it would be made explicit analogously to weakening.

The simplest way to make these observations formal is to adopt a more
traditional syntax, where the subjects of a judgement $\Gamma \vdash a :
A$ are a raw syntax context $\Gamma$, a raw syntax term $a$, and a raw
syntax type $A$.  These weak invariants ``break the loop'' so that the
basic inference rules can be defined prior to the admissible rules ---
otherwise, one requires the admissible rules to know that the
presuppositions of the judgements are satisfied. For example, the
premise of $(\Gamma, \zx :: A) \ctx$ is $\zc{\Gamma} \vdash A \type$,
but $\zc{\Gamma} \ctx$ only follows from an admissible rule.

\subsection{Official Rules}

When making rules in this style precise, there are some somewhat
arbitrary choices about the presuppositions of a judgement.  For
example, a derivation of $\Gamma \vdash a : A$ might
\begin{enumerate}
\item presuppose that $\Gamma$ is well-formed, i.e.\ the subject is really a raw context
$\Gamma$ such that $\Gamma \ctx$.

\item check $\Gamma \ctx$ as part of the derivation.

\item neither of the above, i.e.\ formally one can make derivations of
  $\Gamma \vdash a : A$ for an ill-formed context $\Gamma$, but we
  generally will only be interested in derivations when when $\Gamma
  \ctx$.
\end{enumerate}
The first has the same problem as the algebraic syntax --- it requires the
admissible rules to be mutual with the basic ones --- and the second is a
bit far from an implementation, which inductively maintains the
invariant that the context is well-formed without repeatedly re-checking
it, so we follow the third option for contexts in all judgements.
However, for types, the rules will ensure that the type is well-formed,
and for equality rules, the rules will ensure that the terms are
well-typed:
\begin{itemize}
\item If $\Gamma \ctx$ and $\Gamma \vdash a : A$ then $\Gamma \vdash A \type$.
\item If $\Gamma \ctx$ and $\Gamma \vdash a \defeq a' : A$ then $\Gamma
  \vdash a : A$ and $\Gamma \vdash a' : A$.
\end{itemize}
This is because, following~\cite{streicher:book,hofmann:interp}'s
approach to categorical semantics and initiality, we officially adopt
a fully annotated term syntax, where every inference rule has a direct
typing premise for each term/type metavariable appearing in the rule.
(Of course, this is also a bit far from an implementation.)  Together
with the admissible rules, these premises will be enough to ensure
that the types in a typing judgement and terms in an equality
judgement are well-formed. De Boer, Brunerie, Lumsdaine and M\"ortberg
(see~\cite{deboer20initiality}) have given a fully mechanised
initiality proof for roughly this style of presentation, though we
treat variable binding informally rather than using de Bruijn indices.

According to these conventions, the official basic rules for the type
theory with $\natural$ and $\Pi$ are in Figure~\ref{fig:rules-official}.
\begin{figure}
\begin{mathpar}
 \inferrule*{~}{\cdot \ctx} \qquad
 \inferrule*{\Gamma \ctx \and \Gamma \vdash A \type}
            {\Gamma,x:A \ctx} \qquad
 \inferrule*{\Gamma \ctx \and \zc{\Gamma} \yields A \type}
            {\Gamma, \zx :: A  \ctx}
 \\
 \inferrule*{~}{\Gamma \yields \cdot \tele} \qquad
 \inferrule*{\Gamma \yields \Delta \tele \and \Gamma,\Delta \yields A \type}
            {\Gamma \yields \Delta,x:A \tele} \qquad
 \inferrule*{\Gamma \yields \Delta \tele \and \zc{\Gamma,\Delta} \yields A \type}
            {\Gamma \yields \Delta, \zx :: A  \tele}
 \\
 \inferrule*[left=var]{\Gamma \vdash A \type} {\Gamma, x : A, \Gamma' \yields x : A}
 \quad
 \inferrule*[left=var-zero]{\zc{\Gamma} \vdash A \type}
                           {\Gamma, \zx :: A, \Gamma' \yields \zx : A}
 \quad
 \inferrule*[left=var-roundtrip]{\Gamma \vdash A \type}
                                {\Gamma, x : A, \Gamma' \yields \zx : \zA}\\
\inferrule*[left=$\natural$-form]{\zc{\Gamma} \yields A \type }{\Gamma
  \yields \natural{A} \type}
\qquad
\inferrule*[left=$\natural$-intro]{\zc{\Gamma} \yields A \type \quad
                              \zc{\Gamma} \yields a : A}
                             {\Gamma \yields a_A^\natural : \natural{A}}
\qquad
\inferrule*[left=$\natural$-elim]{\zc{\Gamma} \yields A \type \quad
                              \Gamma \yields b : \natural A}
                             {\Gamma \yields b^A_\natural : A} \\
\inferrule*[left=$\natural$-beta]
           {\zc{\Gamma} \yields A \type \quad \zc{\Gamma} \yields a : A}
           {\Gamma \yields (a_A^\natural{})^A_\natural \defeq a : A}
\qquad
\inferrule*[left=$\natural$-eta]
           {\zc{\Gamma} \yields A \type \quad \Gamma \yields b : \natural A}
           {\Gamma \yields b \defeq (\zb{}^{\zA}_\natural){}_A^\natural : \natural A}
\\
 \inferrule*{\Gamma \vdash A \type \quad
             \Gamma,x:A \yields B \type}
             {\Gamma \yields \Pi x:A. B \type} \\
 \inferrule*{\Gamma \vdash A \type \quad
             \Gamma,x:A \yields B \type \quad
             \Gamma,x:A \yields b : B
            }
            {\Gamma \yields (\lambda x:A.b:B) : \Pi {x:A}. B} \quad
 \inferrule*{\Gamma \vdash A \type \quad
             \Gamma,x:A \yields B \type \quad
             \Gamma \yields f : \Pi x:A.B \quad
             \Gamma \yields a : A
             }
             {\Gamma \yields f (a)_{A,x.B} : B[a/x]}
  \\
  \inferrule*{\Gamma \vdash A \type \quad
              \Gamma,x:A \yields B \type \quad
               \Gamma,x:A \yields b : B \quad
               \Gamma \yields a : A
             }
             {\Gamma \yields (\lambda x:A.b:B) (a)_{A,x.B} \defeq b[a/x] : B[a/x]} \quad
  \inferrule*{\Gamma \vdash A \type \quad
              \Gamma,x:A \yields B \type \quad
              \Gamma \yields f : \Pi x:A .B
             }
             {\Gamma \yields f \defeq \lambda x:A.f(x)_{A,x.B} : B : \Pi
               x:A.B}
\end{mathpar}
We omit the structural rules: type conversion for typing; and reflexivity,
symmetry, transitivity, compatibility for each constructor, and type
conversion for equality.
\caption{Official Rules}\label{fig:rules-official}
\end{figure}

\subsection{Operations on Raw Syntax}

The zeroing operations $\zc{\Gamma}$ and $a^{0\Gamma}$ and
$\Delta^{0\Gamma}$ are defined on raw syntax (prior to typing).  This
accords with the way we use these operations when working informally: we
do not want to have to think about the exact structure of a derivation
as we zero a term.  We use the following judgements for raw terms:

\begin{itemize}
\item $\gamma \yields a \rawterm$ denotes a raw term in the scope
  $\gamma$.  We think of raw syntax as being intrinsically scoped, so a
  raw term $a$ is judged relative to a scope $\gamma$ consisting of
  variable names only, with no associated types or marks.  We do not
  distinguish between ordinary and marked variables in scopes $\gamma$
  because we want precomposition with the unit to be ``silent'', and
  precomposition with the unit changes the marked context extension into
  into unmarked context extension.  At the level of raw syntax, $\zx$ is
  simply a term constructor $\mathtt{underline}(x)$ that takes a
  variable as input.

\item $\Gamma \rawctx$ denotes a raw context consisting of a list of
  variables with a raw term as a `type' of each, and with possible
  marks on the variables. What is missing from $\Gamma \ctx$ is that
  the types are not necessarily well-formed. Every $\Gamma \rawctx$
  has an underlying scope of variable names, written $dom(\Gamma)$.

\item $\gamma \yields \Delta \rawtele$ similarly denotes a raw telescope
  in the scope $\gamma$ with possible marks on the variables.
\end{itemize}

The zeroing operations on raw syntax have shape
\begin{mathpar}
  \inferrule*[fraction={-{\,-\,}-}]{\Gamma \rawctx}{\zc{\Gamma} \rawctx} \and
  \inferrule*[fraction={-{\,-\,}-}]{\gamma \yields \Delta \rawtele}{\gamma \yields \zc{\Delta} \rawtele} \and
  \inferrule*[fraction={-{\,-\,}-}]{\gamma, \gamma' \yields \Delta \rawtele}{\gamma, \gamma' \yields \Delta^{0\gamma} \rawtele} \and
  \inferrule*[fraction={-{\,-\,}-}]{\gamma, \delta \yields a \rawterm}{\gamma, \delta \yields a^{0\gamma} \rawterm}
\end{mathpar}
When we apply these operations to well-typed telescopes and terms, we
will just write, for example, $\Delta^{0\Gamma}$ and $a^{0\Gamma}$, letting
$\Gamma$ represent its underlying list of variables.

The definitions of the operations on raw syntax are in
Figure~\ref{fig:operations}.

\begin{figure}

\begin{mathpar}
\begin{array}{rl}
\Gamma,\cdot &:\defeq \Gamma \\
\Gamma,(\Delta,x:A) &:\defeq (\Gamma,\Delta),x:A \\
\Gamma,(\Delta,\zx::A) &:\defeq (\Gamma,\Delta),\zx::A
\end{array}
\\
\begin{array}{rl}
\zc{\cdot} &:\defeq \cdot \\
\zc{\Gamma, x : A} &:\defeq \zc{\Gamma}, \zx :: \zA \\
\zc{\Gamma, \zx :: A} &:\defeq \zc{\Gamma}, \zx :: A
\end{array}
\\
\begin{array}{rl}
x^{0\gamma} &:\defeq \zx \text{  if $x \in \gamma$} \\
x^{0\gamma} &:\defeq x \text{ if $x \notin \gamma$} \\
\zx^{0\gamma} &:\defeq \zx  \\
(\natural A)^{0\gamma} &:\defeq \natural(A^{0\gamma}) \\ 
(a_A^\natural)^{0\gamma} &:\defeq {a^{0\gamma}_{A^{0\gamma}}}^\natural \\
(a_\natural^A)^{0\gamma} &:\defeq {a^{0\gamma}}^{A^{0\gamma}}_\natural \\
(\lambda x:A.b:B)^{0\gamma} &:\defeq \lambda x:A^{0\gamma}.b^{0\gamma}:B^{0\gamma}\\
f(a)_{A,x.B}^{0\gamma} &:\defeq f^{0\gamma}(a^{0\gamma})_{A^{0\gamma},x.B^{0\gamma}}\\
\end{array}
\\
\begin{array}{rl}
\zc{\cdot} &:\defeq \cdot \\
\zc{\Delta, x : A} &:\defeq \zc{\Delta}, \zx :: \zA \\
\zc{\Delta, \zx :: A} &:\defeq \zc{\Delta}, \zx :: A
\end{array}
\\
\begin{array}{rl}
(\cdot)^{0\gamma} &:\defeq \cdot \\
(\Delta, x : A)^{0\gamma} &:\defeq \Delta^{0\gamma}, x : A^{0\gamma} \\
(\Delta, \zx :: A)^{0\gamma} &:\defeq \Delta^{0\gamma}, \zx :: A
\end{array}
\\
\begin{array}{rl}
  \zx [ a / x ] &:\defeq \za\\
  \zx [ a / y ] &:\defeq \zx  \text{ if $y \neq x$} \\
  x[a/x] &:\defeq a\\
  y[a/x] &:\defeq y \text{ if $y \neq x$}\\
  (\natural A)[a/x] &:\defeq \natural(A[a/x])\\
  (b_A^\natural)[a/x] &:\defeq (b[a/x])_{A[a/x]}^\natural\\
  (b_\natural^A)[a/x] &:\defeq (b[a/x])_\natural^{A[a/x]}\\
  (\lambda (x:A).b:B)[a/x] &:\defeq \lambda (x:A[a/x]).b[a/x]:B[a/x]\\
  f(b)_{A,x.B}[a/x] &:\defeq (f[a/x]) (b[a/x])_{A[a/x],y.B[a/x]}\\
\end{array}
\end{mathpar}
\caption{Operations on Raw Syntax}\label{fig:operations}
\end{figure}

Idempotence properties of contexts and terms now hold on the level of
raw syntax.
\begin{lemma}\label{lem:syn-idem}
  \begin{itemize}
  \item If $\Psi \rawctx$
       and $dom(\Psi) \vdash \Gamma \rawtele$
        and $dom(\Psi,\Gamma) \yields \Delta \rawtele$,
    \begin{align*}
      \zc{\Psi,\Gamma, \Delta} \defeq_\alpha \zc{\Psi,\zc{\Gamma}, \Delta^{0\Gamma}}
    \end{align*}

  \item If $\Gamma \rawctx$
        and $dom(\Gamma) \yields \Delta \rawtele$,
    \begin{align*}
      \zc{\Gamma, \Delta} \defeq_\alpha \zc{\zc{\Gamma}, \Delta^{0\Gamma}}
    \end{align*}

  \item If
  $\gamma, \gamma' \yields \Delta \rawtele$ then
  \begin{align*}
    {(\Delta^{0\gamma})}^{0(\gamma, \gamma')} \defeq_\alpha \Delta^{0(\gamma, \gamma')} \defeq_\alpha {(\Delta^{0(\gamma, \gamma')})}^{0\gamma}
  \end{align*}
  \item If
  $\gamma, \delta, \delta' \yields a \rawterm$. Then
  \begin{align*}
   {(a^{0\gamma})}^{0(\gamma, \delta)} \defeq_\alpha a^{0(\gamma, \delta)} \defeq_\alpha {(a^{0(\gamma, \delta)})}^{0\gamma}
  \end{align*}
  \end{itemize}
\end{lemma}
In words, zeroing a larger piece of the context subsumes zeroing a smaller piece.

Zeroing and weakening/substitution also commute at the level of raw syntax:
\begin{lemma}
  \begin{itemize}
  \item If $\gamma \yields a \rawterm$ then
    \begin{align*}
      \gamma, \delta \yields a^{0\gamma} \defeq_\alpha a^{0(\gamma, \delta)}
    \end{align*}
  \item If $\gamma, \gamma' \yields a \rawterm$ and $\gamma, \gamma', x, \delta \yields b \rawterm$ then
    \begin{align*}
      (b[a/x])^{0\gamma} &\defeq_\alpha b^{0\gamma}[a^{0\gamma}/x]
    \end{align*}
  \item If $\gamma \yields a \rawterm$ and $\gamma, x, \delta, \delta' \yields b \rawterm$ then
    \begin{align*}
      {(b[a/x])}^{0(\gamma, \delta)} \defeq_\alpha b^{0(\gamma, x, \delta)}[a^{0\gamma}/x] \defeq_\alpha b^{0(\gamma, x, \delta)}[a/x]
    \end{align*}
  \end{itemize}
\end{lemma}

Additionally, we have
\begin{lemma}
  If $\zc{\Gamma}, \Delta \yields a : A$ then $a^{0{\zc{\Gamma}}} \defeq_\alpha a$
\end{lemma}
For this, we do need that $a : A$ is a well-typed term, so that
variables from $\zc{\Gamma}$ are used in the `correct' way: well-typed
$a$'s can only use variables from $\zc{\Gamma}$ with a marking, but raw
terms might incorrectly use them unmarked.
\begin{proof}
  Intuitively, the only time $a^{0\gamma}$ is not equal to $a$ is for a
  unmarked variable $x \in \gamma$, but these cannot appear in a
  well-typed term in a context with $\Gamma$ already marked.  The
  induction goes through because in each inference rule, if $\Gamma$ is
  marked in the conclusion, then any variable in $\Gamma$ stays marked
  in all premises where it occurs.
\end{proof}

\subsection{Proofs}

We now check that the intended admissible rules are indeed admissible.
To simplify the task, recall that \rulen{pre-counit} is a special case
of \rulen{pre-counit-gen}, defining $\zc{a} :\defeq a^{0\Gamma}$ for a
term $\Gamma \yields a : A$.  Additionally \rulen{pre-roundtrip-gen} and
\rulen{pre-roundtrip} follow by composing \rulen{pre-unit} and
\rulen{pre-counit-gen}, so we do not need to prove them separately.
Thus, what remains is to verify typing for $\zc{\Gamma}$ and
$a^{0\Gamma}$ and $\Delta^{0\Gamma}$ and that substitution is still
admissible.

First, we verify that \rulen{ctx-zero} and \rulen{pre-counit-gen} are
admissible.

\begin{proposition}
  \rulen{ctx-zero} and \rulen{pre-counit-gen} are admissible.
  \begin{mathpar}
    \inferrule*[left=pre-counit-gen,fraction={-{\,-\,}-}]
    {\Gamma,\Delta \yields a:A}
    {\zc{\Gamma}, \Delta^{0\Gamma} \yields a^{0\Gamma} : A^{0\Gamma}}

    \inferrule*[fraction={-{\,-\,}-}]
    {\Gamma,\Delta \yields a \defeq a' :A}
    {\zc{\Gamma}, \Delta^{0\Gamma} \yields a^{0\Gamma} \defeq a'^{0\Gamma} : A^{0\Gamma}}

    \inferrule*[fraction={-{\,-\,}-}]
    {\Gamma,\Delta \yields A \type}
    {\zc{\Gamma}, \Delta^{0\Gamma} \yields A^{0\Gamma} \type}

    \inferrule*[fraction={-{\,-\,}-}]
    {\Gamma,\Delta \yields A \defeq A' \type}
    {\zc{\Gamma}, \Delta^{0\Gamma} \yields A^{0\Gamma} \defeq A'^{0\Gamma} \type}
  \end{mathpar}
\end{proposition}
\begin{proof}

  All four rules are proved by mutual induction on the typing/type
  formation/equality derivations.  The most important cases are the
  variable rules, but we also show the $\natural$ formation,
  introduction, and elimination rules, and $\Pi$-formation to
  demonstrate bound variables.

  \begin{itemize}
  \item \rulen{var}: There are two sub-cases, depending on whether the variable is in $\Gamma$ or $\Delta$.
    \begin{itemize}
    \item In the former case, we have
      \[
      \inferrule*{\Gamma \vdash A \type}
                 {\Gamma,x:A,\Gamma',\Delta \vdash x : A}
      \]
      Then $\zx :: A^{0\Gamma} \in \zc{\Gamma, x : A, \Gamma'}$. Since
      the context $\Gamma,x:A$ is well-scoped, $A$ only uses variables
      from $\Gamma$, and so $A^{0\Gamma} \defeq A^{0\Gamma, x : A,
        \Gamma'}$ because zeroing commutes with weakening. So
      $\zc{\Gamma, x : A, \Gamma'}, \Delta^{0\Gamma} \yields \zx :
      A^{0\Gamma, x : A, \Gamma'}$ is well-typed by \rulen{var-zero},
      because $\zc{\Gamma} \vdash A^{0\Gamma} \type$ by the inductive
      hypothesis for the premise (with the empty telescope), so
      $\zc{\zc{\Gamma}} \vdash A^{0\Gamma} \type$ by idempotence.
    \item In the latter case, we have
      \[
      \inferrule*{\Gamma,\Delta \vdash A \type}
                 {\Gamma,\Delta,x:A,\Delta' \vdash x : A}
                 \]
      By the IH on the premise, $\zc{\Gamma},\Delta \vdash A^{0\Gamma} \type$, and $(\Delta,x:A,\Delta')^{0\Gamma} \defeq
      \Delta^{0\Gamma},x:A^{0\Gamma},\Delta'^{0\Gamma}$.  Thus, we can
      conclude $\zc{\Gamma},
      \Delta^{0\Gamma},x:A^{0\Gamma},\Delta'^{0\Gamma} \yields x :
      A^{0\Gamma}$ by \rulen{var}.
    \end{itemize}

  \item \rulen{var-zero}:  We again distinguish cases depending on
    whether the variable was in $\Gamma$ or $\Delta$.
    \begin{itemize}
    \item In the former case, we have
      \[
      \inferrule*{\zc{\Gamma} \yields A \type}
                 {\Gamma,\zx::A,\Gamma',\Delta \vdash \zx : A}
      \]
      Expanding the definition of zeroing on contexts, $\zx :: A \in \zc{\Gamma}$.  Since $A$
      is well-typed in a context with $\Gamma$ zeroed, $A \defeq
      A^{0\Gamma} \defeq A^{0(\Gamma,x,\Gamma')}$.  $\zc{\zc{\Gamma}}
      \yields A \type$ holds by idempotence, so so $\zc{\Gamma},
      \Delta^{0\Gamma} \yields \zx : A^{0\Gamma}$ by \rulen{var-zero}.
    \item In the latter case we have
      \[
      \inferrule*{\zc{\Gamma,\Delta} \yields A \type}
                 {\Gamma,\Delta,\zx::A,\Delta' \vdash \zx : A}
      \]
      Then by definition $\zx :: A \in \Delta^{0\Gamma}$, and again $A
      \defeq A^{0\Gamma}$ because $A$ is well-typed in a context
      $\zc{\Gamma,\Delta}$ with $\Gamma$ already zeroed. So
      $\zc{\Gamma}, \Delta^{0\Gamma} \yields \zx : A^{0\Gamma}$ is
      holds by \rulen{var-zero} because the premise
      $\zc{\zc{\Gamma},\Delta^{0\Gamma}} \yields A \type$ holds
      by idempotence.
    \end{itemize}

  \item \rulen{var-roundtrip}: Note that $\zA$ in the conclusion of the
    rule means $A^{0(\Gamma,x:A,\Gamma')}$, i.e. $A$ zeroed with respect
    to all the variables in the entire context.
    We again distinguish cases depending on
    whether the variable was in $\Gamma$ or $\Delta$.
    \begin{itemize}
    \item In the former case, we are given
      \[
      \inferrule*{\Gamma \yields A \type}
                 {\Gamma, x:A, \Gamma',\Delta \yields \zx : \zA}
      \]
      Then
      $\zx :: A^{0\Gamma} \in \zc{\Gamma, x : A, \Gamma'}$, and by idempotence
      \[ {(A^{0\Gamma})} \defeq (A^{0\Gamma})^{0(\Gamma, x : A,
          \Gamma', \Delta)} \defeq {(A^{0(\Gamma, x : A, \Gamma',
          \Delta)})}^{0\Gamma} \defeq \zA^{0(\Gamma, x : A, \Gamma')} \]
      by idempotence/commuting with weakening, so $\zc{\Gamma, x : A, \Gamma'}, \Delta^{(0\Gamma,
        x : A, \Gamma')} \yields \zx : {\zA}^{0(\Gamma, x : A,
        \Gamma')}$ is well-typed by \rulen{var-zero}, because the type
      $\zc{\Gamma} \yields A^{0\Gamma} \type$ is well-formed by the IH on the premise.

    \item If $x : A \in \Delta$, then we were given
      \[
      \inferrule*{\Gamma,\Delta \yields A \type}
                 {\Gamma,\Delta,x:A,\Delta' \yields \zx : \zA}
      \]
      Then
      \[\zc{A^{0\Gamma}} \defeq
           (A^{0\Gamma})^{0(\Gamma, \Delta,x : A,\Delta')} \defeq {(A^{0(\Gamma, \Delta,x : A, \Delta')})}^{0\Gamma}
      \defeq
          {\zA}^{0\Gamma}
          \]
          by idempotence
      so $\zc{\Gamma}, \Delta^{0\Gamma},x:A^{0\Gamma},\Delta'^{0\Gamma}
      \yields \zx : {\zA}^{0\Gamma}$ by \rulen{var-roundtrip}, using the
      IH on the premise to get $\zc{\Gamma},\Delta^{0\Gamma} \yields
      A^{0\Gamma} \type$.
    \end{itemize}

  \item \rulen{$\natural$-form}: Suppose we have
    $\Gamma, \Delta \yields \natural A \type$ because
    $\zc{\Gamma, \Delta} \yields A \type$. By idempotence, the context
    can be rewritten
    $\zc{\zc{\Gamma}, \Delta^{0\Gamma}} \yields A \type$, and also
    $A \defeq A^{0\Gamma}$ because $A$ is a well-typed term in a context
    with $\zc{\Gamma}$ already marked. Reapplying the rule then gives
    $\zc{\Gamma}, \Delta^{0\Gamma} \yields \natural (A^{0\Gamma})$ as
    required.

  \item \rulen{$\natural$-intro}: Suppose we have $\Gamma, \Delta
    \yields a^\natural_A : \natural A$ because $\zc{\Gamma, \Delta}
    \yields A \type$ and $\zc{\Gamma, \Delta} \yields a : A$.
    By idempotence, the context $\zc{\Gamma, \Delta}$ can be rewritten
    as $\zc{\zc{\Gamma}, \Delta^{0\Gamma}}$, so we also have
    $\zc{\zc{\Gamma}, \Delta^{0\Gamma}} \yields A \type$ and
    $\zc{\zc{\Gamma}, \Delta^{0\Gamma}} \yields a : A$.  Because $a$ and
    $A$ are well-formed terms in a context with $\zc{\Gamma}$ already
    marked, we have $a \defeq_\alpha a^{0\Gamma}$ and $A \defeq_\alpha
    A^{0\Gamma}$, so $\zc{\zc{\Gamma}, \Delta^{0\Gamma}} \yields
    A^{0\Gamma} \type$ and $\zc{\zc{\Gamma}, \Delta^{0\Gamma}} \yields
    a^{0\Gamma} : A^{0\Gamma}$ as well.  Reapplying the rule then gives
    $\zc{\Gamma}, \Delta^{0\Gamma} \yields
    (a^{0\Gamma})^\natural_{A^{0\Gamma}} : \natural (A^{0\Gamma})$ as
    required.

  \item \rulen{$\natural$-elim}: Suppose we have $\Gamma,\Delta \yields
    a_\natural^A : A$ because $\zc{\Gamma, \Delta} \yields A \type$ and
    $\Gamma, \Delta \yields a : \natural A$.  The inductive hypothesis
    for $a$ gives $\zc{\Gamma}, \Delta^{0\Gamma} \yields a^{0\Gamma} :
    \natural A^{0\Gamma}$.  For the type, by idempotence $\zc{\Gamma,
      \Delta} \defeq \zc{\zc{\Gamma}, \Delta^{0\Gamma}}$, and $A \defeq
    A^{0\Gamma}$ because $A$ is already well-typed in a context with
    $\Gamma$ marked. So we have $\zc{\zc{\Gamma}, \Delta^{0\Gamma}}
    \yields A^{0\Gamma} \type$ and $\zc{\Gamma}, \Delta^{0\Gamma}
    \yields a^{0\Gamma} : \natural A^{0\Gamma}$ and can reapply the rule
    to get $\zc{\Gamma}, \Delta^{0\Gamma} \yields
    {((a^{0\Gamma})^{A^{0\Gamma}}_\natural)} : A^{0\Gamma}$.

  \item \rulen{$\Pi$-intro}: If the inputs are
    \begin{align*}
      \Gamma, \Delta &\yields A \type \\
      \Gamma, \Delta, x : A &\yields B \type
    \end{align*}
    then inductively
    \begin{align*}
      \zc{\Gamma}, \Delta^{0\Gamma} &\yields A^{0\Gamma} \type \\
      \zc{\Gamma}, \Delta^{0\Gamma}, x : A^{0\Gamma} &\yields B^{0\Gamma} \type
    \end{align*}
    where in the later case we extend the telescope to $\Delta, x : A$. Reapplying the rule gives
    \begin{align*}
      \zc{\Gamma}, \Delta^{0\Gamma} \yields \prd{x : A^{0\Gamma}} B^{0\Gamma} \type
    \end{align*}
    as required.
  \end{itemize}
\end{proof}

\begin{proposition} \rulen{ctx-zero} is admissible.
\[    \inferrule*[fraction={-{\,-\,}-}]{\Gamma \ctx}{\zc{\Gamma} \ctx}
\]
\end{proposition}
\begin{proof}
    \begin{itemize}
  \item In the case for $\cdot$, then $\zc{\Gamma} \defeq \cdot$ is well-formed.
  \item In the case for $\Gamma, x : A \ctx$, then $\Gamma \ctx$ and $\Gamma \yields A \type$, so
    applying the IH gives $\zc{\Gamma} \ctx$ and \rulen{pre-counit} for
    the type gives us $\zc{\Gamma} \yields \zA \type$.  By
    idempotence $\zc{\Gamma} = \zc{\zc \Gamma}$, so $\zc{\zc{\Gamma}}
    \yields \zA \type$ as well, and $\zc{\Gamma} \yields \zc{\Delta}, \zx :: \zA \ctx$ is
    well-formed.
  \item In the case for $\Gamma, \zx :: A \ctx$, we have $\Gamma \ctx$
    and
    $\zc{\Gamma} \yields A \type$. By the IH $\zc{\Gamma} \ctx$.  By idempotence
    $\zc{\Gamma} \defeq \zc{\zc{\Gamma}}$, so applying
    \rulen{ctx-ext-zero} to $\zc{\zc{\Gamma}} \yields A \type$ gives
    that $\zc{\Gamma}, \zx :: A \ctx$ is well-formed.
  \end{itemize}
\end{proof}

\begin{proposition} Zeroing of telescopes is well-formed
  \[    \inferrule*[fraction={-{\,-\,}-}]{\Gamma \yields \Delta \tele}
                                         {\zc{\Gamma} \yields \zc{\Delta} \tele}
  \]
\end{proposition}

\begin{proof}
    \begin{itemize}
  \item In the case for $\cdot$, then $\zc{\Gamma} \yields \cdot \tele$ is well-formed.

  \item In the case for $\Delta, x : A \ctx$, we have $\Gamma \yields
    \Delta \tele$ and $\Gamma,\Delta \yields A \type$, so applying the
    IH gives $\zc{\Gamma} \yields \zc{\Delta} \tele$ and
    \rulen{pre-counit} for the type gives us $\zc{\Gamma,\Delta} \yields
    \zA \type$.  By idempotence, $\zc{\zc{\Gamma},\zc{\Delta}} \yields
    \zA \type$ as well, so so $\zc{\Gamma}\yields \zc{\Delta}, \zx ::
    \zA \ctx$ is well-formed.

  \item In the case for $\Gamma \yields \Delta, \zx :: A \tele$, we have
    $\Gamma \yields \Delta \tele$ and $\zc{\Gamma,\Delta} \yields A
    \type$.  By the IH $\zc{\Gamma} \yields \Delta \tele$.  By
    idempotence $\zc{\Gamma},\zc{\Delta} \defeq
    \zc{\zc{\Gamma},\zc{\Delta}}$, so $\zc{\zc{\Gamma},\zc{\Delta}}
    \yields A \type$ as well, and $\zc{\Gamma} \yields \zc{\Delta}, \zx
    :: A \ctx$ is well-formed.
  \end{itemize}
\end{proof}

\begin{proposition} Precomposition with the counit on a telescope is well-formed
  \[
    \inferrule*[fraction={-{\,-\,}-}]{\Gamma \yields \Delta \tele}{\zc{\Gamma} \yields \Delta^{0\Gamma} \tele}
  \]
\end{proposition}
\begin{proof}
  \begin{itemize}
  \item In the case for $\cdot$ then $\zc{\Gamma} \yields \cdot \tele$ is well-formed.
  \item In the case for $\Delta, x : A \tele$, we have $\Gamma \yields
    \Delta \tele$ and $\Gamma, \Delta \yields A \type$. Inductively,
    $\zc{\Gamma} \yields \Delta^{0\Gamma} \tele$, and applying
    \rulen{pre-counit-gen} to $A$, we have $\zc{\Gamma},
    \Delta^{0\Gamma} \yields A^{0\Gamma} \type$. So we have the
    well-formed telescope $\zc{\Gamma} \yields \Delta^{0\Gamma}, x :
    A^{0\Gamma} \tele$.
  \item In the case for $\Delta, \zx :: A \tele$, we have $\Gamma
    \yields \Delta \tele$ and
    $\zc{\Gamma, \Delta} \yields A \type$.
    Inductively, $\zc{\Gamma} \yields \Delta^{0\Gamma} \tele$.
    By idempotence, $\zc{\Gamma, \Delta}$
    is equal to $\zc{\zc{\Gamma}, \Delta^{0\Gamma}}$, so the type formation
    premise is also
    $\zc{\zc{\Gamma}, \Delta^{0\Gamma}} \yields A \type$, so we can
    form $\zc{\Gamma} \yields \Delta^{0\Gamma}, \zx :: A \tele$.
  \end{itemize}
\end{proof}

\begin{proposition}
  The \rulen{pre-unit} rule is admissible.
  For all rules, assume $\Psi \ctx$ and $\Psi \yields \Gamma \tele$
  and $\Psi, \zc{\Gamma} \yields \Delta \tele$.
  \begin{mathpar}
    \inferrule*[left=pre-unit,fraction={-{\,-\,}-}]
    { \Psi,\zc{\Gamma},\Delta \yields a : A}
    {\Psi,\Gamma,\Delta \yields a : A} \and
    \inferrule*[fraction={-{\,-\,}-}]
    {\Psi,\zc{\Gamma},\Delta \yields a \defeq a' : A}
    {\Psi,\Gamma,\Delta \yields a \defeq a' : A} \and
    \inferrule*[fraction={-{\,-\,}-}]
               {      \Psi,\zc{\Gamma},\Delta \yields A \type}
               {\Psi,\Gamma,\Delta \yields  A \type} \and
    \inferrule*[fraction={-{\,-\,}-}]
               {      \Psi,\zc{\Gamma},\Delta \yields A \defeq A' \type}
               {\Psi,\Gamma,\Delta \yields A \defeq A' \type} \and
 \end{mathpar}
\end{proposition}
\begin{proof}

  All four are proved by mutual induction on derivations.
  First we check some representative cases of \rulen{pre-unit}.
  \rulen{pre-unit} and the analogous statement for judgemental equality
  are proved by mutual induction on typing/equality derivations.
  \begin{itemize}
  \item \rulen{var}: We distinguish cases on which of $\Psi$ or $\Delta$
    the variable $x$ is in---it cannot be in $\zc{\Gamma}$ because all
    variables in $\zc{\Gamma}$ are marked.  In the case for
    \[
    \inferrule*{\Psi \vdash A \type}
               {\Psi,x:A,\Psi',\zc{\Gamma},\Delta \yields x : A}
    \]
    we can simply reapply the rule.
    In the case for
    \[
    \inferrule*{\Psi,\zc{\Gamma},\Delta \vdash A \type}
               {\Psi,\zc{\Gamma},\Delta,x:A,\Delta' \yields x : A}
    \]
    the inductive hypothesis gives ${\Psi,\Gamma,\Delta \vdash A
      \type}$, so we reapply the rule.

  \item \rulen{var-zero}: We again distinguish cases on which of
    $\Psi,\zc{\Gamma},\Delta$ the variable was in.
    In the case for
    \[
    \inferrule*{\zc{\Psi} \vdash A \type}
               {\Psi,\zx::A,\Psi',\zc{\Gamma},\Delta \yields \zx : A}
    \]
    we can reapply the rule.  In the case for
    \[
    \inferrule*{\zc{\Psi,\zc{\Gamma},\Delta} \vdash A \type}
               {\Psi,\zc{\Gamma},\Delta,\zx::A,\Delta' \yields \zx : A}
    \]
    $\zc{\Psi,\zc{\Gamma},\Delta} = \zc{\Psi,\Gamma,\Delta}$ by
      idempotence, so we can reapply the rule with
      $\zc{\Psi,\Gamma,\Delta} \vdash A \type$ for the premise.
    But now it is possible $\zx \in \zc{\Gamma}$, and there are two
    cases, depending on whether $x$ is marked in $\Gamma$:
    \begin{itemize}


     \item If $x$ in was not marked, then we have the instance
       \[
       \inferrule*{ \zc{\Psi,\zc{\Gamma}} \yields {A'}^{0(\Psi,\Gamma)} \type}
                  {\Psi,\zc{\Gamma,x:A',\Gamma'},\Delta \vdash \zx : A}
       \]
       where $A'^{0(\Psi,\Gamma)} \defeq_\alpha A$.  Since zeroing
       commutes with weakening, $A =
       A'^{0(\Psi,\Gamma,x,\Gamma',\Delta)}$ as well, and we want to use
       \rulen{var-roundtrip}
       \[
       \inferrule*[left=var-roundtrip]{\Psi,\Gamma \vdash A' \type}
                                      {\Psi,\Gamma, x : A', \Gamma',\Delta \yields \zx_{A'} : A}
       \]
       For the premise, we need $\Psi,\Gamma \yields A' \type$, we
       which can obtain by inverting the assumption that $\Psi \yields
       \Gamma,x:A',\Gamma' \tele$

     \item
       If $\zx$ was marked, i.e. we have
       \[
       \inferrule*{ \zc{\Psi,\zc{\Gamma}} \yields A \type}
                  {\Psi,\zc{\Gamma,\zx:A,\Gamma'},\Delta \vdash \zx : A}
       \]
       then $\Psi,\Gamma,\zx:A,\Gamma',\Delta \yields \zx : A$ by
       \rulen{var-zero}, because we have
       $\zc{\Psi,\Gamma} \vdash A \type$ by idempotence on the
       premise.
    \end{itemize}

  \item \rulen{var-roundtrip}: We distinguish cases on which of $\Psi$
    or $\Delta$ the variable $x$ is in---again, it cannot be in
    $\zc{\Gamma}$ because all variables in $\zc{\Gamma}$ are marked.  In
    the case for
    \[
    \inferrule*{\Psi \vdash A \type}
               {\Psi,x:A,\Psi',\zc{\Gamma},\Delta \yields \zx : \zA}
    \]
    we can simply reapply the rule to get
    ${\Psi,x:A,\Psi',\zc{\Gamma},\Delta \yields \zx : \zA}$ --- note that
    the old and new contexts have the same sets of variables, so $\zA$
    is zeroed with respect to the same implicit set of variables in both cases.

    In the case for
    \[
    \inferrule*{\Psi,\zc{\Gamma},\Delta \vdash A \type}
               {\Psi,\zc{\Gamma},\Delta,x:A,\Delta' \yields \zx : \zA}
    \]
    the inductive hypothesis (the smaller telescope $\Delta$ is
    well-formed by inversion) gives ${\Psi,\Gamma,\Delta \vdash A \type}$, so we can reapply the rule.

  \item \rulen{$\natural$-form}: Suppose we have
    $\Psi, \zc{\Gamma}, \Delta \yields \natural A \type$ because
    $\zc{\Psi, \zc{\Gamma}, \Delta} \yields A \type$. By idempotence,
    the context can be rewritten
    $\zc{\Psi, \Gamma, \Delta} \yields A \type$. Reapplying the rule
    then gives $\Psi, \Gamma, \Delta \yields \natural A$ as required.

  \item \rulen{$\Pi$-form}: If the inputs are
    \begin{align*}
      \Psi, \zc{\Gamma}, \Delta &\yields A \type \\
      \Psi, \zc{\Gamma}, \Delta, x : A &\yields B \type
    \end{align*}
    then inductively
    \begin{align*}
      \Psi, \Gamma, \Delta &\yields A \type \\
      \Psi, \Gamma, \Delta, x : A &\yields B \type
    \end{align*}
    where for $B$ we have extended $\Delta$ to the well-typed
    telescope $\Psi, \zc{\Gamma} \yields \Delta, x : A \tele$. We can
    then reapply the rule.
  \end{itemize}
\end{proof}

  \begin{proposition} Precomposition with the unit for telescopes is admissible:
    \[
        \inferrule*[fraction={-{\,-\,}-}]
                   {\Psi \ctx \and \Psi \yields \Gamma \tele \and \Psi, \zc{\Gamma} \yields \Delta \tele}
                   {\Psi,\Gamma \yields \Delta \tele}
    \]
  \end{proposition}
\begin{proof}
  \begin{itemize}
  \item If $\Delta \defeq \cdot$ then $\Psi,\Gamma \yields \cdot \tele$
    is well-formed.
  \item Suppose $\Delta, x : A \tele$ because
    $\Psi, \zc{\Gamma} \yields \Delta \tele$ and
    $\Psi, \zc{\Gamma}, \Delta \yields A \type$. Inductively,
    $\Psi, \Gamma \yields \Delta \tele$ and applying \rulen{pre-unit}
    to $A$, we have $\Psi, \Gamma, \Delta \yields A \type$, so
    $\Psi, \Gamma \yields \Delta, x : A \tele$ is well-formed.
  \item Suppose $\Delta, \zx :: A \tele$ because
    $\Psi, \zc{\Gamma} \yields \Delta \tele$ and
    $\zc{\Psi, \zc{\Gamma}, \Delta} \yields A \type$. By the IH,
    $\Psi,\Gamma \vdash \Delta \tele$.
    By idempotence, $\zc{\Psi, \zc{\Gamma}, \Delta} $
     is equal to
    $\zc{\Psi, \Gamma, \Delta} \yields A \type$, so we can
    form $\Psi, \Gamma \yields \Delta, \zx :: A \tele$.
  \end{itemize}

\end{proof}

\begin{proposition} \label{prop:preunit-preserves-size}
\rulen{pre-unit} preserves the size (number of derivation steps) of a
typing derivation.
\end{proposition}
\begin{proof}
Inspection of the above proof.  Note that all of the equality reasoning
involving idempotence etc.\ is about syntactic equality $\defeq_\alpha$,
which does not get recorded in typing derivations.
\end{proof}

\begin{proposition}
  Substitution is admissible.  For all rules, assume $\Gamma \ctx$ and
  $\Gamma \yields A \type$ and $\Gamma,x:A \yields \Delta \tele$.
  \begin{mathpar}
    \inferrule*[fraction={-{\,-\,}-}]
    {\Gamma \yields a : A \and \Gamma, x : A, \Delta \yields b : B}
    {\Gamma, \Delta[a/x] \yields b[a/x] : B[a/x] }
    \quad
    \inferrule*[fraction={-{\,-\,}-}]
    {\Gamma \yields a : A \and \Gamma, x : A, \Delta \yields B \type}
    {\Gamma, \Delta[a/x] \yields B[a/x] \type }
    \\
    \inferrule*[fraction={-{\,-\,}-}]
               {\Gamma \yields a : A \and \Gamma, x : A, \Delta \yields b \defeq b': B}
               {\Gamma, \Delta[a/x] \yields b[a/x] = b'[a/x] : B[a/x] }
    \quad
    \inferrule*[fraction={-{\,-\,}-}]
               {\Gamma \yields a : A \and \Gamma, x : A, \Delta \yields B \defeq B' \type}
               {\Gamma, \Delta[a/x] \yields B[a/x] \defeq B'[a/x] \type }
  \end{mathpar}
\end{proposition}
\begin{proof}
  The first new interesting case is substitution for a roundtripped variable. The inputs are
  \begin{align*}
    \Gamma &\yields a : A \\
    \Gamma, x : A, \Delta &\yields \zx : \zA
  \end{align*}
  and so applying the \rulen{pre-roundtrip} on $a$ gives
  $\Gamma \yields \za : \zA$, which can be weakened to
  $\Gamma, x : A, \Delta \yields \za : \zA$ and taken as the result.

  The other new kind of case is substitution into a rule that has a
  zeroed context as a premise. For example, suppose we are in the case for
  \begin{align*}
    \Gamma &\yields a : A \\
    \Gamma, x : A, \Delta &\yields \natural B \type
  \end{align*}
  because $\zc{\Gamma, x : A, \Delta} \yields B \type$. Then applying
  \rulen{pre-unit} (note that $\zc{\Gamma} \ctx$ and $\zc{\Gamma}
  \yields \zc{A} \type$ and $\zc{\Gamma},x:\zA \yields \zc{\Delta}
  \tele$ by \rulen{ctx-zero} and \rulen{pre-counit} and
  zeroing/\rulen{pre-unit} for telescopes), we have $\zc{\Gamma}, x :
  \zc{A}, \zc{\Delta} \yields B \type$, and by \rulen{pre-counit} we
  have $\zc{\Gamma} \vdash \za : \zA$.  Because \rulen{pre-unit}
  preserves the size of a derivation, we can appeal to the inductive
  hypothesis to form $\zc{\Gamma}, \zc{\Delta}[\za/x] \yields B[\za/x]
  \type$. Because zeroing commutes with substitution,
  $\zc{\Gamma},\zc{\Delta}[\za/x]$ is equal to
  $\zc{\Gamma,\Delta[a/x]}$, and $B[\za/x]$ is equal to $B[a/x]$.  So
  reapplying the rule gives $\Gamma, \Delta[a/x] \yields
  \natural(B[a/x]) \type$.
\end{proof}

\section{Categorical Semantics}\label{sec:semantics}

\subsection{Models}

We now describe a notion of model for our type theory. Categories with
families (CwFs)~\cite{dybjer:cwf} are a standard notion of model for
dependent type theories, and we begin by recalling their definition.
We generally follow the same notation for CwFs as in~\cite{drats}.

\newcommand{\cat}[1]{\mathcal{#1}}
\newcommand{\C}{\cat{C}}
\newcommand{\D}{\cat{D}}
\newcommand{\co}[2]{#1{.}#2}
\newcommand{\coo}[2]{#1{..}#2}
\newcommand{\CWF}{\C}
\newcommand{\Elt}[3][\CWF]{#1(#2\mathbin{\scriptstyle\vdash}#3)}
\newcommand{\Fam}[2][\CWF]{#1(#2)}
\newcommand{\FamU}[3][\C]{#1(#2,{#3})}
\newcommand{\gen}{\mathsf{q}}
\newcommand{\Hom}[3][\CWF]{#1(#2,#3)}
\newcommand{\projsub}{\mathsf{p}}
\newcommand{\pair}[2]{\left(#1,#2\right)}

\newcommand{\den}[1]{\llbracket#1\rrbracket}
\newcommand{\bigden}[1]{\left\llbracket#1\right\rrbracket}
\newcommand{\transp}[1]{\overline{#1}}
\newcommand{\kleq}{\approx}

\newcommand{\Rt}{\mathsf{R}}
\newcommand{\Subst}{\mathsf{S}}
\newcommand{\Wk}{\mathsf{P}}
\newcommand{\ZWk}{\zc{\mathsf{P}}}

\begin{definition}\label{def:cwf}
  A \emph{category with families} (CwF) is specified by:
  \begin{enumerate}
  \item A category $\C$ with a terminal object $1$.

  \item For each object $\Gamma\in\C$, a set $\Fam{\Gamma}$ of
    \emph{families} over $\Gamma$.

  \item For each object $\Gamma\in\C$ and family $A\in\Fam{\Gamma}$, a
    set $\Elt{\Gamma}{A}$ of \emph{elements} of the family $A$ over
    $\Gamma$.

  \item\label{item:2} For each morphism
    $\theta\in\Hom{\Delta}{\Gamma}$, \emph{re-indexing}
    functions $A\in\Fam{\Gamma} \mapsto A[\theta]\in\Fam{\Delta}$ and
    $a\in\Elt{\Gamma}{A}\mapsto a[\theta]\in\Elt{\Delta}{A[\theta]}$,
    satisfying $A[\id]=A$, $A[\theta\circ\delta] = A[\theta][\delta]$,
    $a[\id]=a$ and $a[\theta\circ\delta] = a[\theta][\delta]$.

  \item\label{item:1} For each object $\Gamma\in\C$ and family
    $A\in\Fam{\Gamma}$, a \emph{comprehension object}
    $\co{\Gamma}{A}\in\C$ equipped with a \emph{projection morphism}
    $\projsub_A\in\C(\co{\Gamma}{A},\Gamma)$, a \emph{generic element}
    $\gen_A\in\Elt{\co{\Gamma}{A}}{A[\projsub_A]}$ and a \emph{pairing
      operation}
    $\theta\in\Hom{\Delta}{\Gamma}, a\in\Elt{\Delta}{A[\theta]}
    \mapsto \pair{\theta}{a}\in\C(\Delta,\co{\Gamma}{A})$ satisfying
    $\projsub_A\circ\pair{\theta}{a}= \theta$,
    $\gen_A[\pair{\theta}{a}] = a$,
    $\pair{\theta}{a}\circ\delta =
    \pair{\theta\circ\delta}{a[\delta]}$ and
    $\pair{\projsub_A}{\gen_A} = \id$.
  \end{enumerate}
\end{definition}

\begin{definition}\label{def:cwf-pi}
A \emph{$\Pi$-structure} on a CwF consists of, for each $\Gamma \in C$, $A \in \Fam{\Gamma}$, $B \in \Fam{\co{\Gamma}{A}}$, a family $\Pi(A, B) \in \Fam{\Gamma}$ and functions
\begin{align*}
\lambda_{\Gamma,A,B} &: \Elt{\co{\Gamma}{A}}{B} \to \Elt{\Gamma}{\Pi(A, B)} \\
\mathrm{app}_{\Gamma, A, B} &: \Elt{\Gamma}{\Pi(A, B)} \to \prd{a \in \Elt{\Gamma}{A}} \Elt{\Gamma}{B[\pair{\id_\Gamma}{a}]}
\end{align*}
such that
\begin{align*}
\Pi(A, B)[\theta] &\defeq \Pi(A[\theta], B[\pair{\theta\circ \projsub}{\gen}]) \\
\lambda_{\Gamma,A,B}(b)[\theta] &\defeq \lambda_{\Delta, A[\theta], B[\pair{\theta\circ \projsub}{\gen}]}(b[\pair{\theta\circ \projsub}{\gen}]) \\
\mathrm{app}_{\Gamma, A, B}(f, a)[\theta] &\defeq \mathrm{app}_{\Delta, A[\theta], B[\pair{\theta\circ \projsub}{\gen}]}(f[\theta], a[\theta]) \\
\mathrm{app}_{\Gamma, A, B}(\lambda_{\Gamma,A,B}(b), a) &\defeq b[\pair{\id_\Gamma}{a}] \\
\lambda_{\Gamma,A,B}(\mathrm{app}_{\Gamma.A, A, B}(b[\projsub_A]), \gen_A) &\defeq b
\end{align*}
\end{definition}

$\Pi$-types are not required to interpret our modality, but we will show that the interpretation of marked variables and zeroing does not interfere with that of $\Pi$-types. The other standard type formers are similar.

\begin{definition}\label{def:weak-cwf-mor}
  A \emph{weak CwF morphism} $F$ between CwFs consists of a functor $F : \C\to\D$, an operation on families mapping $A\in\Fam{\Gamma}$ to a family $F A\in\Fam[\D]{F\Gamma}$ and an operation on elements mapping $a\in\Elt{\Gamma}{A}$ to an element $F a\in\Elt[\D]{F\Gamma}{F A}$, such that
 \begin{enumerate}
\item The functor $F : \C\to\D$ preserves terminal objects (up to isomorphism)
\item The operations on families and elements commute with reindexing in the sense that
$F A [F \theta] = F (A[\theta])$ and $F t[F \theta] = F (t[\theta])$.
\item The maps $\pair{F \projsub_A}{F \gen_A} : F(\co{\Gamma}{A}) \to \co{F\Gamma}{F A}$
are isomorphisms for all $\Gamma$ and $A$. We write $\nu_{\Gamma,A}$ for the inverse.
\end{enumerate}
\end{definition}

\begin{lemma}
For any weak CwF morphism, the following equations hold:
\begin{align}
F(\projsub_A) \circ \nu_{\Gamma,A} &\defeq \projsub_{FA}\label{eq:proj-nu} \\
F(\gen_A)[\nu_{\Gamma,A}] &\defeq \gen_{FA}\label{eq:gen-nu} \\
\nu_{\Gamma,A} \circ (F\theta, F a) &\defeq F(\theta, a)\label{eq:nu-sub}
\end{align}
\end{lemma}

We now come to our notion of model for our type theory.

\begin{definition}
  A \emph{category with families with a bireflector} (CwB) consists of a CwF
  $\C$ with a weak endomorphism $N : \C \to \C$, such that there are
  natural transformations $\eta : \id_\C \Rightarrow N$ and
  $\varepsilon : N \Rightarrow \id_\C$ for the underlying functor of
  $N$, such that $\eta \circ \varepsilon \defeq \id_{N-}$.
\end{definition}

\begin{remark}
  The CwF structure aside, this is a category equipped with a
  `bireflective subcategory' as studied in~\cite{bireflectivity}. The
  definition we have given here uses an equivalent characterisation of
  bireflectivity~\cite[Theorem 13]{bireflectivity}, as a
  `split-idempotent natural transformation on $\id_\C$ with specified
  splitting'.
\end{remark}

There is some derived structure in any CwB $\C$.

\begin{lemma}
  For any $\Gamma \in \C$,
  \begin{align*}
    N\eta_\Gamma &\defeq \eta_{N\Gamma} \\
    N\varepsilon_\Gamma &\defeq \varepsilon_{N\Gamma}
  \end{align*}
\end{lemma}
\begin{proof}
  Naturality gives that
  \[ N\eta_\Gamma \defeq
    N\eta_\Gamma \circ \eta_\Gamma \circ \varepsilon_\Gamma \defeq
    \eta_{N\Gamma} \circ \eta_\Gamma \circ \varepsilon_\Gamma \defeq
    \eta_{N\Gamma} \]
  The counit is similar.
\end{proof}

\begin{proposition}
  $N$ is an idempotent monad and comonad, with multiplication and
  comultiplication given by
  \begin{align*}
    \mu_\Gamma &:\defeq N\varepsilon_{\Gamma} \\
    \delta_\Gamma &:\defeq N\eta_{\Gamma} \\
  \end{align*}
\end{proposition}
\begin{proof}
  The associativity law follows because
  \[ N\mu_\Gamma \defeq NN\varepsilon_{\Gamma} \defeq
    N\varepsilon_{N\Gamma}\defeq \mu_{N\Gamma}. \]
  Unit laws follow by naturality:
  \[ \mu_\Gamma \circ \eta_{N\Gamma} \defeq N\varepsilon_{\Gamma}
    \circ \eta_{N\Gamma} \defeq \eta_{\Gamma} \circ
    \varepsilon_{\Gamma} \defeq \id_{N\Gamma} \] with the other
  holding because $\eta_{N\Gamma} \defeq N\eta_{\Gamma}$. For
  idempotence, we have just seen that $\eta_{N\Gamma}$ is a right
  inverse to $\mu_\Gamma$, and the other direction follows by assumption:
  \[ \eta_{N\Gamma} \circ \mu_\Gamma \defeq \eta_{N\Gamma} \circ
    N\varepsilon_{\Gamma} \defeq \eta_{N\Gamma} \circ
    \varepsilon_{N\Gamma} \defeq \id_{NN\Gamma} \]

  The comonad structure follows by symmetry.
\end{proof}

\begin{remark}
A CwB is also a `CwF+A'~\cite[Definition 16]{drats}, where the left and
right adjoints are both $N$. By~\cite[Lemma 17]{drats}, any CwB is also
a category with a `dependent right adjoint'~\cite[Definition
  2]{drats}. The action of this `dependent right adjoint' on types is
not quite the action of $N$ coming from the weak CwF morphism, instead
for $\Gamma \in \C$ and $A \in \Fam{N \Gamma}$ the `right adjoint type'
$ N_\Gamma A \in \Fam{\Gamma}$  is given by
\[ N_\Gamma A :\defeq (N A)[N\eta_\Gamma \circ \eta_\Gamma]. \] Note
that $\eta$ as we have defined it is the unit of $N$ as a monad, so
that $N\eta_\Gamma \circ \eta_\Gamma$, is the unit of the $N \dashv N$
adjunction.  There is then a bijection of elements
\[ \Elt{N \Gamma}{A} \cong \Elt{\Gamma}{N_\Gamma A}. \] We could base
our interpretation of the type theory around this bijection, but it
does not match nicely with our judgemental structure: our operation
applying $N$ to an entire context is an admissible one $\zc{\Gamma}$,
which is defined in terms of a special context extension, so we would
need to define and use an interpretation of the admissible rule.
\end{remark}


In our interpretation, we collapse the distinction between modal
variables and variables of type $\natural A$, so that `modal context
extension' by $A$ becomes ordinary context extension by $\natural A$.
The \rulen{var-zero} rule then becomes an application of the semantic
counit map `$\natural A \to A$', or just \rulen{$\natural$-elim}
applied to an ordinary variable, which we capture in the following
definition.

\begin{definition}
If $A \in \Fam{\Gamma}$, let $\zc{\gen}_A \in \Elt{\Gamma.(NA)[\eta_\Gamma]}{A[\varepsilon_\Gamma \circ \eta_\Gamma][\projsub_{(NA)[\eta_\Gamma]}]}$ be defined by
\begin{align*}
\zc{\gen}_A :\defeq \gen_A[\varepsilon_{\Gamma.A} \circ \nu_{\Gamma, A} \circ \pair{\eta_{\Gamma} \circ \projsub_{(NA)[\eta_\Gamma]}}{\gen_{(NA)[\eta_{\Gamma}]}}]
\end{align*}
where we have calculated the type to be
\begin{align*}
A[\projsub_A][\varepsilon_{\Gamma.A} \circ \nu_{\Gamma, A} \circ \pair{\eta_{\Gamma} \circ \projsub_{(NA)[\eta_\Gamma]}}{\gen_{(NA)[\eta_{\Gamma}]}}]
&\defeq A[\varepsilon_{\Gamma} \circ N(\projsub_A) \circ \nu_{\Gamma, A} \circ \pair{\eta_{\Gamma} \circ \projsub_{(NA)[\eta_\Gamma]}}{\gen_{(NA)[\eta_{\Gamma}]}}] \\
&\defeq A[\varepsilon_{\Gamma} \circ \projsub_{NA} \circ \pair{\eta_{\Gamma} \circ \projsub_{(NA)[\eta_\Gamma]}}{\gen_{(NA)[\eta_{\Gamma}]}}] \\
&\defeq A[\varepsilon_{\Gamma} \circ \eta_{\Gamma} \circ \projsub_{(NA)[\eta_\Gamma]}]
\end{align*}
\end{definition}

\begin{lemma}\label{lem:mod-var-sub}
The counit term satisfies the following equations:
\begin{align}
\zc{\gen}_A[\pair{\id_\Gamma}{(Nn)[\eta_\Gamma]}] &\defeq n[\varepsilon_{\Gamma} \circ \eta_{\Gamma}]\label{eq:zq-beta} \\
(N\zc{\gen}_A)[\eta_{\Gamma.(NA)[\eta_\Gamma]}] &\defeq \gen_{(NA)[\eta_{\Gamma}]}\label{eq:zq-eta}\\
\zc{\gen}_A[\varepsilon_{\Gamma.(NA)[\eta_\Gamma]} \circ \eta_{\Gamma.(NA)[\eta_\Gamma]}] &\defeq \zc{\gen}_A\label{eq:zq-idem} \\
\zc{\gen}_A[\pair{\theta \circ \projsub_{(NA[\theta])[\eta]}}{\gen_{(NA[\theta])[\eta]}}] &\defeq \zc{\gen}_{A[\theta]}\label{eq:zq-sub}
\end{align}
\end{lemma}
Thinking of $\zc{\gen}_A$ as the interpretation of `$x_\natural$',
from top to bottom, this is a $\beta$-rule
`$x_\natural[a^\natural/x] \defeq \za$', an $\eta$-rule
`$(x_\natural)^\natural \defeq x$' idempotence
`$\zc{x_\natural} \defeq x_\natural$', and stability under
substitution for other variables
`$x_\natural[\theta] \defeq x_\natural$' if $\theta$ misses $x$.
\begin{proof}
The $\beta$-rule:
\begin{align*}
\zc{\gen}_A[\pair{\id_\Gamma}{(Nn)[\eta_\Gamma]}]
&\defeq \gen_A[\varepsilon_{\Gamma.A} \circ \nu_{\Gamma, A} \circ \pair{\eta_{\Gamma} \circ \projsub_{(NA)[\eta_\Gamma]}}{\gen_{(NA)[\eta_{\Gamma}]}}][\pair{\id_\Gamma}{(Nn)[\eta_\Gamma]}] \\
&\qquad \text{(Substitution into a substitution extension)} \\
&\defeq \gen_A[\varepsilon_{\Gamma.A} \circ \nu_{\Gamma, A} \circ \pair{\eta_{\Gamma} \circ \projsub_{(NA)[\eta_\Gamma]} \circ \pair{\id_\Gamma}{(Nn)[\eta_\Gamma]}}{\gen_{(NA)[\eta_{\Gamma}]}[\pair{\id_\Gamma}{(Nn)[\eta_\Gamma]}]}] \\
&\qquad \text{($\beta$-rules for substitution extension)} \\
&\defeq \gen_A[\varepsilon_{\Gamma.A} \circ \nu_{\Gamma, A} \circ \pair{\eta_{\Gamma}}{(Nn)[\eta_\Gamma]}] \\
&\qquad \text{(Substitution into a substitution extension)} \\
&\defeq \gen_A[\varepsilon_{\Gamma.A} \circ \nu_{\Gamma, A} \circ \pair{N\id_\Gamma}{Nn} \circ \eta_{\Gamma}] \\
&\qquad \text{(Equation~\ref{eq:nu-sub})} \\
&\defeq \gen_A[\varepsilon_{\Gamma.A} \circ N\pair{\id_\Gamma}{n} \circ \eta_{\Gamma}] \\
&\qquad \text{(Naturality of $\varepsilon$)} \\
&\defeq \gen_A[\pair{\id_\Gamma}{n} \circ \varepsilon_{\Gamma} \circ \eta_{\Gamma}] \\
&\qquad \text{($\beta$ for variable rule)} \\
&\defeq n[\varepsilon_{\Gamma} \circ \eta_{\Gamma}]
 \end{align*}
The $\eta$-rule:
\begin{align*}
(N\zc{\gen}_A)[\eta_{\Gamma.(NA)[\eta_\Gamma]}]
&\defeq (N(\gen_A[\varepsilon_{\Gamma.A} \circ \nu_{\Gamma, A} \circ \pair{\eta_{\Gamma} \circ \projsub_{(NA)[\eta_\Gamma]}}{\gen_{(NA)[\eta_{\Gamma}]}}]))[\eta_{\Gamma.(NA)[\eta_\Gamma]}] \\
&\qquad \text{($N$ commutes with substitution into terms)} \\
&\defeq (N\gen_A)[N(\varepsilon_{\Gamma.A} \circ \nu_{\Gamma, A} \circ \pair{\eta_{\Gamma} \circ \projsub_{(NA)[\eta_\Gamma]}}{\gen_{(NA)[\eta_{\Gamma}]}}) \circ \eta_{\Gamma.(NA)[\eta_\Gamma]}] \\
&\qquad \text{(Naturality of $\eta$)} \\
&\defeq (N\gen_A)[\eta_{\Gamma.A} \circ \varepsilon_{\Gamma.A} \circ \nu_{\Gamma, A} \circ \pair{\eta_{\Gamma} \circ \projsub_{(NA)[\eta_\Gamma]}}{\gen_{(NA)[\eta_{\Gamma}]}}] \\
&\qquad \text{($\eta \circ \varepsilon$ is the identity)} \\
&\defeq (N\gen_A)[\nu_{\Gamma, A} \circ \pair{\eta_{\Gamma} \circ \projsub_{(NA)[\eta_\Gamma]}}{\gen_{(NA)[\eta_{\Gamma}]}}] \\
&\qquad \text{(Equation~\ref{eq:gen-nu})} \\
&\defeq \gen_{NA}[\pair{\eta_{\Gamma} \circ \projsub_{(NA)[\eta_\Gamma]}}{\gen_{(NA)[\eta_{\Gamma}]}}] \\
&\qquad \text{($\beta$-rule for variables)} \\
&\defeq \gen_{(NA)[\eta_{\Gamma}]}
\end{align*}
Idempotence:
\begin{align*}
&\zc{\gen}_A[\varepsilon_{\Gamma.(NA)[\eta_\Gamma]} \circ \eta_{\Gamma.(NA)[\eta_\Gamma]}] \\
&\qquad \text{(Definition of $\zc{\gen}$)} \\
&\defeq \gen_A[\varepsilon_{\Gamma.A} \circ \nu_{\Gamma, A} \circ \pair{\eta_{\Gamma} \circ \projsub_{(NA)[\eta_\Gamma]}}{\gen_{(NA)[\eta_{\Gamma}]}}][\varepsilon_{\Gamma.(NA)[\eta_\Gamma]} \circ \eta_{\Gamma.(NA)[\eta_\Gamma]}] \\
&\qquad \text{(Naturality of $\varepsilon$ and $\eta$)} \\
&\defeq \gen_A[\varepsilon_{\Gamma.A} \circ \eta_{\Gamma.A} \circ \varepsilon_{\Gamma.A} \circ \nu_{\Gamma, A} \circ \pair{\eta_{\Gamma} \circ \projsub_{(NA)[\eta_\Gamma]}}{\gen_{(NA)[\eta_{\Gamma}]}}] \\
&\qquad \text{($\eta \circ \varepsilon$ is the identity)} \\
&\defeq \gen_A[\varepsilon_{\Gamma.A} \circ \nu_{\Gamma, A} \circ \pair{\eta_{\Gamma} \circ \projsub_{(NA)[\eta_\Gamma]}}{\gen_{(NA)[\eta_{\Gamma}]}}] \\
&\qquad \text{(Definition of $\zc{\gen}$)} \\
&\defeq \zc{\gen}_A
\end{align*}
Stability under substitution:
\begin{align*}
&\zc{\gen}_A[\pair{\theta \circ \projsub_{(NA[\theta])[\eta]}}{\gen_{(NA[\theta])[\eta]}}] \\
&\qquad \text{(Equation~\ref{eq:zq-eta})} \\
&\defeq \zc{\gen}_A[\pair{\theta \circ \projsub_{(NA[\theta])[\eta]}}{(N\zc{\gen}_{A[\theta]})[\eta_{\Gamma.(NA[\theta])[\eta_\Gamma]}]}] \\
&\qquad \text{(Equation~\ref{eq:zq-beta})} \\
&\defeq \zc{\gen}_{A[\theta]}[\varepsilon_{\Delta.(NA[\theta])[\eta_\Delta]} \circ \eta_{\Delta.(NA[\theta])[\eta_\Delta]}] \\
&\qquad \text{(Equation~\ref{eq:zq-idem})} \\
&\defeq \zc{\gen}_{A[\theta]}
\end{align*}
\end{proof}

We record an easy Lemma stating that the non-trivial roundtrip has no effect on terms of $\natural A$.
\begin{lemma}\label{lem:n-rt-triv}
If $n \in \Elt{\Gamma}{(N A)[\eta_\Gamma]}$ then $n \defeq n[\varepsilon_\Gamma \circ \eta_\Gamma]$
\end{lemma}
\begin{proof}
\begin{align*}
n
&\defeq \gen_{(N A)[\eta_\Gamma]}[\pair{\id_\Gamma}{n}] \\
&\defeq (N\zc{\gen}_A)[\eta_{\Gamma.(NA)[\eta_\Gamma]}] [\pair{\id_\Gamma}{n}] && \text{(Equation~\ref{eq:zq-eta})} \\
&\defeq (N(\zc{\gen}_A[\varepsilon_{\Gamma.(NA)[\eta_\Gamma]} \circ \eta_{\Gamma.(NA)[\eta_\Gamma]}]))[\eta_{\Gamma.(NA)[\eta_\Gamma]}] [\pair{\id_\Gamma}{n}] && \text{(Equation~\ref{eq:zq-idem})} \\
&\defeq (N\zc{\gen}_A)[\eta_{\Gamma.(NA)[\eta_\Gamma]}][\pair{\id_\Gamma}{n}][\varepsilon_{\Gamma} \circ \eta_{\Gamma}] && \text{(Naturality)} \\
&\defeq \gen_{(N A)[\eta_\Gamma]}[\pair{\id_\Gamma}{n}][\varepsilon_{\Gamma} \circ \eta_{\Gamma}] && \text{(Equation~\ref{eq:zq-eta})} \\
&\defeq n[\varepsilon_{\Gamma} \circ \eta_{\Gamma}]
\end{align*}
\end{proof}

\subsection{Interpretation of Syntax}


\begin{figure}
\begin{align*}
\den{\cdot} &:\defeq 1 \\
\den{\Gamma, x : A} &:\defeq \co{\den{\Gamma}}{\den{\Gamma \yields A}} \\
\den{\Gamma, \zx :: A} &:\defeq \co{\den{\Gamma}}{(N\den{\Gamma \yields A})[\eta_{\den{\Gamma}}]} && \text{if $A \defeq_\alpha \zA$} \\
\den{\Gamma,x:A,x_1:_?A_1,\ldots,x_n:_?A_n \yields x} &:\defeq \gen_{\den{\Gamma \yields A}}[\projsub_?{}_{\den{A_1}}\circ\cdots\circ\projsub_?{}_{\den{A_n}}] \\
\den{\Gamma,x:A,x_1:_?A_1,\ldots,x_n:_?A_n \yields \zx} &:\defeq \gen_{\den{\Gamma \yields A}}[\varepsilon_{\den{\Gamma, x : A}} \circ \eta_{\den{\Gamma, x : A}}][\projsub_?{}_{\den{A_1}}\circ\cdots\circ\projsub_?{}_{\den{A_n}}] \\
\den{\Gamma,\zx::A,x_1:_?A_1,\ldots,x_n:_?A_n \yields \zx} &:\defeq \zc{\gen}_{\den{\Gamma \yields A}}[\projsub_?{}_{\den{A_1}}\circ\cdots\circ\projsub_?{}_{\den{A_n}}] \\
\den{\Gamma \yields \natural A} &:\defeq (N\den{\Gamma \yields A})[\eta_{\den{\Gamma}}] \\
\den{\Gamma \yields a^\natural_A} &:\defeq (N_{\den{\Gamma},\den{\Gamma \yields A}}\den{\Gamma \yields a})[\eta_{\den{\Gamma}}] \\
\den{\Gamma \yields n_\natural^A} &:\defeq \zc{\gen}_{\den{\Gamma \yields A}}[\pair{\id_{\den{\Gamma}}}{\den{\Gamma \yields n}}]\\
\den{\Gamma \yields \Pi x:A.B} &:\defeq \Pi(\den{\Gamma \yields A}, \den{\Gamma,x:A \yields B})\\
\den{\Gamma \yields \lambda x:A.b:B} &:\defeq \lambda_{\den{\Gamma},\den {\Gamma \yields A},\den {\Gamma,x:A \yields B}} (\den{\Gamma,x:A \yields b})\\
\den{\Gamma \yields f(a)_{A,x.B}} &:\defeq \mathrm{app}_{\den{\Gamma},\den{\Gamma \yields A},\den{\Gamma,x:A \yields B}} (\den{\Gamma \yields f}, \den{\Gamma \yields a} )
\end{align*}
where $x_1 :_? A_1$ denotes either $x_1 : A_1$ or $\zx_1 :: A_1$, and $\projsub_?{}_{\den{A_1}}$ denotes $\projsub_{\den{A_1}}$ or $\projsub_{(N\den{\Gamma \yields A_1})[\eta_\Gamma]}$ accordingly.
\caption{Partial Interpretation}
\label{fig:partial-interp}
\end{figure}

We now show that the rules of our type theory can be soundly interpreted
in any CwB.  The exact form of the induction we use for the partial
interpretation hews closer to~\cite[Section 3.5]{hofmann:interp}
than~\cite{drats}: we define the interpretation on fully-annotated raw
syntax as in Section~\ref{sec:admissible-rules}, and interpret raw terms
prior to knowing what family they will live in. The partial
interpretation function $\den{-}$
\begin{enumerate}
\item Maps $\Gamma \rawctx$ to objects of $\C$;
\item For $\Gamma \rawctx$ with $\den{\Gamma}$ defined, maps $dom(\Gamma) \yields A \rawterm$ to a family in $\Fam{\den{\Gamma}}$; and,
\item For $\Gamma \rawctx$ with $\den{\Gamma}$ defined, maps $dom(\Gamma) \yields a \rawterm$ to (pairs $T \in \Fam{\den{\Gamma}}$ and elements in $\Elt{\den{\Gamma}}{T}$).
\end{enumerate}
The partial interpretation of the syntax is defined by mutual induction
in Figure~\ref{fig:partial-interp}.  The induction metric is the size of
the context for part 1, the size of the type for part 2, and the size of
the term for part 3.
The soundness theorem will demonstrate that these assignments are in
fact total on well-typed contexts, types and terms, and that the
interpretation of a well-typed term $\Gamma \yields a : A$ is an element
of the correct family $\den{A} \in \Fam{\den{\Gamma}}$.  We now explain
some of the conventions and subtleties in this definition:
\begin{itemize}
\item
Our interpretation function follows the shape of the
`dull-variable-less' rules of
Figure~\ref{fig:rules-no-marked-context-extension}. This lets us avoid
unnecessary manoeuvring between the syntactic and semantic unit and
counit, instead only requiring the roundtrip to match up
(Lemma~\ref{lem:sem-roundtrip}). This choice requires using the
admissible \rulen{pre-unit} rules in the proof of totality of the
interpretation function, which in turn requires that proof to be by
induction on the \emph{size} of a derivation, the total number of rules
used in forming it.  This works because \rulen{pre-unit} preserves size
by Proposition~\ref{prop:preunit-preserves-size}.  For example, in the
rule for $\den{\Gamma, \zx :: A}$, we write $\den{\Gamma \yields A}$ on
the right hand side, which differs from the typing premise $\zc{\Gamma}
\yields A$.  This will be handled in the totality proof by the
admissible unzeroing rule \rulen{pre-unit}.  The same move also happens
for $\den{\Gamma \yields \natural A}$ and $\den{\Gamma \yields
  a^\natural : \natural A}$.

\item
  Following \cite{deboer20initiality}, we formally think of a partial
  function as a function to partial elements of the codomain,
  representing a partial element of $Z$ by a pair of a proposition
  $\phi$ and function $\phi \to Z$.

\item
  We write the action of $\den{-}$ on pairs $(\Gamma \rawctx,
  dom(\Gamma) \yields A \rawterm)$ as $\den{\Gamma \yields A}$ to make
  clear the separation between the context and the term, or just as
  $\den{A}$ if the context is clear, and similarly for terms.  We have
  to be careful with these abbreviations, though, because the semantics
  is ``intrinsically typed'', so semantic types in different semantic
  contexts are never equal (and similarly for terms).  For example, in
  syntax, the same raw type can satisfy $\Gamma \yields A \type$ and
  $\zc{\Gamma} \yields A \type$, but $\den{\Gamma \yields A} \in
  \C(\Gamma)$ and $\den{\zc{\Gamma} \yields A} \in \C(\zc{\Gamma})$
  (though there is an explicit unit operation taking one to the other).

\item
We make definedness of $\den{\Gamma}$ a precondition for the function of
$\den{\Gamma \yields A}$, so it only makes sense to write the latter
when we have established the former.  This follows the presupposition
modes of the judgements, where the context is assumed to be well-formed
but the type is implied to be well-formed.

\item
We adopt the convention that a right-hand side is defined only if the
interpretation of each subexpression in it is. I.e., we implicitly lift
semantic operations on total elements to partial elements, rather than
thinking of the semantic operations like $N$, $\lambda$ etc. as being
defined on partial elements.  In particular this means that when
$\den{\Gamma}$ is defined, implies $\den{\Gamma^{<x} \yields A}$ is
defined for each $x:A$ in $\Gamma$.

\item
On terms, the interpretation function is supposed to provide both a
family $T \in \Fam{\den{\Gamma}}$ and an element of that family. In the
defining equations above, we only provide the second component and leave
the first component to be inferred.

\item
We adopt the convention that any semantic expression that does not
`type-check' is undefined. This is essential in the interpretation of
\rulen{$\natural$-elim}, for example, where we semantically substitute
$\den{\Gamma \yields n}$ into $\zc{\gen}_{\den{\Gamma \yields A}}$: a
priori we do not know that the family $T$ such that $\den{\Gamma \yields
  n} \in \Elt{\den{\Gamma}}{T}$ lines up with the family $(N\den{\Gamma
  \yields A})[\eta_{\den{\Gamma}}] \in \Fam{\den{\Gamma}}$.  Similarly,
for \rulen{$\natural$-intro}, because we have written
$N_{\den{\Gamma},\den{\Gamma \yields A}}$, the interpretation is defined
only if the family $T$ for which $\den{a} \in \C(\den{\Gamma} \yields T)$ is in fact
$\den{\Gamma \yields A}$.
\end{itemize}


For the next several lemmas, which state that syntactic and semantic
roundtrip/weakening/substitution agree, we use an unpacked statement
that is equivalent to those used in \cite[Lemmas
  3.36,3.37]{hofmann:interp} and \cite[Lemma 4.2.1,Lemma 4.2.6]{deboer20initiality}.
\cite{hofmann:interp} uses Kleene equality (if one side is defined, then
so is the other, and they're equal), but also makes sufficient
assumptions such that one side is defined, so the Kleene equality
reduces to equality.  The statements in \cite{deboer20initiality} assume
that e.g. $\den{a}[\theta]$ is defined rather than assuming separately
that $\den{a}$ and $\theta$ are defined, but these are equivalent.


The following crucial Lemma relates the syntactic roundtrip operation to the semantic roundtrip maps available in any CwB.
\begin{lemma}\label{lem:sem-roundtrip}
  Suppose $\den{\Gamma, \Delta^{0\Gamma}}$ and $\den{\Gamma, \Delta}$
  are defined. Then there is a morphism
  $\Rt_{\Gamma; \Delta} : \den{\Gamma, \Delta^{0\Gamma}}
  \to\den{\Gamma, \Delta}$ such that whenever
  $\den{\Gamma, \Delta \yields \judge}$ is defined then
  $\den{\Gamma, \Delta^{0\Gamma} \yields \judge^{0\Gamma}}$ is defined and
  \[ \den{\Gamma, \Delta^{0\Gamma} \yields \judge^{0\Gamma}} \defeq \den{\Gamma, \Delta \yields \judge}[\Rt_{\Gamma; \Delta}] \]
  where $\judge$ denotes either a type or term.

  Moreover, $\Rt_{\Gamma;\cdot} :\defeq \varepsilon_{\den{\Gamma}} \circ
  \eta_{\den{\Gamma}}$, so in particular, if $\den{\Gamma}$ and
  $\den{\Gamma \yields \judge}$ are well-defined then $\den{\Gamma
    \yields \zc{\judge}}$ is defined and
  \[ \den{\Gamma \yields \zc{\judge}} \defeq \den{\Gamma \yields \judge}[\varepsilon_{\den{\Gamma}} \circ \eta_{\den{\Gamma}}] \]
\end{lemma}
\begin{proof}

The definition of the morphism $\Rt$ and the proof of the equation for
substituting by it are mutually recursive, because the equation is
needed for the definition of $\Rt$ to have the stated type.  The
induction metric is: for $\Rt_{\Gamma;\Delta}$ the total size of the raw
context $\Gamma,\Delta$; and for the equation, the total size of the raw
context $\Gamma,\Delta$ plus the size of $\judge$.  The size of a
context adds $1$ for each variable in addition to counting the size of
the type (e.g. $size(\Gamma,x:A) = size(\Gamma) + size(A) + 1$), which
means that the definition of the morphism for $\Gamma,\Delta,x:A$ can
use the equation at $\Gamma,\Delta \yields A$, but the proof of the
equation for $\Gamma,\Delta \yields J$ only uses the morphism at
$\Gamma,\Delta$.

The morphism $\Rt_{\Gamma; \Delta}$ is defined by
\begin{align*}
\Rt_{\Gamma;\cdot} &:\defeq \varepsilon_{\den{\Gamma}} \circ \eta_{\den{\Gamma}} \\
\Rt_{\Gamma;\Delta, x : A} &:\defeq \pair{\Rt_{\Gamma;\Delta} \circ \projsub_{\den{A^{0\Gamma}}}}{\gen_{\den{A^{0\Gamma}}}} \\
\Rt_{\Gamma;\Delta, \zx :: A} &:\defeq \pair{\Rt_{\Gamma;\Delta} \circ \projsub_{(N\den{\Gamma, \Delta^{0\Gamma} \yields A})[\eta_{\den{\Gamma, \Delta^{0\Gamma}}}]}}{\gen_{(N\den{\Gamma, \Delta^{0\Gamma} \yields A})[\eta_{\den{\Gamma, \Delta^{0\Gamma}}}]}}
\end{align*}

In the second case, $\den{\Gamma,\Delta,x:A} \defeq
\den{\Gamma,\Delta}.\den{\Gamma,\Delta \vdash A}$, and we have
$\den{\Gamma,\Delta^{0\Gamma},x:A^{0\Gamma}} \defeq
\den{\Gamma,\Delta^{0\Gamma}}.\den{\Gamma,\Delta^{0\Gamma} \vdash
  A^{0\Gamma}}$.  By the IH, since $\Gamma,\Delta$ is a smaller context
than $\Gamma,\Delta,x:A$, we have $ \Rt_{\Gamma;\Delta} :
\den{\Gamma,\Delta^{0\Gamma}} \to \den{\Gamma,\Delta}$, and
$\projsub_{\den {A^{0\Gamma}}} :
\den{\Gamma,\Delta^{0\Gamma}}.\den{\Gamma,\Delta^{0\Gamma} \vdash
  A^{0\Gamma}} \to \den{\Gamma,\Delta^{0\Gamma}}$.  Thus, by the pairing
rule for substitutions in a CwF, the second component of the result, we
require a term in $\C(\den{\Gamma,\Delta^{0\Gamma}}.\den{A^{0\Gamma}}
\yields \den{A}[\Rt_{\Gamma,\Delta} \circ \projsub_{\den
    {A^{0\Gamma}}}])$.  Since the size of $\Gamma,\Delta \yields A$ is
the same as the size of the context $\Gamma,\Delta, x : A$ (we count 1
for the variable $x$), we can use the inductive hypothesis for the
second part for $A$, and this type is equal to
$\den{A}[\Rt_{\Gamma,\Delta}][\projsub_{\den {A^{0\Gamma}}}]$,
justifying the use of $\gen_{\den{A^{0\Gamma}}}$.

The third case is similar, but we need to know that
${(N\den{\Gamma, \Delta \yields A})[\eta_{\den{\Gamma, \Delta}}][\Rt_{\Gamma;\Delta}]}
\defeq
{(N\den{\Gamma, \Delta^{0\Gamma} \yields A})[\eta_{\den{\Gamma, \Delta^{0\Gamma}}}]}$.
Again using the IH on the smaller type $A$, and commuting the
semantic substitution by $\Rt_{\Gamma,\Delta}$ outside of the $N$, it suffices to
show that
$\eta_{\den{\Gamma, \Delta}} \circ \Rt_{\Gamma;\Delta}
\defeq
N(\Rt_{\Gamma,\Delta}) \circ \eta_{\den{\Gamma, \Delta^{0\Gamma}}}  :
\den{\Gamma,\Delta^{0\Gamma}} \to N\den{\Gamma,\Delta}$, which are equal by
naturality of $\eta$.

The use of $\Rt$ in the statement of the second part is justified
because $\Gamma;\Delta$ is smaller than the total size of $\Gamma;\Delta
\yields \judge$.  For the proof, we go by cases.  In the cases for
terms, because the term interpretation returns a pair of a family and a
term, we must show that both components are defined and equal.

In each case, we first show that the families
are equal and then the terms.  We orient the calculations from the
left-hand side of the theorem statement to the right-hand side, so for
definedness we need to argue that if the ``unsubstituted'' type/term at
the bottom is defined, then the top is.  For each calculation, we show
the steps and then argue why each step in it is defined.

\begin{itemize}
\item \rulen{var}: Suppose $x$ is the final variable in the context.  There are two cases:
\begin{itemize}
\item $x \in \Delta$:

In this case, we have an inductive hypothesis for $A$, which is smaller
than the input $\Gamma,\Delta,x:A \yields x$, unfolding the gives the
result.

For the families, we have
\begin{align*}
&\den{{\Gamma, \Delta^{0\Gamma}} \yields A^{0\Gamma}}[\projsub_{\den{A^{0\Gamma}}}]\\
&\defeq \den{\Gamma,\Delta \yields A}[\Rt_{\Gamma,\Delta}][\projsub_{\den{A^{0\Gamma}}}] && \text{(IH for $A$)}\\
&\defeq \den{\Gamma,\Delta \yields A}[\projsub_A][\Rt_{\Gamma,\Delta,x:A}] && \text{(Definition of $\Rt_{\Gamma;\Delta,x:A}$)}\\
\end{align*}

For the terms, we have
\begin{align*}
&\den{\Gamma, \Delta^{0\Gamma}, x : A^{0\Gamma} \yields x} \\
&\defeq \gen_{\den{\Gamma,\Delta^{0\Gamma} \yields A^{0\Gamma}}} \\
&\defeq \gen_{\den{\Gamma,\Delta \yields A}[\Rt_{\Gamma;\Delta}]} && \text{(IH for $A$)}\\
&\defeq \gen_{\den{A}} [\pair{\Rt_{\Gamma;\Delta} \circ \projsub_{\den{A^{0\Gamma}}}}{\gen_{\den{A^{0\Gamma}}}}] && \text{(Equation \ref{eq:zq-sub})}\\
&\defeq \den{\Gamma, \Delta, x : A \yields x}[\pair{\Rt_{\Gamma;\Delta} \circ \projsub_{\den{A^{0\Gamma}}}}{\gen_{\den{A^{0\Gamma}}}}] \\
&\defeq \den{\Gamma, \Delta, x : A \yields x}[\Rt_{\Gamma;\Delta, x : A}]
\end{align*}

For definedness, we have assumed that both the top and bottom contexts
are defined, so $\den{\Gamma, \Delta^{0\Gamma}, x : A^{0\Gamma}} \defeq
\den{\Gamma, \Delta^{0\Gamma}}.\den{A^{0\Gamma}}$ and similarly on the
bottom.  Definedness of the contexts implies definedness of $\den{A}$
and $\den{A^{0\Gamma}}$, so all of the subterms.

\item $x \notin \Delta$ (so in this case $\Delta \defeq \cdot$):

For the families, both sides have type
$\den{\Gamma \yields A}[\projsub_{\den A}][\varepsilon_{\den{\Gamma, x : A}} \circ \eta_{\den{\Gamma, x : A}}]$
and this is defined because $\den{\Gamma \yields A}$ is, because the
context is assumed to be well-formed.

For the terms,
\begin{align*}
&\den{\Gamma, x : A \yields \zx} \\
&\defeq \gen_{\den{A}}[\varepsilon_{\den{\Gamma, x : A}} \circ \eta_{\den{\Gamma, x : A}}] \\
&\defeq \den{\Gamma, x : A \yields x}[\varepsilon_{\den{\Gamma, x : A}} \circ \eta_{\den{\Gamma, x : A}}] \\
&\defeq \den{\Gamma, x : A \yields x}[\Rt_{\Gamma, x : A;\cdot}]
\end{align*}

When $\den{x}$ on the bottom is defined, all of the steps are defined,
because $\den{\Gamma,x:A}$ and therefore $\den{A}$ are assumed.

\end{itemize}

The more general situation of a variable in the middle of the context
follows by an easy induction, because $\Rt$ is constructed by iterated
pairings, which cancel with the projections in the interpretation of the
variable term.

\item \rulen{var-zero}: Again suppose $\zx :: A$ is the final variable
  in the context.

\begin{itemize}
\item $x \in \Delta$: First, because $\den{\Gamma, \Delta, \zx :: A}$
  is defined, we know $A$ has only marked variables, i.e.
  $A \defeq \zA \defeq A^{0(\Gamma,\Delta)}$. Syntactic idempotence
  (Lemma~\ref{lem:syn-idem}) also gives that
  $\zA \defeq \zA^{0\Gamma}$, so also
  $A \defeq \zA^{0\Gamma} \defeq A^{0\Gamma}$.  Since $A$ is smaller
  than ${\Gamma, \Delta, \zx :: A} \yields \zx$, we have an inductive
  hypothesis for $A$.

  For the families we have
\begin{align*}
  &\den{\Gamma, \Delta^{0\Gamma} \yields A}[\varepsilon_{\den{\Gamma,\Delta^{0\Gamma}}} \circ \eta_{\den{\Gamma,\Delta^{0\Gamma}}}][\projsub_{N (\den {\Gamma,\Delta^{0\Gamma} \yields A})[\eta_{\den{\Gamma,\Delta^{0\Gamma}}}]}] \\
  &\defeq \den{\Gamma, \Delta^{0\Gamma} \yields A^{0\Gamma}}[\varepsilon_{\den{\Gamma,\Delta^{0\Gamma}}} \circ \eta_{\den{\Gamma,\Delta^{0\Gamma}}}][\projsub_{N (\den {\Gamma,\Delta^{0\Gamma} \yields A})[\eta_{\den{\Gamma,\Delta^{0\Gamma}}}]}] & (A \defeq A^{0\Gamma}) \\
  &\defeq \den{\Gamma, \Delta \yields A}[\Rt_{\Gamma;\Delta}][\varepsilon_{\den{\Gamma,\Delta^{0\Gamma}}} \circ \eta_{\den{\Gamma,\Delta^{0\Gamma}}}][\projsub_{N (\den {\Gamma,\Delta^{0\Gamma} \yields A})[\eta_{\den{\Gamma,\Delta^{0\Gamma}}}]}] & \text{(IH for $A$)}\\
  &\defeq \den{\Gamma, \Delta \yields A}[\varepsilon_{\den{\Gamma,\Delta}} \circ \eta_{\den{\Gamma,\Delta}}][\Rt_{\Gamma;\Delta}][\projsub_{N (\den {\Gamma,\Delta^{0\Gamma} \yields A})[\eta_{\den{\Gamma,\Delta^{0\Gamma}}}]}] & \text{(Naturality of $\varepsilon \circ \eta$)}\\
  &\defeq \den{\Gamma, \Delta \yields A}[\varepsilon_{\den{\Gamma,\Delta}} \circ \eta_{\den{\Gamma,\Delta}}][\projsub_{N (\den {\Gamma,\Delta \yields A})[\eta_{\den{\Delta,\Gamma}}]}][\Rt_{\Gamma;\Delta,x::A}] & \text{(Definition of $\Rt$)}
\end{align*}

\begin{align*}
&\den{\Gamma, \Delta^{0\Gamma}, \zx :: A \yields \zx} \\
  &\defeq \zc{\gen}_{\den{\Gamma, \Delta^{0\Gamma} \yields A}} \\
  &\defeq \zc{\gen}_{\den{\Gamma, \Delta^{0\Gamma} \yields A^{0\Gamma}}} && (A \defeq A^{0\Gamma})\\
  &\defeq \zc{\gen}_{\den{\Gamma, \Delta \yields A}[\Rt_{\Gamma; \Delta}]} && \text{(Induction)} \\
&\defeq \zc{\gen}_{\den{\Gamma, \Delta \yields A}}[\pair{\Rt_{\Gamma; \Delta} \circ \projsub}{\gen}] && \text{(Equation~\ref{eq:zq-sub})} \\
&\defeq \den{\Gamma, \Delta, \zx :: A \yields \zx}[\Rt_{\Gamma; \Delta, \zx :: A}]
\end{align*}

In both calcualations, because of the assumption that
$\den{\Gamma,\Delta^{0\Gamma},x::A}$ is defined, all of the
subexpressions are defined.

\item $x \notin \Delta$:

For the families,
\[
\den{\Gamma \yields A}[\varepsilon_{\den{\Gamma}} \circ \eta_{\den{\Gamma}}][\projsub_{N (\den {\Gamma \yields A})[\eta_{\den{\Gamma}}]}]
\defeq
\den{\Gamma \yields A}[\varepsilon_{\den{\Gamma}} \circ \eta_{\den{\Gamma}}][\projsub_{N (\den {\Gamma \yields A})[\eta_{\den{\Gamma}}]}][\varepsilon_{\den{\Gamma,x::A}} \circ \eta_{\den{\Gamma,x::A}}]
\]
by naturality and idempotence of $\varepsilon \circ \eta$.

\begin{align*}
&\den{\Gamma, \zx :: A \yields \zx} \\
&\defeq \zc{\gen}_{\den{A}} \\
&\defeq \zc{\gen}_{\den{A}}[\varepsilon_{\den{\Gamma, \zx :: A}} \circ \eta_{\den{\Gamma, \zx :: A}}] && \text{(Equation~\ref{eq:zq-idem})} \\
&\defeq \den{\Gamma, \zx :: A \yields \zx}[\Rt_{\Gamma, \zx :: A;\cdot}]
\end{align*}

In both calculations, all subexpressions are defined because $\den{\Gamma,x::A}$ is defined.

\end{itemize}

We again need an induction for non-last variables.

\item \rulen{var-roundtrip}: Again assume $x$ is the last variable in
  the context.

\begin{itemize}
\item $x \in \Delta$:

For the families,
\[
\den{\Gamma, \Delta^{0\Gamma} \yields A^{0\Gamma}}[\projsub_{\den{\Gamma, \Delta^{0\Gamma} \yields A^{0\Gamma}}}]
\defeq
\den{\Gamma, \Delta \yields A}[\projsub_{\den{\Gamma, \Delta \yields A}}][\Rt_{\Gamma; \Delta, x : A}]
\]
by definition of $\Rt_{\Gamma; \Delta, x : A}$ and the IH for the type $A$, which is smaller than $\Gamma,\Delta,x:A \vdash \zx$.

For the terms
\begin{align*}
&\den{\Gamma, \Delta^{0\Gamma}, x : A^{0\Gamma} \yields \zx} \\
&\defeq \gen_{\den{A^{0\Gamma}}}[\varepsilon_{\den{\Gamma, \Delta^{0\Gamma}, x : A^{0\Gamma}}} \circ \eta_{\den{\Gamma, \Delta^{0\Gamma}, x : A^{0\Gamma}}}] \\
&\defeq \gen_{\den{A}}[\pair{\Rt_{\Gamma; \Delta} \circ \projsub_{\den{A^{0\Gamma}}}}{\gen_{\den{A^{0\Gamma}}}}][\varepsilon_{\den{\Gamma, \Delta^{0\Gamma}, x : A^{0\Gamma}}} \circ \eta_{\den{\Gamma, \Delta^{0\Gamma}, x : A^{0\Gamma}}}] \\
&\defeq \gen_{\den{A}}[\varepsilon_{\den{\Gamma, \Delta, x : A}} \circ \eta_{\den{\Gamma, \Delta, x : A}}][\pair{\Rt_{\Gamma; \Delta} \circ \projsub_{\den{A^{0\Gamma}}}}{\gen_{\den{A^{0\Gamma}}}}] &&\text{(Naturality)} \\
&\defeq \den{\Gamma, \Delta, x : A \yields \zx}[\Rt_{\Gamma; \Delta, x : A}]
\end{align*}

In both calculations, $\den{\Gamma; \Delta, x : A}$ and $\den{\Gamma; \Delta^{0\Gamma}, x : A^{0\Gamma}}$ defined implies all of the subterms are defined.

\item $x \notin \Delta$:

  For the families, we have
  \[
  \den{\Gamma \yields A}[\projsub_{\den{A}}][\varepsilon_{\den{\Gamma, x : A}} \circ \eta_{\den{\Gamma, x : A}}]
  \defeq
  \den{\Gamma \yields A}[\projsub_{\den{A}}][\varepsilon_{\den{\Gamma, x : A}} \circ \eta_{\den{\Gamma, x : A}}][\varepsilon_{\den{\Gamma, x : A}} \circ \eta_{\den{\Gamma, x : A}}]
  \]
  by idempotence of $\varepsilon \circ \eta$.

  For the terms
\begin{align*}
&\den{\Gamma, x : A \yields \zx} \\
&\defeq \gen_{\den{A}}[\varepsilon_{\den{\Gamma, x : A}} \circ \eta_{\den{\Gamma, x : A}}] \\
&\defeq \gen_{\den{A}}[\varepsilon_{\den{\Gamma, x : A}} \circ \eta_{\den{\Gamma, x : A}}][\varepsilon_{\den{\Gamma, x : A}} \circ \eta_{\den{\Gamma, x : A}}] &&\text{($\eta \circ \varepsilon$ is identity)}\\
&\defeq \den{\Gamma, x : A \yields \zx}[\varepsilon_{\den{\Gamma, x : A}} \circ \eta_{\den{\Gamma, x : A}}] \\
&\defeq \den{\Gamma, x : A \yields \zx}[\Rt_{\Gamma, x : A;\cdot}]
\end{align*}
Once again, in both cases,
$\den{\Gamma, x : A}$ defined implies all of the subexpressions are defined.

\end{itemize}

We again need an induction for non-last variables.

\item As a final variable case, it may be that
  $\zx :: A \in \Gamma, \Delta$, but we have the raw syntax $x$ as the
  judgement $\judge$. This is in fact not possible, as it contradicts
  the assumption that $\den{\Gamma, \Delta, \zx :: A,\Delta' \yields x}$ is
  defined, as there is no clause for this in the definition of the partial interpretation.

\item \rulen{$\natural$-form}:
\begin{align*}
&\den{\Gamma, \Delta^{0\Gamma} \yields \natural A^{0\Gamma}} \\
&\defeq (N\den{\Gamma, \Delta^{0\Gamma} \yields A^{0\Gamma}})[\eta_{\den{\Gamma,\Delta^{0\Gamma}}}] \\
&\defeq (N\den{\Gamma, \Delta \yields A}[\Rt_{\Gamma;\Delta}])[\eta_{\den{\Gamma,\Delta^{0\Gamma}}}] && \text{(Induction)}  \\
&\defeq (N\den{\Gamma, \Delta \yields A})[N\Rt_{\Gamma;\Delta}][\eta_{\den{\Gamma,\Delta^{0\Gamma}}}] && \text{($N$ commutes with substitution)}  \\
&\defeq (N\den{\Gamma, \Delta \yields A})[\eta_{\den{\Gamma,\Delta}}][\Rt_{\Gamma;\Delta}] && \text{(Naturality)}  \\
&\defeq \den{\Gamma, \Delta \yields \natural A}[\Rt_{\Gamma;\Delta}]
\end{align*}
For definedness, we assume that $\den{\Gamma,\Delta}$ and
$\den{\Gamma,\Delta^{0\Gamma}}$ are defined, and definedness of
$\den{\Gamma, \Delta \yields \natural A}$ on the bottom implies
definedness of $\den{\Gamma, \Delta \yields A}$, so $\den{\Gamma,
  \Delta^{0\Gamma} \yields A^{0\Gamma}}$ is defined by the IH for the
smaller type $A$.  This makes all of the subexpressions defined.

\item \rulen{$\natural$-intro}:

For the families, $\Gamma,\Delta \yields A$ is smaller than
$\Gamma,\Delta\yields a_A^\natural$, and $\den{a_A^\natural}$ defined
implies $\den{A}$ defined, so we have
\[
(N\den{\Gamma, \Delta^{0\Gamma} \yields A^{0\Gamma}})[\eta_{\den{\Gamma,\Delta^{0\Gamma}}}] \defeq (N\den{\Gamma, \Delta \yields A})[\eta_{\den{\Gamma,\Delta}}][\Rt_{\Gamma;\Delta}]
\]
by the same reasoning as in the previous case.

For terms,
\begin{align*}
&\den{\Gamma, \Delta^{0\Gamma} \yields ({a^{0\Gamma}_{A^{0\Gamma}}})^\natural)} \\
&\defeq (N_{\den{\Gamma,\Delta^{0\Gamma} \yields A^{0\Gamma}}}{\den{\Gamma, \Delta^{0\Gamma} \yields a^{0\Gamma}}})[\eta_{\den{\Gamma,\Delta^{0\Gamma}}}] \\
&\defeq (N_{\den{\Gamma,\Delta \yields A}[\Rt_{\Gamma;\Delta}]}{\den{\Gamma, \Delta \yields a}[\Rt_{\Gamma;\Delta}]})[\eta_{\den{\Gamma,\Delta^{0\Gamma}}}] && \text{(IH for $A$ and $a$)}  \\
&\defeq (N_{\den{\Gamma, \Delta \yields A}}\den{\Gamma, \Delta \yields a})[N\Rt_{\Gamma;\Delta}][\eta_{\den{\Gamma,\Delta^{0\Gamma}}}] && \text{($N$ commutes with substitution)}  \\
&\defeq (N_{\den{\Gamma, \Delta \yields A}}\den{\Gamma, \Delta \yields a})[\eta_{\den{\Gamma,\Delta}}][\Rt_{\Gamma;\Delta}] && \text{(Naturality)}  \\
&\defeq \den{\Gamma, \Delta \yields a_A^\natural}[\Rt_{\Gamma;\Delta}]
\end{align*}
The IH on $A$ and $a$ give that $\den{A^{0\Gamma}}$ and
$\den{a^{0\Gamma}}$ are defined, $\den{\Gamma,\Delta}$ and
$\den{\Gamma,\Delta^{0\Gamma}}$ are assumed, so all of the
subexpressions are defined.  For $N_{\den{\Gamma,\Delta^{0\Gamma}
    \yields A^{0\Gamma}}}{\den{\Gamma, \Delta^{0\Gamma} \yields
    a^{0\Gamma}}}$ to be defined, we need that the family for
${\den{\Gamma, \Delta^{0\Gamma} \yields a^{0\Gamma}}}$ is equal to
${\den{\Gamma, \Delta^{0\Gamma} \yields A^{0\Gamma}}}$.  This holds
because the family part of $\den{a}$ is equal to $\den{A}$ because
$\den{a^\natural_A}$ is defined (by the implicit conditions that the
righthand sides must type check), so the family part of
$\den{a^0\Gamma}$ is equal to $\den{A}[\Rt]$ by the family part of the
IH for $a$, and by the IH for $A$, $\den{A}[\Rt] \defeq
\den{A^{0\Gamma}}$.

\item \rulen{$\natural$-elim}:

Assuming $\den{\Gamma,\Delta \yields n_\natural^A}$ is defined, both
$\den{A}$ and $\den{n}$ are defined, so by the inductive hypotheses on
$A$ and $n$, $\den{\Gamma,\Delta^{0\Gamma} \yields A^{0\Gamma}}$ is
defined and equal to $\den{A}[\Rt_{\Gamma;\Delta}]$, and
$\den{n^{0\Gamma}}$ is defined and equal to
$\den{n}[\Rt_{\Gamma;\Delta}]$.  For the families,

\begin{align*}
 & \den{\Gamma,\Delta^{0\Gamma} \yields A^{0\Gamma}}[\varepsilon_{\den{\Gamma,\Delta^{0\Gamma}}} \circ\eta_{\den{\Gamma,\Delta^{0\Gamma}}}]\\
 &\defeq \den{\Gamma,\Delta \yields A}[\Rt_{\Gamma;\Delta}][\varepsilon_{\den{\Gamma,\Delta^{0\Gamma}}} \circ\eta_{\den{\Gamma,\Delta^{0\Gamma}}}] &\text{(IH for $A$)} \\
 &\defeq \den{\Gamma,\Delta \yields A}[\varepsilon_{\den{\Gamma,\Delta}} \circ\eta_{\den{\Gamma,\Delta}}][\Rt_{\Gamma;\Delta}] & \text{(naturality)}
\end{align*}
Since $\den{A}$ and $\den{A^{0\Gamma}}$ and the denotations of the contexts are all defined, all of the steps are defined.

For the terms
\begin{align*}
  &\den{\Gamma,\Delta^{0\Gamma} \yields (n^{0\Gamma})^{A^{0\Gamma}}_\natural} \\
  &\defeq \zc{\gen}_{\den{A^{0\Gamma}}}[\pair{\id_{\den{\Gamma,
    \Delta^{0\Gamma}}}}{\den{\Gamma,\Delta^{0\Gamma} \yields
    n^{0\Gamma}}}]
    && \text{(Definition)} \\
  &\defeq \zc{\gen}_{\den{A}[\Rt_{\Gamma;\Delta}]}[\pair{\id_{\den{\Gamma,
    \Delta^{0\Gamma}}}}{\den{\Gamma,\Delta \yields
    n}[\Rt_{\Gamma;\Delta}]}]
    && \text{(IH for $A$ and $n$)}\\
  &\defeq \zc{\gen}_{\den{A}}[\pair{\Rt_{\Gamma;\Delta} \circ \projsub_{\dots}}{\gen_{\dots}}]
    [\pair{\id_{\den{\Gamma,\Delta^{0\Gamma}}}}{\den{\Gamma,\Delta \yields n}[\Rt_{\Gamma;\Delta}]}]
    && \text{(Equation~\ref{eq:zq-sub})} \\
  &\defeq \zc{\gen}_{\den{A}}[\pair{\Rt_{\Gamma;\Delta}}{\den{\Gamma,\Delta \yields n}[\Rt_{\Gamma;\Delta}]}]
    && \text{(Composing substitutions)} \\
  &\defeq \zc{\gen}_{\den{A}}[\pair{\id_{\den{\Gamma,\Delta}}}{\den{\Gamma,\Delta \yields n}}][\Rt_{\Gamma;\Delta}]
    && \text{(Un-composing substitutions)} \\
  &\defeq \den{\Gamma,\Delta \yields n_\natural^A}[\Rt_{\Gamma;\Delta}]
\end{align*}
Since $\den{n}$ and $\den{n^{0\Gamma}}$ and the denotations of the contexts
and types are all defined, all of the subterms are defined.  For
$\zc{\gen}_{\den{A^{0\Gamma}}}[\pair{\id_{\den{\Gamma,
        \Delta^{0\Gamma}}}}{\den{\Gamma,\Delta^{0\Gamma} \yields
      n^{0\Gamma}}}]$ to be defined, we need that
$\den{\Gamma,\Delta^{0\Gamma} \yields n^{0\Gamma}}$ has family
$N{\den{A^{0\Gamma}}}[\eta_{\den{\Gamma,\Delta^{0\Gamma}}}]$.  Because
$\den{n_\natural^A}$ is defined, we know that the family part of
$\den{n}$ is equal to $N {\den{\Gamma,\Delta \yields A}} [\eta_{\Gamma,\Delta}]$,
so by the IH the family part of $\den{n^{0\Gamma}}$ is equal to
$N {\den{\Gamma,\Delta \yields A}} [\eta_{\Gamma,\Delta}][\Rt_{\Gamma;\Delta}]$,
and this is equal to $N{\den{A^{0\Gamma}}}[\eta_{\den{\Gamma,\Delta^{0\Gamma}}}]$
by the IH for $A$ via the same reasoning as in the \rulen{$\natural$-form} case.

\item \rulen{$\Pi$-form}:

  If $\bigden{\Gamma, \Delta \yields \prd{x : A} B}$ is defined, then so
  are both of $\den{\Gamma, \Delta \yields A}$ and $\den{\Gamma, \Delta,
    x : A \yields B}$. We can apply the IH (on the same contexts) to $A$
  to see that $\den{\Gamma, \Delta^{0\Gamma} \yields A^{0\Gamma}}$ is
  defined.

  By assumption, $\den{\Gamma, \Delta}$ and
  $\den{\Gamma, \Delta^{0\Gamma}}$ are defined so both of
  $\den{\Gamma, \Delta, x : A}$ and
  $\den{\Gamma, \Delta^{0\Gamma}, x : A^{0\Gamma}}$ are also. We can
  then apply the IH  to $B$ to see
  $\den{\Gamma, \Delta^{0\Gamma}, x : A^{0\Gamma} \yields
    B^{0\Gamma}}$ is defined. It follows that
  $\bigden{\Gamma, \Delta^{0\Gamma} \yields \prd{x : A^{0\Gamma}}
    B^{0\Gamma}}$ is also defined.

  The morphism $\Rt_{\Gamma;\Delta,x:A}$ is available by the inductive
  hypothesis--- the size of $\Gamma,\Delta,x:A$ is less than the size of
  $\Gamma,\Delta \yields \prd{x:A} B$ because $B$ (like all types/terms)
  has size at least 1.

  To show the equation holds,
\begin{align*}
&\bigden{\Gamma, \Delta^{0\Gamma} \yields \prd{x : A^{0\Gamma}} B^{0\Gamma}} \\
&\defeq \Pi(\den{\Gamma, \Delta^{0\Gamma} \yields A^{0\Gamma}}, \den{\Gamma, \Delta^{0\Gamma}, x : A^{0\Gamma} \yields B^{0\Gamma}}) \\
&\defeq \Pi(\den{\Gamma, \Delta \yields A}[\Rt_{\Gamma;\Delta}], \den{\Gamma, \Delta, x : A \yields B}[\Rt_{\Gamma;\Delta, x : A}]) && \text{(IH on $A$ and $B$)} \\
&\defeq \Pi(\den{\Gamma, \Delta \yields A}[\Rt_{\Gamma;\Delta}], \den{\Gamma, \Delta, x : A \yields B}[\Rt_{\Gamma;\Delta} \circ \projsub_{\den{A^{0\Gamma}}}, \gen_{\den{A^{0\Gamma}}}]) && \text{(Definition of $\Rt_{\Gamma;\Delta, x : A}$)} \\
&\defeq \Pi(\den{\Gamma, \Delta \yields A}, \den{\Gamma, \Delta, x : A \yields  B})[\Rt_{\Gamma;\Delta}] && \text{(Semantic $\Pi$ commutes with substitution)} \\
&\defeq \bigden{\Gamma, \Delta \yields \prd{x : A} B}[\Rt_{\Gamma;\Delta}]
\end{align*}

\item \rulen{$\Pi$-intro}:

Since $\den{\Gamma,\Delta}$ and $\den{\Gamma,\Delta^{0\Gamma}}$ and
$\den{\Gamma,\Delta \yields \lambda x:A.b:B}$ are defined by assumption,
all of $\den{\Gamma,\Delta \yields A}$ and $\den{\Gamma,\Delta,x:A}$ and
$\den{\Gamma,\Delta,x:A \yields B}$ $\den{\Gamma,\Delta,x:A \yields b}$
are defined, so the inductive hypotheses imply that
$\den{\Gamma,\Delta^{0\Gamma} \yields A^{0\Gamma}}$,
$\den{\Gamma,\Delta^{0\Gamma},x:A^{0\Gamma}}$,
$\den{\Gamma,\Delta^{0\Gamma},x:A^{0\Gamma} \yields B^{0\Gamma}}$,
$\den{\Gamma,\Delta^{0\Gamma},x:A^{0\Gamma} \yields b^{0\Gamma}}$ are
defined.  The morphism $\Rt_{\Gamma;\Delta,x:A}$ is available because
$\Gamma,\Delta,x:A$ is smaller than $\Gamma,\Delta \yields \lambda x:A.b:B$.

The families
\[
\Pi(\den{\Gamma, \Delta^{0\Gamma} \yields A}, \den{\Gamma, \Delta^{0\Gamma}, x : A^{0\Gamma} \yields  B^{0\Gamma}})
\defeq
\Pi(\den{\Gamma, \Delta \yields A}, \den{\Gamma, \Delta, x : A \yields  B})[\Rt_{\Gamma;\Delta}]
\]
are defined and equal as in the previous case.

For the terms
\begin{align*}
& \den{\Gamma,\Delta^{0\Gamma} \yields \lambda x:A^{0\Gamma}.b^{0\Gamma}:B^{0\Gamma}}\\
& \defeq \lambda(\den{\Gamma,\Delta^{0\Gamma}},\den{\Gamma,\Delta^{0\Gamma} \yields  A^{0\Gamma}},\den{\Gamma,\Delta^{0\Gamma},x:A^{0\Gamma} \yields B^{0\Gamma}},\den{\Gamma,\Delta^{0\Gamma},x:A^{0\Gamma} \yields b^{0\Gamma}})\\
& \defeq \lambda(\den{\Gamma,\Delta^{0\Gamma}},\den{\Gamma,\Delta \yields  A}[\Rt_{\Gamma;\Delta}],\den{\Gamma,\Delta,x:A \yields B}[\Rt_{\Gamma;\Delta,x:A}],\den{\Gamma,\Delta,x:A \yields b}[\Rt_{\Gamma;\Delta ,x:A}  ])  && \text{(IH)}\\
& \defeq \lambda(\den{\Gamma,\Delta},\den{\Gamma,\Delta \yields  A},\den{\Gamma,\Delta,x:A \yields B},\den{\Gamma,\Delta,x:A \yields b})[\Rt_{\Gamma;\Delta}] && \text{(Substitution into $\lambda$)} \\
&  \defeq \den{\Gamma,\Delta \yields \lambda x:A.b:B}[\Rt_{\Gamma;\Delta}]
\end{align*}
All of the subterms are defined, so to see that the top is defined, we just need to know that
$\den{\Gamma,\Delta^{0\Gamma},x:A^{0\Gamma} \yields b^{0\Gamma}}$ has type
$\den{\Gamma,\Delta^{0\Gamma},x:A^{0\Gamma} \yields B^{0\Gamma}}$.
Because the bottom is well-typed, the family part of $b$ is equal to
$\den{\Gamma,\Delta,x:A \yields B}$, so the family part of the IH for $b$ says
that its type is equal to $\den{\Gamma,\Delta,x:A \yields B}[\Rt_{\Gamma;\Delta,x:A}]$,
which by the IH for $B$ is equal to what we want.

\item \rulen{$\Pi$-elim}:

Since $\den{\Gamma,\Delta}$ and $\den{\Gamma,\Delta^{0\Gamma}}$ and
$\den{\Gamma,\Delta \yields f(a)_{A,x.B}}$ are defined by assumption,
all of $\den{\Gamma,\Delta \yields A}$ and $\den{\Gamma,\Delta,x:A}$ and
$\den{\Gamma,\Delta,x:A \yields B}$ and $\den{\Gamma,\Delta \yields f}$
and
$\den{\Gamma,\Delta \yields a}$
are defined, so the inductive hypotheses imply that
$\den{\Gamma,\Delta^{0\Gamma} \yields A^{0\Gamma}}$,
$\den{\Gamma,\Delta^{0\Gamma},x:A^{0\Gamma}}$,
$\den{\Gamma,\Delta^{0\Gamma},x:A^{0\Gamma} \yields B^{0\Gamma}}$,
$\den{\Gamma,\Delta^{0\Gamma} \yields f^{0\Gamma}}$,
$\den{\Gamma,\Delta^{0\Gamma} \yields a^{0\Gamma}}$ are
defined.  The morphism $\Rt_{\Gamma;\Delta,x:A}$ is available because
$\Gamma,\Delta,x:A$ is smaller than $\Gamma,\Delta \yields f(a)_{A,x.B}$.

For the families, we have
\begin{align*}
&\den{\Gamma,\Delta^{0\Gamma},x:A^{0\Gamma} \yields B^{0\Gamma}}[(\id_{\den{\Gamma,\Delta^{0\Gamma}}},\den{\Gamma,\Delta^{0\Gamma}\yields a^{0\Gamma}})] \\
&\defeq \den{\Gamma,\Delta,x:A \yields B}[\Rt_{\Gamma;\Delta,x:A}][(\id_{\den{\Gamma,\Delta^{0\Gamma}}},\den{\Gamma,\Delta \yields a}[\Rt_{\Gamma;\Delta}])]\ && \text{(IH)}\\
&\defeq \den{\Gamma,\Delta,x:A \yields B}[(\Rt_{\Gamma;\Delta} \circ \projsub, \gen)](\id_{\den{\Gamma,\Delta^{0\Gamma}}},\den{\Gamma,\Delta \yields a}[\Rt_{\Gamma;\Delta}]) && \text{(Definition of $\Rt_{\Gamma;\Delta,x:A}$)}\\
&\defeq \den{\Gamma,\Delta,x:A \yields B}[(\Rt_{\Gamma;\Delta},\den{\Gamma,\Delta \yields a}[\Rt_{\Gamma;\Delta}])] \\
&\defeq \den{\Gamma,\Delta,x:A \yields B}[(\id_{\den{\Gamma,\Delta}},\den{\Gamma,\Delta \yields a})][\Rt_{\Gamma;\Delta}]
\end{align*}

And for the elements:
\begin{align*}
& \den{\Gamma,\Delta \yields f(a)_{A,x.B}}\\
&\defeq \mathrm{app}_{\den{\Gamma,\Delta^{0\Gamma}},\den{\Gamma,\Delta^{0\Gamma} \yields A^{0\Gamma}},\den{\Gamma,\Delta^{0\Gamma},x: A^{0\Gamma} \yields B^{0\Gamma}  }}(\den{\Gamma,\Delta^{0\Gamma} \yields f^{0\Gamma}},\den{\Gamma,\Delta \yields a^{0\Gamma}})\\
&\defeq \mathrm{app}_{\den{\Gamma,\Delta^{0\Gamma}},\den{\Gamma,\Delta \yields A}[\Rt_{\Gamma;\Delta}],\den{\Gamma,\Delta,x: A \yields B  }[\Rt_{\Gamma;\Delta,x:A}]}(\den{\Gamma,\Delta \yields f}[\Rt_{\Gamma;\Delta}],\den{\Gamma,\Delta \yields a}[\Rt_{\Gamma;\Delta}]) & \text{(IH)}\\
&\defeq \mathrm{app}_{\den{\Gamma,\Delta},\den{\Gamma,\Delta \yields A},\den{\Gamma,\Delta,x: A \yields B  }}(\den{\Gamma,\Delta \yields f},\den{\Gamma,\Delta \yields a}) [\Rt_{\Gamma;\Delta}]\\
&\defeq \den{\Gamma,\Delta \yields f(a)_{A,x.B}}[\Rt_{\Gamma;\Delta}]
\end{align*}
All of the subexpressions are defined, so to check that the top is
defined, we need that the family part of
$\den{\Gamma,\Delta^{0\Gamma} \yields f^{0\Gamma}}$ is
$\Pi(\den{\Gamma,\Delta^{0\Gamma}},\den{\Gamma,\Delta^{0\Gamma} \yields
  A^{0\Gamma}},\den{\Gamma,\Delta^{0\Gamma},x: A^{0\Gamma} \yields
  B^{0\Gamma} })$ and that the family part of $\den{\Gamma,\Delta^{0\Gamma}
  \yields a^{0\Gamma}}$ is $\den{\Gamma,\Delta^{0\Gamma} \yields
  A^{0\Gamma}}$.  Since $\den{\Gamma,\Delta \yields f(a)_{A,x.B}}$ is
defined, we know that the family part of $\den{a}$ is $\den{\Gamma,\Delta
  \yields A}$ and the family part of $\den{f}$ is
$\Pi({\den{\Gamma,\Delta},\den{\Gamma,\Delta \yields A},\den{\Gamma,\Delta,x: A \yields B }})$.  The
family part of the IH for $a$ and $f$ implies that the family parts of $\den{a^{0\Gamma}}$ and $\den{f^{0\Gamma}}$ are
$\den{\Gamma,\Delta \yields A}[\Rt_{\Gamma;\Delta}]$ and
$\Pi({\den{\Gamma,\Delta},\den{\Gamma,\Delta \yields A},\den{\Gamma,\Delta,x: A \yields B }})[\Rt_{\Gamma;\Delta}]$,
which are equal to what we want by the inductive hypotheses for $A$ and $B$ and a calculation as in the case for \rulen{$\Pi$-form}.

\end{itemize}
\end{proof}

As usual we also need to define semantic morphisms corresponding to substitution and weakening. We need weakening for both marked and unmarked variables, but can get away with substitution just for ordinary variables using the `dull substitution' trick (Definition~\ref{def:dull-subst}).

\begin{lemma}[Semantic substitution]
  Suppose all of
  \begin{align*}
    &\den{\Gamma, x : A, \Gamma'} \\
    &\den{\Gamma \yields a} \\
    &\den{\Gamma, \Gamma'[a/x]}
  \end{align*}
  are defined and that $\den{a} \in \C(\den{\Gamma} \yields \den{\Gamma
    \yields A})$ ($\den{\Gamma \yields A}$ is defined because
  $\den{\Gamma,x:A,\Gamma'}$ is).  Then there is a morphism
  $\Subst(\Gamma;A;\Gamma;a) : \den{\Gamma, \Gamma'[a/x]} \to
  \den{\Gamma, x : A, \Gamma'}$ such that if $\den{\Gamma, x : A,
    \Gamma' \yields \judge}$ is defined then $\den{\Gamma, \Gamma'[a/x]
    \yields \judge[a/x]}$ is defined and
  \[\den{\Gamma, \Gamma'[a/x] \yields \judge[a/x]} \defeq \den{\Gamma,x
      : A, \Gamma' \yields \judge}[\Subst(\Gamma;A;\Gamma;a)]\] where
  $\judge$ is either a type or term. The $\Subst(\Gamma;A;\Gamma;a)$
  morphism is given by
\begin{align*}
\Subst(\Gamma;A;\cdot;a) &:\defeq \pair{\id_{\den{\Gamma}}}{\den{a}} \\
\Subst(\Gamma;A;\Gamma', y : B;a) &:\defeq \pair{\Subst(\Gamma;A;\Gamma';a) \circ \projsub_{\den{B[a/x]}}}{\gen_{\den{B[a/x]}}} \\
\Subst(\Gamma;A;\Gamma', y :: B;a) &:\defeq \pair{\Subst(\Gamma;A;\Gamma';a) \circ \projsub_{(N\den{\Gamma,\Gamma'[a/x] \yields B})[\eta_{\den{\Gamma,\Gamma'[a/x]}}]}}{\gen_{(N\den{\Gamma,\Gamma'[a/x] \yields B})[\eta_{\den{\Gamma,\Gamma'[a/x]}}]}}
\end{align*}
\end{lemma}
\begin{proof}
  The induction metric is, for $\Subst(\Gamma;A;\Gamma';a)$, the size of
  $\Gamma, x : A, \Gamma'$, and for the equation, the size of the raw
  context $\Gamma, x : A, \Gamma'$ plus the size of $\judge$.  Size is
  defined in such a way that the morphism can use the equation on any
  type in the context.

  In the $\Subst(\Gamma;A;\Gamma', y : B;a)$ case of the definition of
  $\Subst$, we have that
  \[\den{\Gamma, x : A,\Gamma', y : B} \defeq \den{\Gamma, x :
      A,\Gamma'}.\den{\Gamma, x : A,\Gamma' \yields B},\] and
  \[\den{\Gamma,\Gamma'[a/x], y : B[a/x]} \defeq
    \den{\Gamma,\Gamma'[a/x]}.\den{\Gamma,\Gamma'[a/x] \yields
      B[a/x]}\] are defined.  By the IH, since $\Gamma'$ is smaller
  than $\Gamma', y : B$, we have a morphism
  $\Subst(\Gamma;A;\Gamma';a) : \den{\Gamma,\Gamma'[a/x]} \to
  \den{\Gamma, x : A, \Gamma'}$. To apply the paring rule, we need an
  element of
  $\Elt{\den{\Gamma,\Gamma'[a/x]}.\den{B[a/x]}}{\den{B}[\Subst(\Gamma;A;\Gamma';a)
    \circ \projsub_{\den{B[a/x]}}]}$. Since the type $\Gamma,x:A,\Gamma' \yields B$ is defined to be smaller
  than the context $\Gamma, x : A,\Gamma', y : B$, we can use the IH
  to conclude
  $\den{B[a/x]} \defeq \den{B}[\Subst(\Gamma;A;\Gamma';a)]$, which is
  what we need to use $\gen_{\den{B[a/x]}}$ as the second component of
  the pairing.

  The $\Subst(\Gamma;A;\Gamma', y :: B; a)$ case is similar: as the
  second component of the pairing we need an element
  \[\Elt{\den{\Gamma,\Gamma'[a/x]}.(N\den{B[a/x]})[\eta_{\den{\Gamma,\Gamma'[a/x]}}]}{(N\den{B})[\eta_{\Gamma, x : A, \Gamma' }][\Subst(\Gamma;A;\Gamma';a)
      \circ \projsub_{(N\den{B[a/x]})[\eta_{\den{\Gamma,\Gamma'[a/x]}}]}]}.\]
  Again, $\gen_{(N\den{B[a/x]})[\eta_{\den{\Gamma,\Gamma'[a/x]}}]}$
  does the job, as the size of $\Gamma,x:A,\Gamma' \yields B$ is defined to be smaller than
  $\Gamma, x : A, \Gamma', y : B$, so the IH gives us
  $\den{B[a/x]} \defeq \den{B}[\Subst(\Gamma;A;\Gamma';a)]$, and by
  naturality,
  \begin{align*}
    (N\den{B})[\eta_{\Gamma, x : A, \Gamma' } \circ \Subst(\Gamma;A;\Gamma';a)]
    &\defeq (N\den{B})[N\Subst(\Gamma;A;\Gamma';a) \circ \eta_{\Gamma, \Gamma'[a/x]}] \\
    &\defeq (N\den{B}[\Subst(\Gamma;A;\Gamma';a)])[\eta_{\Gamma, \Gamma'[a/x]}] \\
    &\defeq (N\den{B[a/x]})[\eta_{\Gamma, \Gamma'[a/x]}]
  \end{align*}

The cases for terms for the modality are essentially the same as in Lemma~\ref{lem:sem-roundtrip}. The only interesting new case is substituting into \rulen{var-roundtrip}.

Suppose we have done the roundtrip on $x$, the variable being
substituted for, and this variable lies at the end of the context.

Because all of $\den{\Gamma}$, $\den{\Gamma \yields A}$ and $\den{\Gamma \yields a}$ are defined, by
Lemma~\ref{lem:sem-roundtrip} we know that $\den{\Gamma \yields \zA}$ and $\den{\Gamma \yields \za}$
are defined, and
\begin{align*}
 \den{\Gamma \yields \zA} &\defeq \den{\Gamma \yields A}
[\varepsilon_{\den{\Gamma}} \circ \eta_{\den{\Gamma}}] \\
 \den{\Gamma \yields \za} &\defeq \den{\Gamma \yields a}
[\varepsilon_{\den{\Gamma}} \circ \eta_{\den{\Gamma}}]
\end{align*}

For the families, we then have:
\begin{align*}
  \den{\Gamma \yields \zA}
  &\defeq \den{\Gamma \yields \zA}[\projsub_{\den{A}}][\pair{\id_{\den{\Gamma}}}{\den{a}}] && \text{(Semantic projection cancels with substitution)} \\
  &\defeq \den{\Gamma \yields \zA}[\Wk(\Gamma; A; \cdot)][\Subst(\Gamma;A;\cdot;a)] && \text{(Definitions of $\Wk$ and $\Subst$)} \\
  &\defeq \den{\Gamma, x : A \yields \zA}[\Subst(\Gamma;A;\cdot;a)] && \text{(Lemma~\ref{lem:sem-wk})}
\end{align*}
using that $\den{\Gamma}$, $\den{\Gamma, x : A}$ and
$\den{\Gamma \yields \zA}$ are defined to apply
Lemma~\ref{lem:sem-wk}.

And for the elements:
\begin{align*}
&\den{\Gamma \yields \za} \\
&\defeq \den{\Gamma \yields a} [\varepsilon_{\den{\Gamma}} \circ \eta_{\den{\Gamma}}]\\
&\defeq \gen_{\den{A}}[\pair{\id_{\den{\Gamma}}}{\den{a}}][\varepsilon_{\den{\Gamma}} \circ \eta_{\den{\Gamma}}] && \text{(Substitution for variable)} \\
&\defeq \gen_{\den{A}}[\Subst(\Gamma;A;\cdot;a)][\varepsilon_{\den{\Gamma}} \circ \eta_{\den{\Gamma}}] && \text{(Definition of $\Subst$)} \\
&\defeq \gen_{\den{A}}[\varepsilon_{\den{\Gamma, x : A}} \circ \eta_{\den{\Gamma, x : A}}][\Subst(\Gamma;A;\cdot;a)] && \text{(Naturality)} \\
&\defeq \den{\Gamma, x : A \yields \zx}[\Subst(\Gamma;A;\cdot;a)] && \text{(Definition)}
\end{align*}

Again the general case follows by induction on the length of $\Gamma'$, as the definition of $\Subst(\Gamma;A;\Gamma';a)$ will commute past the iterated projections in the interpretation of $\zx$.
\end{proof}

\begin{lemma}[Semantic weakening]\label{lem:sem-wk}
  Suppose $\den{\Gamma, x : A, \Gamma'}$ and $\den{\Gamma, \Gamma'}$ are
  well-defined. Then there is a morphism $\Wk(\Gamma;A;\Gamma') :
  \den{\Gamma, x : A, \Gamma'} \to \den{\Gamma, \Gamma'}$ such that if
  $\den{\Gamma, \Gamma' \yields \judge}$ is defined then $\den{\Gamma, x : A, \Gamma \yields \judge}$
  is defined and \[\den{\Gamma, x : A, \Gamma \yields \judge} \defeq \den{\Gamma, \Gamma' \yields \judge}[\Wk(\Gamma;A;\Gamma)]\] where $\judge$ is a type or a term.

Similarly, if $\den{\Gamma, x :: A, \Gamma'}$ and $\den{\Gamma,
  \Gamma'}$ are well-defined then there is a morphism
$\ZWk(\Gamma;A;\Gamma') : \den{\Gamma, x :: A, \Gamma'} \to \den{\Gamma,
  \Gamma'}$ such that if $\den{\Gamma, \Gamma' \yields \judge}$ is
defined then $\den{\Gamma, x :: A, \Gamma \yields \judge}$ is defined
and
\[\den{\Gamma, x :: A, \Gamma \yields \judge} \defeq \den{\Gamma, \Gamma' \yields \judge}[\ZWk(\Gamma;A;\Gamma)].\]
\end{lemma}
\begin{proof}
  The metric is the size of $\Gamma,x:_?A,\Gamma'$ for the morphism and the size of $\Gamma,x:_?A,\Gamma' \yields J$ for the equation.
  The maps are defined in the
  expected way, with semantic projection in the base case and
  variables for the remaining telescope, and using the IH for the equation on $\Gamma,x:_?A,\Gamma' \yields B$ in the context extension cases.
  We just show the definition of the maps:
  \begin{align*}
    \Wk(\Gamma;A;\cdot) &:\defeq \projsub_{\den{\Gamma \yields A}} \\
    \Wk(\Gamma;A;\Gamma', y : B) &:\defeq \pair{\Wk(\Gamma;A;\Gamma') \circ \projsub_{\den{\Gamma, x : A, \Gamma' \yields B}}}{\gen_{\den{\Gamma, x : A, \Gamma' \yields B}}} \\
    \Wk(\Gamma;A;\Gamma', y :: B) &:\defeq \pair{\Wk(\Gamma;A;\Gamma') \circ \projsub_{(N\den{\Gamma, x : A, \Gamma' \yields B})[\eta_{\den{\Gamma, x : A, \Gamma'}}]}}{\gen_{(N\den{\Gamma, x : A, \Gamma' \yields B})[\eta_{\den{\Gamma, x : A, \Gamma'}}]}} \\
    \ZWk(\Gamma;A;\cdot) &:\defeq \projsub_{(N\den{\Gamma \yields A})[\eta_{\den{\Gamma}}]} \\
    \ZWk(\Gamma;A;\Gamma', y : B) &:\defeq \pair{\ZWk(\Gamma;A;\Gamma') \circ \projsub_{\den{\Gamma, x :: A, \Gamma' \yields B}}}{\gen_{\den{\Gamma, x :: A, \Gamma' \yields B}}} \\
    \ZWk(\Gamma;A;\Gamma', y :: B) &:\defeq \pair{\ZWk(\Gamma;A;\Gamma') \circ \projsub_{(N\den{\Gamma, x :: A, \Gamma' \yields B})[\eta_{\den{\Gamma, x :: A, \Gamma'}}]}}{\gen_{(N\den{\Gamma, x :: A, \Gamma' \yields B})[\eta_{\den{\Gamma, x :: A, \Gamma'}}]}}
  \end{align*}
\end{proof}

We now have the pieces required for the soundness theorem.
\begin{theorem}[Soundness]
  The partial interpretation function enjoys the following properties:
  \begin{itemize}
  \item If $\Gamma \ctx$ then $\den{\Gamma}$ is a well-defined object of $\C$,
  \item If $\den{\Gamma}$ is defined and $\Gamma \yields A \type$ then $\den{\Gamma \yields A} \in \Fam{\den{\Gamma}}$,
  \item If $\den{\Gamma}$ is defined and $\Gamma \yields a : A$ then $\den{\Gamma \yields A}$ is defined and $\den{\Gamma \yields a} \in \Elt{\den{\Gamma}}{\den{\Gamma \yields A}}$,
  \item If $\Gamma \defeq \Gamma' \ctx$ then $\den{\Gamma} \defeq \den{\Gamma'}$,
  \item If $\den{\Gamma}$ is defined and $\Gamma \yields A \defeq B \type$ then both $\den{A}$ and $\den{B}$ are defined and $\den{A} \defeq \den{B}$,
  \item If $\den{\Gamma}$ is defined and $\Gamma \yields a \defeq a' : A$ then all of $\den{A}, \den{a}$ and $\den{a'}$ are defined, both $\den{a}$ and $\den{a'}$ are in $\Elt{\den{\Gamma}}{\den{\Gamma \yields A}}$, and $\den{a} \defeq \den{a'}$.
  \end{itemize}
\end{theorem}
\begin{proof}
Routine, using the previous Lemmas as necessary. We show a few cases:
\begin{itemize}

\item \rulen{var-roundtrip}: Since the context is assumed to be
  defined, all of the subexpressions of the definition are defined.  For
  the type, we want $\den{\Gamma \yields \zA}$ defined and have that
  $\den{\Gamma \yields A}$ defined, so we apply
  Lemma~\ref{lem:sem-roundtrip}.  Modulo a number of semantic
  projections corresponding to the weakening of $A$ through the rest of
  the context, we want the definition to have type $\den{\Gamma \yields
    \zA}$, but it has type $\den{\Gamma \yields A}[\varepsilon_{\den
      \Gamma} \circ \eta_{\den \Gamma}]$, and these are equal.

\item \rulen{$\natural$-intro}: Suppose $\zc{\Gamma} \yields A \type$
  and $\zc{\Gamma} \yields a : A$ and that $\den{\Gamma}$ is defined.  We want
  to show $\den{a_A^\natural} \in \den{\natural A}$.  Because
  $\rulen{pre-unit}$ preserves size, we can apply the inductive
  hypothesis to the resulting derivations of $\Gamma \yields A \type$
  and $\Gamma \yields a: A$ to get that $\den{\Gamma \yields A}$ and
  $\den{\Gamma \yields a} \in \Elt{\den{\Gamma}}{\den{\Gamma \yields A}}$
  are defined.  Thus, $\den{\natural A} \defeq (N\den A)[\eta_{\den
      \Gamma}]$ and $\den{a_A^\natural} \defeq (N_{\den{\Gamma \yields
      A}}\den{a})[\eta_{\den \Gamma}]$ are defined, and the term has the correct type.

\item \rulen{$\natural$-elim}: Suppose $\zc{\Gamma} \yields A \type$ and
  $\Gamma \yields a : \natural A$ and that $\den{\Gamma}$ is defined.
  We want to show $\den{a^A_\natural} \in \den{A}$.  Because
  $\rulen{pre-unit}$ preserves size, we can apply the inductive
  hypothesis to the resulting derivation of $\Gamma \yields A \type$ and
  to $\Gamma \yields a: \natural A$ to get that $\den{\Gamma \yields A}$
  and $\den{\Gamma \yields a} \in \Elt{\den{\Gamma}}{N \den{\Gamma \yields
    A} [\eta_{\den \Gamma}]}$ are defined.  Thus, $\den{a^A_\natural}$
  is defined, since all of its pieces are, and $\den{a}$ has the
  necessary type.  Its type is $\den{A}[\projsub_{\den
      A}][(\id_{\den{\Gamma}},\den{a})]$, which is defined because all of
  its pieces are, and reduces to $\den{A}$ as desired.

\item \rulen{$\natural$-beta}: Assume the typing premises $\zc{\Gamma}
  \yields A \type$ and $\zc{\Gamma} \yields a : A$.
  Then $\den{A}$
  and $\den{a} \in \Elt{\den{\Gamma}}{\den{\Gamma \yields A}}$
  are defined because we can apply \rulen{pre-unit} to both before using the
  inductive hypothesis.  Second, $\den{(a^\natural)_\natural}$ is defined
  and has type $\den{A}$ by composing the previous two cases, which
  works because we have the same inductive hypotheses available on the
  typing premises.
  Finally, these are equal because:
\begin{align*}
\den{\Gamma \yields a^\natural{}_\natural}
&\defeq \zc{\gen}_{\den{A}}[\pair{\id_{\den{\Gamma}}}{\den{\Gamma \yields a^\natural}}] \\
&\defeq \zc{\gen}_{\den{A}}[\pair{\id_{\den{\Gamma}}}{(N\den{\Gamma \yields a})[\eta_{\den{\Gamma}}]}] \\
&\defeq \den{\Gamma \yields a}[\varepsilon_{\den{\Gamma}} \circ \eta_{\den{\Gamma}}] && \text{(Equation~\ref{eq:zq-beta})} \\
&\defeq \den{\Gamma \yields \za} && \text{(Lemma~\ref{lem:sem-roundtrip})} \\
&\defeq \den{\Gamma \yields a} && \text{($a$ is already zeroed syntactically)}
\end{align*}

\item \rulen{$\natural$-eta}: Suppose the typing premises $\zc{\Gamma}
  \yields A \type$ and $\Gamma \yields n : \natural A$.  The inductive
  hypothesis on the latter gives $\den{\natural A}$ and $\den{n} \in
  \Elt{\den \Gamma}{\den{\natural A}}$ defined.  By Lemma~\ref{lem:sem-roundtrip},
  $\den{\Gamma \yields \zn} \in \Elt{\den \Gamma}{\den{\natural A}[\varepsilon_{\den{\Gamma}} \circ
    \eta_{\den{\Gamma}}]}$ is defined, and $\den{\natural A}[\varepsilon_{\den{\Gamma}} \circ
    \eta_{\den{\Gamma}}] \defeq \den{\natural A}$.
  This implies $\den{(\zn_\natural^A)^\natural_A}$ is defined and is in
  $\Elt{\den \Gamma}{\den{\natural A}}$.
\begin{align*}
\den{\Gamma \yields \zn_\natural{}^\natural}
&\defeq (N\den{\Gamma \yields \zn_\natural})[\eta_{\den{\Gamma}}] \\
&\defeq (N(\zc{\gen}_{\den{A}}[\pair{\id_{\den{\Gamma}}}{\den{\Gamma \yields \zn}}]))[\eta_{\den{\Gamma}}] \\
&\defeq (N\zc{\gen}_{\den{A}})[\eta_{\den{\Gamma}.(N\den{A})[\eta_{\den{\Gamma}}]}][\pair{\id_{\den{\Gamma}}}{\den{\Gamma \yields \zn}}] && \text{(Naturality of $\eta$)} \\
&\defeq \gen_{(N\den{A})[\eta_{\den{\Gamma}}]}[\pair{\id_{\den{\Gamma}}}{\den{\Gamma \yields \zn}}] && \text{(Equation~\ref{eq:zq-eta})} \\
&\defeq \den{\Gamma \yields \zn} \\
&\defeq \den{\Gamma \yields n}[\varepsilon_{\den{\Gamma}} \circ \eta_{\den{\Gamma}}] && \text{(Lemma~\ref{lem:sem-roundtrip})} \\
&\defeq \den{\Gamma \yields n} && \text{(Lemma~\ref{lem:n-rt-triv})}
\end{align*}
\end{itemize}
\end{proof}

\begin{corollary}
  MLTT extended as in Figures~\ref{fig:structural} and~\ref{fig:natural-rules} is logically consistent.
\end{corollary}
\begin{proof}
  For any CwF there is a CwB with $N$ the identity functor, so the
  extension interprets in any existing model.
\end{proof}

\subsection{Syntactic Category}

Conversely, we can build a syntactic category.

\begin{theorem}[Completeness]
  The term model of the type theory forms a CwB.
\end{theorem}
\begin{proof}
  The CwF is constructed in the usual way, with the underlying
  category given by contexts and total substitutions between them, the
  families for a context given by types in that context up to
  equality, and elements of each family by terms up to equality in
  that type.

  Total substitutions are defined as normal for the empty context and
  ordinary context extension. Substitution into $\Gamma,\zx::A$ is
  defined to match \rulen{dull-subst}:
  \begin{mathpar}
    \inferrule*[left=subst-ext-zero]
    {\Delta \yields \theta : \Gamma \and
      \zc{\Delta} \yields a : A[\zc{\theta}]
    }
    {\Delta \yields (\theta, a/\zx) : \Gamma, \zx :: A}
  \end{mathpar}

  The identity substitution is given as usual for ordinary context
  extension, and by
  $\id_{\Gamma, \zx :: A} :\defeq \id_\Gamma, \zx/\zx$ for modal
  context extension. Composition of total substitutions
  $\theta : \Delta \yields \Gamma$ and
  $\theta' : \Gamma \yields \Gamma'$ is defined by induction on
  $\theta'$. The new case, when $\Gamma'$ ends with a modal context
  extension, is defined by
  $(\theta', a'/\zx)[\theta] \defeq (\theta'[\theta],
  a'[\zc{\theta}]/\zx)$. Checking associativity and unit laws of
  composition is then straight forward.

  Turning to the CwB structure, the functor $N$ is defined on contexts
  by $N(\Gamma) :\defeq \zc{\Gamma}$, on types by
  $N(A) :\defeq \natural \zA$ and terms by
  $N(a) :\defeq \za^\natural$. On total substitutions, we define
  $N(\theta) :\defeq \zc{\theta}$, where
  \begin{align*}
    \zc{\cdot} &:\defeq \cdot \\
    \zc{\theta, a/x} &:\defeq \zc{\theta}, \za/\zx \\
    \zc{\theta, a/\zx} &:\defeq \zc{\theta}, a/\zx
  \end{align*}
  This admissible rule was used in \rulen{subst-ext-zero} above.

  To show $N$ is a weak CwF morphism, $N$ preserves the terminal
  object (empty context) by definition, and commutes with reindexing
  on types and terms by the definition of substitution as an
  admissible operation. We just need that
  $N(\Gamma.A) \to N\Gamma.NA$ is an isomorphism, and unfolding the
  definitions, this is the substitution
  \[\zc{\Gamma, x : A} \yields (\id_{\zc{\Gamma}}, \zx^\natural) :
    \zc{\Gamma}, \zy :: \natural \zA \]
  This has an inverse
  \[\zc{\Gamma}, \zy :: \natural \zA \yields (\id_{\zc{\Gamma}}, \zy_\natural) :
    \zc{\Gamma, x : A} \] by the $\beta$- and $\eta$-rules for
  $\natural$. The $\beta$-rule will give us $\zx$ in the last
  component, but this matches the definition of the identity
  substitution for a marked variable.

  The unit and counit substitutions
  \begin{mathpar}
    \inferrule*[fraction={-{\,-\,}-}]{~}{\Gamma \yields \mathsf{unit}_\Gamma : \zc{\Gamma}} \and
    \inferrule*[fraction={-{\,-\,}-}]{~}{\zc{\Gamma} \yields \mathsf{counit}_\Gamma : \Gamma}
  \end{mathpar}
  are given by inductively defining
  \begin{align*}
    \mathsf{unit}_{\cdot} &:\defeq \cdot \\
    \mathsf{unit}_{\Gamma, x : A} &:\defeq \mathsf{unit}_{\Gamma},\zx / \zx \\
    \mathsf{unit}_{\Gamma, \zx :: A} &:\defeq \mathsf{unit}_{\Gamma},\zx / \zx \\
    \mathsf{counit}_{\cdot} &:\defeq \cdot \\
    \mathsf{counit}_{\Gamma, x : A} &:\defeq \mathsf{counit}_{\Gamma},\zx / x \\
    \mathsf{counit}_{\Gamma, \zx :: A} &:\defeq \mathsf{counit}_{\Gamma},\zx / \zx
  \end{align*}
  Naturality is easy to verify by induction. The round-trip on
  $\zc{\Gamma}$ is exactly the identity substitution, as
  $\id_{\zc{\Gamma}}$ by definition consists of tuples of zeroed
  variables.
\end{proof}

\section{Type-Theoretic Description of Parametrised Pointed Types}\label{sec:toy-model}

We present a concrete implementation in parametrised pointed
types, which has a similar structure to parametrised spectra, but is
simpler to describe.  The $\infty$-topos of \emph{parametrised pointed
  spaces} is defined similarly to $P\Spec$, but with the
$\infty$-category of pointed spaces $\mathcal{S}_\star$ taking the place
of $\Spec$, so an object consists of a space $B$ and a family $E :
\mathrm{Fun}(B,\mathcal{S}_\star)$ of pointed spaces.  In intensional
type theory, we can define a weak version of this model, where many
equations that should be strict equalities (such as the $\beta$- and
$\eta$-rules for the type formers) hold only up to paths.
Indeed, in this internal presentation,
function types in the theory are interpreted as functions that
preserve the point only up to a path whereas we conjecture that there is
an external model in parametrised pointed spaces where functions are pointed up to equality.

This model has been considered independently in some related work.
Some aspects, including the universe, were described in a talk by
Buchholtz~\cite{buchholtz:toy-models}, which discussed simple,
internally definable models of cohesion.  The paper of Kraus and
Sattler~\cite{kraus-sattler:spacediagrams} discuss how, for certain
Reedy diagrams, the data of a coherent diagram has a finite
description internally, without needing to specify infinitely many
coherences.  Parameterised pointed spaces can be seen as the
category of homotopy coherent diagrams over the walking
section/retraction.  Intuitively, an $\infty$-presheaf $\Gamma$ on
$\{s : 1 \to 0, r : 0 \to 1 \mid r \circ s = \id_0\}$ consists of a
space $\Gamma(0)$ and a space $\Gamma(1)$ with a projection
$\Gamma(1) \to \Gamma(0)$ and a section of it
$\Gamma(0) \to \Gamma(1)$.  Taking the fibre of the projection, we can
think of $\Gamma(1)$ as a dependent type on $\Gamma(0)$, with each
type in the fibre equipped with a distinguished point determined by the
section $s$. Kraus and Sattler's construction
applies to this index category, and their construction yields a
description equivalent to the one presented here.

Coquand, Ruch, and Sattler~\cite{crs:sheaf-models} also consider
constructive sheaf models of univalent type theory on the walking
section-retraction.

\subsection{Contexts and Substitutions}

In this (weak) model, each `context' `$\Gamma \ctx$' is interpreted as
an element of a record type
\begin{definition}[Contexts in Pointed Spaces]
\begin{align*}
  \mathsf{Ctx} :\defeq \{
  \dnst &: \univ, \\
  \upst &: \dnst \to \univ, \\
  \ptst &: \prd{g : \dnst} \upst (g)\}
\end{align*}
\end{definition}
\noindent consisting of a type $\dnst$, a family of types over it
$\upst$, and a section $\ptst$.  Equivalently, we could package $\upst$
and $\ptst $ together into a family of pointed types.  This corresponds
to the section/retraction pair
\begin{mathpar}
\begin{tikzcd}
\sm{g : \dnst\Gamma} \upst\Gamma(g) \ar[d, "\proj_1",bend left] \\
\dnst\Gamma \ar[u, "{(g, \ptst g)}" bend left] &
\end{tikzcd}
\end{mathpar}
With this in mind, we refer to $\dnst \Gamma$ as the \emph{base} or
\emph{index space}, $\upst \Gamma$ as the \emph{fibre}, and $\ptst$ as
the \emph{section} or the \emph{point} (though this is a bit imprecise,
since really it is a family of points, one for each object of the base).

A `substitution' $\Gamma \yields \theta : \Delta$ between $\Gamma, \Delta
: \mathsf{Ctx}$ is then a map of the base and a map of the fibres that
commute with the section, represented by the record type
\begin{definition}[Substitutions in Pointed Spaces]
\begin{align*}
  \{\dnst &: \dnst \Gamma \to \dnst \Delta, \\
    \upst &: \prd{g : \dnst \Gamma} \upst \Gamma(g) \to \upst \Delta(\dnst (g)), \\
    \ptst &: \prd{g : \dnst \Gamma} \upst(g, \ptst \Gamma (g)) =_{\upst \Delta( \dnst(g))} \ptst \Delta(\dnst (g)) \}
\end{align*}
\end{definition}

\begin{definition}[Empty Context in Pointed Spaces]
The empty context `$\emptyset \ctx$' is given by the unit type over the unit type.
\begin{align*}
\dnst(\emptyset) &:\defeq 1 \\
\upst(\emptyset)(\star) &:\defeq 1 \\
\ptst(\emptyset)(\star) &:\defeq \star
\end{align*}
\end{definition}
\noindent Note that we use ``copattern'' notation~\cite{copatterns} to
describe elements of records/functions, i.e.\ this desugars to the
record $\emptyset :\defeq \{\dnst = 1, \upst = \lambda g. 1, \ptst =
\lambda g. \star\}$

For a context $\Gamma : \mathsf{Ctx}$, the context $\natural\Gamma :
\mathsf{Ctx}$ is defined by keeping the base the same but replacing the
fibre with the point:
\begin{definition}[Natural Context in Pointed Spaces]
\begin{align*}
  \dnst(\natural\Gamma) &:\defeq \dnst\Gamma
  \\ \upst(\natural\Gamma)(g) &:\defeq 1
  \\ \ptst(\natural\Gamma)(g) &:\defeq \star
\end{align*}
\end{definition}

The unit substitution $\Gamma \vdash \eta : \natural \Gamma$ is given by
\begin{align*}
  \dnst \eta (g) &:\defeq g \\
  \upst \eta (g,g') &:\defeq \star \\
  \ptst \eta (g) &:\defeq \refl{}
\end{align*}
i.e.\ it is the identity in the base, and sends everything in $\upst
\Gamma$ to the point, which in particular sends the sections of $\Gamma$
to the point.  Because the fibres are pointed, there is
also a counit $\natural \Gamma \vdash \varepsilon : \Gamma$
\begin{align*}
   \dnst \varepsilon (g) &:\defeq g \\
   \upst \varepsilon (g,\star) &:\defeq \ptst \Gamma(g) \\
   \ptst \varepsilon (g,\star) &:\defeq \refl{}
\end{align*}
i.e.\ it is again the identity on the base, and sends the point in the
fibre to the chosen point in each fibre.

For a context $\Gamma$ (which determines a closed type), $\natural
\Gamma$ is contractible in this model when $\dnst \Gamma$ is a
contractible type. A $\natural$-connected context has the form
$\{\dnst \Gamma \equiv 1, \upst \Gamma : 1 \to \univ, \ptst \Gamma :
\upst \Gamma(\star)\}$ and thus is exactly a single pointed type.
Moreover, the substitutions between them are exactly pointed maps.

Supposing that $\Gamma$ consists of a single type $A$, we are used to thinking of
substitutions $\emptyset \vdash \theta : A$ as corresponding to `elements
of $A$'. In our setting, however, if $A$ is a $\natural$-connected type then the type of such
substitutions is contractible --- there is a unique pointed map from $1$
to $A$ that sends the point in the fibre to the point of $A$. Thus,
substitutions from $\emptyset$ do not give access to the elements of the \emph{fibre} of $A$.
To access the elements of the fibre, we can instead consider
substitutions from $\mathbb{B} \vdash \theta : A$, where $\mathbb{B}$ is
the booleans over the unit type
\begin{definition}
  \begin{align*}
    \dnst \mathbb{B} &:\defeq 1\\
    \upst \mathbb{B}(\star) &:\defeq 2\\
    \ptst \mathbb{B}(\star) &:\defeq \mathsf{true}
  \end{align*}
\end{definition}
Substitutions $\mathbb{B} \vdash \theta : A$ correspond to pointed maps
from the booleans into $\upst A$.  Such a pointed map sends
$\mathsf{true}$ to the point, but has one remaining degree of freedom
--- where they send $\mathsf{false}$.  Thus, the substitutions from
$\mathbb{B}$ into a $\natural$-connected type $A$ are equivalent to
$\upst A(\star)$.  We use an analogue of $\mathbb{B}$ in
Section~\ref{sec:redu} below.

The fact that substitutions $\emptyset \vdash \theta : A$ `miss' data
from $A$ does not conflict with the equivalence of types
$A \equiv (1 \to A)$: we will see that the definition of the internal
hom does capture the fibre of $A$ in the fibre of the function type.

\subsection{Types and Terms}

A dependent type in this model `$\Gamma \yields A \type$' is specified
by
an element of the record type
\begin{definition}[Dependent Type in the Pointed Spaces Model]
\begin{align*}
\{
\dnst &: \dnst \Gamma \to \univ, \\
\upst &: \prd{g : \dnst\Gamma} \dnst (g) \to \upst\Gamma(g) \to \univ, \\
\ptst &: \prd{g : \dnst\Gamma} \prd{a : \dnst (g)} \upst (g, a, \ptst\Gamma(g))
\}
\end{align*}
\end{definition}
\noindent The base of a type $A$ depends only on the base of the
context,
while the fibre depends on both the base and the fibre of the context,
and the base of the type.
Not every fibre of the $\upst A$ family has a specified point, only
the fibres that additionally lie over the basepoint of $\upst \Gamma$.

The data of a type in context can arranged into the following diagram.
\begin{mathpar}
\begin{tikzcd}
\sm{g : \dnst\Gamma} \sm{a : \dnst A(g)} \sm{e : \upst \Gamma(g)} \upst A(g,a,e) \ar[d] \ar[r] & \sm{g : \dnst\Gamma} \upst\Gamma(g) \ar[d] \\
\sm{g : \dnst\Gamma} \dnst A(g) \ar[u, bend left] \ar[r] & \dnst\Gamma \ar[u, bend left]
\end{tikzcd}
\end{mathpar}
where the upwards maps are sections definable from $\ptst\Gamma$ and $\ptst A$.

This section/retraction on the left is exactly the interpretation of context extension `$\Gamma.A \ctx$', and the horizontal maps describe the projection substitution $\Gamma.A \to \Gamma$
\begin{definition}[Context Extension in Pointed Spaces]
The context extension `$\Gamma.A \ctx$' is defined by
\begin{align*}
\dnst(\Gamma.A) &:\defeq \sm{g : \dnst\Gamma} \dnst A(g) \\
\upst(\Gamma.A)(g, a) &:\defeq \sm{e : \upst \Gamma(g)} \upst A(g,a,e) \\
\ptst(\Gamma.A)(g, a) &:\defeq (\ptst \Gamma(g), \ptst A(g, a))
\end{align*}
\end{definition}
A `term in context' must be the data of a section of the projection substitution `$\Gamma.A \to
\Gamma$' shown in the diagram above --- this is a section in the
\emph{horizontal} direction.  Building the section equation in via dependency gives
\begin{definition}[Terms in Pointed Spaces]
A term `$\Gamma \yields t : A$' is an element of the record type
\begin{align*}
\{\homdnst &: \prd{g : \dnst \Gamma} \dnst A(g) \\
\homupst &: \prd{g : \dnst \Gamma} \prd{e : \upst \Gamma(g)} \upst A(g, \homdnst (g), e)\\
\homptst &: \prd{g : \dnst \Gamma} \homupst (g, \ptst \Gamma (g)) =_{\upst A(g,\homdnst (g), \ptst \Gamma(g))} \ptst A(g, \homdnst (g))\}
\end{align*}
\end{definition}

For $\natural \Gamma \vdash A \type$, we can define the natural type
$\Gamma \vdash \natural A \type$ by replacing the fibres with the unit
type:
\begin{definition}[Natural Type in Pointed Spaces]
\begin{align*}
\dnst (\natural A) (g) &:\defeq \dnst A (g) \\
\upst (\natural A) (g,a,e) &:\defeq 1 \\
\upst (\natural A) (g,a) &:\defeq \star
\end{align*}
\end{definition}
\noindent Note that this definition only uses $g : \dnst \Gamma$ and not
the fibre or the section of $\Gamma$, so the definition is independent
of whether we ask for an $A$ that depends on $\Gamma$ or on $\natural
\Gamma$.

\subsection{Additional Type Constructors}

This internal model supports both dependent sums and dependent products
satisfying the expected equations propositionally. Our goal here is to
demonstrate how these type formers depend on the pieces of the context,
so we instead describe non-dependent $\times$-types and $\to$-types, but
the analogous definitions replacing $\times$ with $\Sigma$ and $\to$
with $\Pi$ give the dependent versions.


Product types are easy: we form the product both downstairs and upstairs. The pairing and projection terms are defined using the pairing and projection of the underlying $\times$-types.
\begin{definition}[Product Types in Pointed Spaces]
For two types `$\Gamma \yields A \type$' and `$\Gamma \yields B \type$' define `$\Gamma \yields A \times B \type$' by
\begin{align*}
\dnst(A \times B)(g) &:\defeq \dnst A(g)\times \dnst B(g) \\
\upst(A \times B)(g, (a, b), e) &:\defeq \upst A(g, a, e) \times \upst B(g, b, e)  \\
\ptst(A \times B)(g,(a, b)) &:\defeq (\ptst A(g, a), \ptst B(g, b))
\end{align*}
\end{definition}

%

Function types are more interesting. The base of $A \to B$ is not the
type of functions $\dnst A \to \dnst B$, but rather the entire type of
substitutions $\Gamma.A$ to $\Gamma.B$ over $\Gamma$. For a fixed $g :
\dnst \Gamma$, the data of such a substitution unwinds to triples
\begin{align*}
  \Homst(A, B)(g : \dnst \Gamma) :=
 \{ & \homdnst : \dnst A(g) \to \dnst B(g) \\
    & \homupst : \prd{a : \dnst A(g)} \upst A(g, a, \ptst \Gamma(g)) \to \upst B(g, \homdnst (a), \ptst \Gamma(g)) \\
    & \homptst : \prd{a : \dnst A(g)} \homupst (a, \ptst A(g, a)) =_{} \ptst B(g, \homdnst (a)) \}
\end{align*}

\begin{definition}[Function Types in Pointed Spaces]
For two types `$\Gamma \yields A \type$' and `$\Gamma \yields B \type$' define `$\Gamma \yields A \to B \type$' by
\begin{align*}
\dnst(A \to B)(g) &:\defeq \Homst(A, B)(g) \\
\upst(A \to B)(g,(\homdnst f, \homupst f, \homptst f),e) &:\defeq \prd{a : \dnst A(g)} \upst A(g, a, e) \to \upst B(g, \homdnst f(a), e) \\
\ptst(A \to B)(g,(\homdnst f, \homupst f, \homptst f)) &:\defeq \homupst f
\end{align*}
\end{definition}
This definition makes it clear that the $\natural$ modality does
\emph{not} preserve $\Pi$-types: the base of $\natural(A \to B)$ is
$\Homst(A,B)$, while the base of $(\natural A \to \natural B)$ is
essentially $\dnst A \to \dnst B$ (because the fibres of $\natural \zB$
are $1$, the rest of the data of $\Homst(\natural A,\natural B)$ is
determined).

We can also define identity types and the universe. Similar to $\Sigma$-types, $\Idsym$-types are given component-wise. The type `$a=a'$' has as its base paths in the base of `$A$', with the family over a path $p : \dnst(a)(g) = \dnst(a')(g)$ given by dependent paths in $\upst A(g,-,e)$ that lie over it.
\begin{definition}[Identity Types in Pointed Spaces]
For `$\Gamma \yields a : A$' and `$\Gamma \yields a' : A$', define `$\Gamma \yields (a = a') \type$' by
\begin{align*}
\dnst(a = a')(g) &:\defeq (\dnst(a)(g) = \dnst(a')(g)) \\
\upst(a = a')(g, p, e) &:\defeq (\upst(a)(g, e) =^{\ap_{\upst A(g,-,e)}(p)} \upst(a')(g, e)) \\
\ptst(a = a')(g, p) &:\defeq \ptst(a)(g) \pathcat \mathsf{apd}_{\ptst A(g, \dnst B(-)(g))}(p) \pathcat \ptst(a')(g)^{-1}
\end{align*}
The types of the paths used to define the basepoint of each family above are
\begin{align*}
\ptst(a)(g) &:  \upst(a)(g, \ptst \Gamma (g)) = \ptst A(g, \dnst(a)(g)) \\
\mathsf{apd}_{\ptst A(g, \dnst B(-)(g))}(p) &: \ptst A(g, \dnst(a)(g)) =^{\ap_{\upst A(g,-,e)}(p)} \ptst A(g, \dnst(a')(g)) \\
\ptst(a')(g)^{-1} &: \ptst A(g, \dnst(a')(g)) = \upst(a')(g, \ptst \Gamma (g))
\end{align*}
\end{definition}

Finally, the universe has as its base the entire type of pointed families, and
as its upstairs, the type of unpointed families over the same base.
\begin{definition}[The Universe in Pointed Spaces]
The universe `$\emptyset \yields \univ \type$' is given by
\begin{align*}
\dnst(\univ)(\star) &:\defeq \sm{B : \univ}\sm{E : B \to \univ} \prd{b : B} E(b) \\
\upst(\univ)(\star, (B, E, p), \star) &:\defeq B \to \univ \\
\ptst(\univ)(\star, (B, E, p)) &:\defeq E
\end{align*}
\end{definition}



\section{Future Work}

This paper lays the groundwork for synthetic stable homotopy theory,
but much remains to be done. The most obvious next step is to investigate
synthetic cohomology using reduced types:

\begin{definition}
  For a dull pointed space $\zX$ and reduced type $\zE$, the
  \emph{$n$th cohomology of $\zX$ with coefficients in $\zE$} is defined
  by the formula
  \begin{align*}
    \zE^n(\zX) :\defeq \pi_n(\natural(\Sigma^\infty \zX \to \zE))
  \end{align*}
\end{definition}

Using Axiom S, we expect these groups to form a `cohomology theory',
in the sense of the Eilenberg-Steenrod axioms (see~\cite[Section
  3]{cavallo:cohomology}) for any reduced
type $\zE$. To recover ordinary cohomology, we need to define a
reduced type `$H\ZZ$' that satisfies $\pi^s_0(H\ZZ) = \ZZ$ and
$\pi^s_n(H\ZZ) = 0$ for $n \neq 0$. Axiom N implies that
$\pi^s_0 \sphere = \ZZ$, so we expect to form such an $H\ZZ$ by
\emph{truncating} $\sphere$ to remove the stable homotopy groups above
dimension 0. This truncation operation will be given by a higher
inductive type, nullification at $\Sigma\sphere$.

To discuss synthetic \emph{homology}, we will need a version of the
`smash product of spectra', so that we can make the definition
`$\zE_n(\zX) :\defeq \pi^s_n(\Sigma^\infty \zX \otimes \zE)$'. This
product is analogous to the tensor product of modules/chain complexes
in algebra, and spectra come equipped with an internal hom for this
product. Adding type formers corresponding to these operations will
require a more radical modification to the way contexts are structured
and remains work in progress.

We have also begun work on a modification of Agda to include our
$\natural$-modality, which would allow us to formalise the results of
this paper. Agda already has infrastructure to support a class of
modalities, including the $\flat$ modality of spatial type
theory~\cite{agda-flat} and run-time irrelevance~\cite{atkey:qtt}. The
common factor in all these modalities is that they restrict access to
the variables in the context in various ways, but when a variable is
available, it is used via the ordinary variable rule. Including our
\rulen{var-zero} rule involves touching many more aspects of the
type-checker, and in particular it is not yet clear how marked
variables should interact with the module system. This is also left
for future work.

Finally, it remains to work out the details of the models, constructing
instances of a CwB for model categories presenting parametrised pointed
spaces and for parametrised spectra.  For this, we anticipate building
on \cite{mike:all} and the strictification
techniques~\cite{lumsdaine-warren:local-universes}
that have previously been applied to modalities~\cite{drats}.

\printbibliography

\end{document}